\documentclass[11pt, letterpaper]{amsart}

\usepackage[
  letterpaper,
  width=6.5in,
  height=8.75in,
  left=1in,
  right=1in,
  top=1.125in,
  bottom=1.125in,
  includehead
]{geometry}

\usepackage{amsmath}
\usepackage{mathrsfs} 
\usepackage{amsthm}
\usepackage{amsfonts}
\usepackage{amssymb}
\usepackage{latexsym}
\usepackage{tikz} 
\usetikzlibrary{cd}
\tikzcdset{scale cd/.style={every label/.append style={scale=#1},
    cells={nodes={scale=#1}}}}
\usepackage{esint}
\usepackage{mathtools}
\usepackage{stmaryrd}

\usepackage[
backend=biber,
style=alphabetic,
sorting=anyt, doi=false, url=false, isbn=false, giveninits=true, maxalphanames=10, minalphanames=5,maxbibnames=10
]{biblatex}

\usepackage[OT2,T1]{fontenc}
\DeclareSymbolFont{cyrletters}{OT2}{wncyr}{m}{n}
\DeclareMathSymbol{\Sha}{\mathalpha}{cyrletters}{"58}

\usepackage{hyperref}
\usepackage{caption}
\usepackage{float}
\usepackage{setspace}

\usepackage{comment}
\usepackage{quiver}

\usepackage{xcolor}

\DeclareMathAlphabet{\mathcal}{OMS}{cmsy}{m}{n}

\definecolor{mediumtealblue}{rgb}{0.0, 0.33, 0.71}
\definecolor{red(munsell)}{rgb}{0.96, 0.12, 0.24}
\definecolor{forestgreen}{rgb}{0.13, 0.55, 0.13}
\definecolor{ginger}{rgb}{0.69, 0.4, 0.0}

\newcommand\myshade{85}
\colorlet{mylinkcolor}{red(munsell)}
\colorlet{mycitecolor}{mediumtealblue}
\colorlet{myurlcolor}{forestgreen}

\usepackage{mathptmx}

\usepackage[utf8]{inputenc}

\usepackage[inline,shortlabels]{enumitem}

\usepackage{amssymb,mathrsfs,stmaryrd}
\usepackage{mathtools}
\usepackage{thmtools,thm-restate}
\usepackage{colonequals}
\usepackage{tikz}
\usetikzlibrary{cd}

\usepackage[capitalize,noabbrev]{cleveref}

\hypersetup{
	linkcolor  = mylinkcolor!\myshade!black,
	citecolor  = mycitecolor!\myshade!black,
	urlcolor   = myurlcolor!\myshade!black,
	colorlinks = true
}

\mathchardef\mhyphen="2D

\theoremstyle{plain}
\newtheorem{theorem}{Theorem}[section]

\newtheorem{proposition}[theorem]{Proposition}

\newtheorem{lemma}[theorem]{Lemma}

\newtheorem{corollary}[theorem]{Corollary}

\theoremstyle{definition}
\newtheorem{definition}[theorem]{Definition}
\newtheorem{example}[theorem]{Example}
\newtheorem{construction}[theorem]{Construction}
\newtheorem{set-up}[theorem]{Set-up}

\newtheorem{assumption}[theorem]{Assumption}

\newtheorem{remark}[theorem]{Remark}

\newcommand{\IA}{\mathbb{A}}
\newcommand{\IB}{\mathbb{B}}
\newcommand{\IC}{\mathbb{C}}
\newcommand{\ID}{\mathbb{D}}

\newcommand{\IF}{\mathbb{F}}

\newcommand{\IL}{\mathbb{L}}

\newcommand{\IN}{\mathbb{N}}

\newcommand{\IP}{\mathbb{P}}
\newcommand{\IQ}{\mathbb{Q}}
\newcommand{\IR}{\mathbb{R}}

\newcommand{\IV}{\mathbb{V}}

\newcommand{\IZ}{\mathbb{Z}}

\newcommand{\sC}{\mathcal{C}}
\newcommand{\sK}{\mathcal{K}}

\newcommand{\End}{\mathrm{End}}

\newcommand{\Hom}{\mathrm{Hom}}
\newcommand{\Aut}{\mathrm{Aut}}

\renewcommand{\deg}{\mathrm{deg}}
\newcommand{\supp}{\mathrm{supp \,}}

\newcommand{\Spec}{\mathrm{Spec}}

\newcommand{\rank}{\rm rank} 

\newcommand\iso{\,{\cong}\,} 
\newcommand\tensor{{\otimes}}

\newcommand{\<}{\langle}
\renewcommand{\>}{\rangle}

\newcommand{\into}{\hookrightarrow}

\def\d/{/\mspace{-6.0mu}/}

\def\wt{\widetilde}

\def\what{\widehat}

\newcommand{\sU}{\mathcal{U}}

\newcommand{\sV}{\mathcal{V}}

\newcommand{\sF}{\mathcal{F}}
\newcommand{\sH}{\mathcal{H}}

\newcommand{\w}{\omega}

\newcommand{\Ohm}{\Omega}

\newcommand{\fm}{\mathfrak{m}}

\newcommand{\Proj}{\mathrm{Proj\,}}
\newcommand{\Pic}{\mathrm{Pic}\,}

\newcommand{\sE}{\mathcal{E}}
\newcommand{\sL}{\mathcal{L}}

\newcommand{\sO}{\mathcal{O}}

\newcommand{\Cl}{\mathrm{Cl}}

\newcommand{\Gal}{\mathrm{Gal}}

\newcommand{\NS}{\mathrm{NS}}

\newcommand{\Sch}{\mathsf{Sch}}

\newcommand{\et}{\mathrm{{\acute{e}}t}}

\newcommand{\Sh}{\mathrm{Sh}}
\newcommand{\shS}{\mathscr{S}}

\newcommand{\shZ}{\mathscr{Z}}

\newcommand{\sfD}{\mathsf{D}}
\newcommand{\sfM}{\mathsf{M}}

\newcommand{\sA}{\mathcal{A}}
\newcommand{\sX}{\mathcal{X}}
\newcommand{\sY}{\mathcal{Y}}
\newcommand{\sT}{\mathcal{T}}
\newcommand{\bL}{\mathbf{L}}
\newcommand{\LEnd}{\mathrm{LEnd}}
\renewcommand{\H}{\mathrm{H}}
\newcommand{\PH}{\mathrm{PH}}

\newcommand{\crys}{\mathrm{crys}}
\newcommand{\PNS}{\mathrm{PNS}}
\newcommand{\sfK}{\mathsf{K}}
\newcommand{\sfH}{\mathsf{H}}

\newcommand{\sfW}{\mathsf{W}}
\newcommand{\sfB}{\mathsf{B}}

\newcommand{\Def}{\mathrm{Def}}

\newcommand{\sS}{\mathcal{S}}

\newcommand{\bH}{\mathbf{H}}
\newcommand{\bP}{\mathbf{P}}
\newcommand{\dR}{\mathrm{dR}}
\newcommand{\SO}{\mathrm{SO}}
\newcommand{\tf}{\mathrm{tf}}

\newcommand{\sto}{\stackrel{\sim}{\to}}

\newcommand{\FCrys}{\mathrm{FCrys}}

\newcommand{\CSpin}{\mathrm{CSpin}}

\newcommand{\fh}{\mathfrak{h}}

\newcommand{\Loc}{\mathsf{Loc}}

\newcommand{\alg}{{\mathrm{alg}}}
\renewcommand{\sc}{\mathsf{c}}

\newcommand{\Fil}{\mathrm{Fil}}
\newcommand{\bpi}{\boldsymbol{\pi}}

\definecolor{mb}{rgb}{0.36, 0.54, 0.66}

\newcommand{\fS}{\mathfrak{S}}

\renewcommand{\sp}{\mathrm{sp}}

\newcommand{\mpr}{\textcolor{purple}}

\newcommand{\GSp}{\mathrm{GSp}}

\newcommand{\sD}{\mathcal{D}}

\newcommand{\geo}{{\mathrm{geo}}}
\renewcommand{\div}{\mathrm{div}}
\newcommand{\val}{\mathrm{val}}

\newcommand{\shC}{\mathscr{C}}

\usepackage[colorinlistoftodos,textsize=scriptsize]{todonotes}
\usepackage{marginnote}

\usepackage{changepage}

\usepackage{xparse}
\makeatletter
\RenewDocumentCommand{\title}{om}{%
   \IfNoValueTF{#1}
     {\gdef\shorttitle{The Tate conjecture and singularities}}
     {\gdef\shorttitle{#1}}%
   \gdef\@title{#2}%
}
\makeatother

\title{The Tate conjecture for surfaces of geometric genus one---embracing singularities}

\author{Haoyang Guo and Ziquan Yang}

\date{June 10, 2025}

\begin{document}

\begin{abstract}
    In this article, we aim to largely complete the program of proving the Tate conjecture for surfaces of geometric genus one, by introducing techniques to analyze those surfaces whose ``natural models'' are singular. 
    As an application, we show that every elliptic curve of height one over a global function field of genus one and characteristic $p \ge 11$ satisfies the Birch--Swinnerton-Dyer conjecture. 
\end{abstract}

\maketitle

\vspace{-2em}
\tableofcontents

\section{Introduction}
\subsection{Statements of main results}
This paper aims to largely complete the program of proving the Tate conjecture via the Kuga--Satake construction for surfaces of geometric genus one, whenever the moduli is of an expected form. 
Previously, this has been achieved for K3 surfaces, after several decades of efforts made by many people in various contexts and generality. 
Here is a non-exhaustive list: \cite{DelK3Weil, Nygaard, NO, Tan1, Tan2, Andre, Charles, Maulik, MPTate, KimMP16, IIK}. 
Later on, in characteristic $0$, Moonen \cite{Moonen} proved the divisorial Tate conjecture (as well as the Mumford--Tate conjecture for the cohomology in degree $2$) for general varieties with Hodge number $h^{2, 0} = 1$, under a mild assumption on moduli. 
This was then extended to postive characteristics by the second author with Hamacher and Zhao in \cite{HYZ}.

In this paper, we introduce techniques to treat the singularities on the natural models of elliptic and general type surfaces, an important issue left open in \cite{HYZ}. 
Here the ``natural models'' refer to the Weierstrass normal forms for elliptic surfaces, and the canonical models for general type surfaces.
\footnote{Note that, interestingly, for K3 surfaces there is no such thing as a ``natural model'', and one has to artificially choose a quasi-polarization in order to consider moduli spaces.}
The challenge posed by the singularities arises from the \textit{failure of purity} of abelian schemes over a general positive or mixed characteristic base, which obstructs a direct construction of mixed characteristic period morphisms. 
We shall explain more about this in \cref{sub:ideals_of_proof} below.

Our motivation to tackle this problem is twofold. One is to investigate the mysterious link between singularities and the Tate conjecture, which was already noted in the 1982 paper by Rudakov--Zink--Shafarevich \cite{RZS} (cf. \cref{rmk: singularities and Tate}). The other is to contribute to the literature a clean example of the BSD conjecture, for which evidence is still meager.

\begin{theorem}
\emph{(\cref{thm: BSD_more_general})}
\label{thm: BSD}
    Over a global function field $K$ of genus $1$ and characteristic $p \ge 11$, every elliptic curve $\sE$ of height $1$ satisfies the BSD conjecture.\footnote{It is also proved for $p = 5, 7$ except for 3 special cases in \cref{prop: inspearable j} (cf. \cref{rmk: purity and supersingular}).}
\end{theorem}

Let us briefly recall the terminology. 
Suppose that $K$ is the function field of a smooth proper curve $C$ and $\pi : X \to C$ is the elliptic surface (with a zero section) obtained by taking the minimal model of $\sE$. 
Then the height of $\sE$ is the degree of the fundamental line bundle $(R^1 \pi_* \sO_X)^\vee$. 
We refer the reader to \cite{Mil75} and \cite{KT03} on the equivalence between the Tate conjecture for $X$ and the BSD conjecture for $\sE$, and remark that elliptic K3 surfaces, treated back in \cite{ASD}, correspond to the case when the genus of $K$ is $0$ and the height is $2$. 
In \cite[Thm.\ A]{HYZ}, the above result was obtained when all fibers of $\pi$ are irreducible, or equivalently the Weierstrass normal form of $X$ is smooth.

Similarly, in \cite[Thm.\ C]{HYZ}, we proved the Tate conjecture for certain surfaces of geometric genus one over $k$, assuming that the canonical divisor is ample and not just big and nef. Now we can drop this assumption. 
Below for a smooth projective surface $X$, we use $K_X$ to denote its canonical divisor, and use $p_g = h^{2,0}$ to denote its geometric genus.
\begin{theorem}
\label{thm: general type}
    Let $k$ be a field finitely generated over $\IF_p$ for $p \ge 5$. Then any minimal smooth projective surface $X$ over $k$ with $p_g = K_X^2 = 1$ and $h^{1,0} = h^{0, 1} = 0$ satisfies the Tate conjecture.
\end{theorem}

Both \cref{thm: BSD} and \cref{thm: general type} are special cases of a general theorem.
We state it in the rough form below and refer the reader to \cref{thm: general theorem} for the precise conditions. 

\begin{theorem}
\emph{(\cref{thm: general theorem})}
\label{intro_main_thm}
    Let $p\geq 5$ be a prime number, let $\sfM$ be a smooth $\IZ_{(p)}$-scheme and let $\sX \to \sfM$ be a projective family of surfaces whose geometric fibers are generically smooth with $p_g = 1$ and have at worst rational double point (RDP) singularities. 
    
    Assume that $\sfM$ admits a smooth morphism to another smooth $\IZ_{(p)}$-scheme $\sfB$ with a relative compactification $\overline{\sfM}$, and satisfies two additional conditions: 
    \begin{enumerate}[label=\upshape{(\roman*)}]
        \item The discriminant locus $\sfD\subseteq \sfM$ of the family $\sX\to \sfM$ and the relative boundary $\sfH \colonequals \overline{\sfM} \smallsetminus \sfM$ are both generically reduced in codimension $1$ modulo $p$.
        \item Over $\IC$, the Kodaira--Spencer map generically has rank $\ge 2$ and does not vanish along a general fiber of $\sfM \to \sfB$. 
    \end{enumerate}
    Then the minimal (and hence any) resolution of any fiber of $\sX$ satisfies the Tate conjecture as long as it has the same Hodge diamond as the generic fiber or is non-supersingular\footnote{In this paper, a surface is said to be \textit{supersingular} if its second crystalline cohomology is purely of slope $1$.}. 
\end{theorem}

In the above theorem, $\sfM$ should be interpreted as a coarse moduli space, $\sX\to \sfM$ the family of natural models, and $\sfB$ the moduli space of some auxilliary varieties. 
For example, when $\sfM$ parametrizes fibered surfaces, $\sfB$ parametrizes the base curves; and when $\sfM$ parametrizes complete intersections of $r$ hypersurfaces in some fixed ambient space, $\sfB$ parametrizes the $3$-fold given by intersecting the first $(r - 1)$ hypersurfaces. 
Moreover, the fibers of $\sfM \to \sfB$ are often open subspaces of certain linear systems, and hence can be compactified in a natural way, giving rise to the relative compactification $\overline{\sfM}$---this is exactly what happens for both \cref{thm: BSD} and~\ref{thm: general type}. We also remind the reader that on natural models of elliptic or general type surfaces over an algebraically closed field, singularities are at worst rational double points. 

\begin{remark}[Other examples]
    We expect that the hypotheses of our general theorem apply to almost every class of surfaces with $p_g = 1$ (e.g., those listed in \cite[Thm.~9.4]{Moonen}) with minor modifications. 
    For another example, the canonical model of a projective surface $X$ over $\IC$ with $p_g = q = 1$ and $K_X^2 = 2$ is a double cover of the symmetric square $\mathrm{Sym}^2 E$ of an elliptic curve $E$, ramified along a divisor in a fixed linear system on $\mathrm{Sym}^2 E$ (\cite[Prop.~20]{Catanese3}). We can naturally consider such surfaces in positive characteristics, and construct a coarse moduli space $\sfM$ by taking $\sfB$ to parametrize $E$ and $\sfM$ to parametrize suitable divisors on $\mathrm{Sym}^2 E$. 
    Then \cref{intro_main_thm} applies to $\sfM$ for all but a few $p$ where condition (i) might fail (the generic fiber of $\sX \to \sfM$ is polarized by its canonical divisor, which has degree $2$. To check the condition on Kodaira--Spencer in (ii), see \cite[p.283]{Catanese3}).
However, we do not investigate these other examples here and leave the task to the interested readers, as verifying the hypotheses of \cref{intro_main_thm} becomes largely algebro-geometric considerations, like those in \cref{sec: thms 1.1 and 1.2}. 
\end{remark}

\begin{remark}[Singularities and the Tate conjecture]
\label{rmk: singularities and Tate} Before we delve into the proof, we emphasize that singularities in natural models are not mere technical obstacles.
    Rather, they are fundamentally linked to the Tate conjecture: 
    \begin{itemize}
        \item The resolution of these singularities is itself an important source of divisors on the smooth models of these surfaces. 
        \item Many coarse moduli spaces $\sfM$ are nearly simply connected, meaning the singularities on the fibers are \emph{necessary} to induce non-trivial geometric monodromy on the family $\sX$.
    Otherwise, there would be no interesting variations of zeta functions or geometric Picard number among fibers. 
    \end{itemize}
    In fact, this latter point is manifested in our arguments. 
    Let us also remark that back in 1982, Rudakov, Zink and Shafarevich (\cite{RZS}) proved the Tate conjecture for supersingular K3 surfaces with a degree $2$ polarization precisely by considering singular degenerations: if such a surface is elliptic, then this follows from \cite{ASD}; otherwise, the surface can be deformed to a singular surface, whose resolution must be elliptic, so one concludes by combining \cite{ArtinSS}.  
    This strategy was later vastly extended by Maulik in \cite{Maulik}. 
    In this paper, by both making use of singularities and overcoming the obstacles they induce, we hope to further elucidate the mysterious connection singularities and the Tate conjecture. 
\end{remark}

\subsection{Ideas of proof}
\label{sub:ideals_of_proof}
We now explain the ideas of proof.
Our approach builds upon Madapusi' work in \cite{MPTate}, but introduces several key innovations that greatly expand its scope. Therefore, we start with a recollection of Madapusi's method and point out the difficulties for generalizations.

\textbf{Recollection of Madapusi's Method.} Let $\sfM$ be the moduli space of (quasi-polarized) K3 surfaces over $\IZ_{(p)}$ and $\sX\to \sfM$ be the universal family.
The first step is to construct a period morphism $\rho:\sfM \to \shS$ (up to an \'etale cover of $\sfM$), where $\shS$ is the canonical integral model of an orthogonal Shimura variety, over which there is a universal Kuga--Satake abelian scheme $\sA$. 
Then the second step is to show that for each point $s \in \sfM$, the original Tate conjecture for divisors of the K3 surface $\sX_s$ can be translated to a variant of Tate's theorem for a distinguished class of endomorphisms $\LEnd(\sA_{\rho(s)})\subset \End(\sA_{\rho(s)})$, called \textit{special endomorphisms}, of the abelian variety $\sA_{\rho(s)}$.

In order to achieve the above strategy, the following two facts played essential roles in Madapusi's arguments:
\begin{enumerate}[label=\upshape{(\Alph*)}]
	\item\label{Problem_A}  Quasi-polarized K3 surfaces over $\IZ_{(p)}$ have a \textit{fine} and regular moduli space $\sfM$, and abelian schemes over it satisfy a \emph{purity} condition. 
	\item\label{Problem_B}  The integral period morphism $\rho : \wt{\sfM} \to \shS$ is \textit{\'etale}, where $\wt{\sfM}$ is some \'etale cover of $\sfM$. 
\end{enumerate}
Property \ref{Problem_A} is needed for the very construction of an integral period morphism $\rho : \sfM \to \shS$. 
Indeed, unlike Shimura varieties of PEL type, those of Hodge or abelian type (including $\shS$ used here) lack a direct moduli interpretation on their integral models (cf. \cite{KisinInt}). 
Instead, the only known general method of constructing morphisms to those integral Shimura varieties is to use their extension property as in \textit{loc. cit.}, where the test schemes must satisfy a purity condition for abelian schemes.\footnote{For a given scheme $T$, the purity of abelian schemes over $T$ is a condition on $T$.
	It states that any abelian scheme $A_U\to U$ can be extended over the entire scheme $T$, where $U\subset T$ is an open subscheme away from a codimension $2$ locus.}
Even for Siegel moduli spaces, the purity condition is strictly necessary for the extension property (cf. \cite[\S6~Remarks]{VasiuZink}).
Luckily, deformations of quasi-polarized K3 surfaces are sufficiently well understood so that the purity condition can be checked explicitly using Vasiu--Zink's criterion in \textit{loc. cit.} (\cite[Thm.~3.8]{MPTate}).

In \cite{MPTate}, property \ref{Problem_B} was used to establish the aforementioned translation between the divisor classes of $\sX_s$ and the special endomorphisms of $\sA_{\rho(s)}$. 
More precisely, the \'etaleness is used to show that special endomorphisms can be lifted to characteristic $0$, so that one can construct the corresponding divisor on $\sX_s$ via the Lefschetz $(1, 1)$-theorem and a specialization argument.

For general surfaces of geometric genus one, \textit{neither} \ref{Problem_A} nor \ref{Problem_B} is available.
Our contribution is to introduce new ideas that enable us to extend the above strategy without these two conditions, and to provide techniques for studying the singular fibers and their periods.
Before we explain the details, the following picture illustrates the relationship of various geometric objects that appear in the arguments.


\tikzset{every picture/.style={line width=0.75pt}} 

\vspace{1em}
\tikzset{every picture/.style={line width=0.75pt}} 
\begin{center}
\scalebox{0.7}{
\begin{tikzpicture}[x=0.75pt,y=0.75pt,yscale=-1,xscale=1]

\draw   (135.66,137.6) -- (420.2,137.6) -- (420.2,267.6) -- (135.66,267.6) -- cycle ;
\draw  [dash pattern={on 4.5pt off 4.5pt}]  (137.15,198.07) .. controls (210.77,172.85) and (283.24,212.49) .. (292.12,241.46) ;
\draw  [dash pattern={on 4.5pt off 4.5pt}]  (292.12,241.46) .. controls (365.64,200.06) and (398.11,208.04) .. (420.2,226.54) ;
\draw   (129,48.48) .. controls (129,44.68) and (141.58,41.6) .. (157.1,41.6) .. controls (172.62,41.6) and (185.2,44.68) .. (185.2,48.48) .. controls (185.2,52.28) and (172.62,55.36) .. (157.1,55.36) .. controls (141.58,55.36) and (129,52.28) .. (129,48.48) -- cycle ;
\draw    (129,48.48) -- (181.35,106.91) ;
\draw    (185.2,48.48) -- (131.57,106.91) ;
\draw    (131.57,106.91) .. controls (145.28,117.77) and (167.02,117.77) .. (181.35,106.91) ;
\draw   (362,27.98) .. controls (362,25.01) and (374.13,22.6) .. (389.1,22.6) .. controls (404.07,22.6) and (416.2,25.01) .. (416.2,27.98) .. controls (416.2,30.95) and (404.07,33.36) .. (389.1,33.36) .. controls (374.13,33.36) and (362,30.95) .. (362,27.98) -- cycle ;
\draw    (364.48,73.68) .. controls (377.7,82.18) and (398.67,82.18) .. (412.48,73.68) ;
\draw    (362,27.98) .. controls (372.75,46.16) and (391.81,62.06) .. (364.48,73.68) ;
\draw    (416.2,27.98) .. controls (388.38,48.5) and (400.57,69.55) .. (412.48,73.68) ;
\draw   (248,89.47) .. controls (248,77.95) and (250.48,68.6) .. (253.54,68.6) .. controls (256.6,68.6) and (259.08,77.95) .. (259.08,89.47) .. controls (259.08,101) and (256.6,110.35) .. (253.54,110.35) .. controls (250.48,110.35) and (248,101) .. (248,89.47) -- cycle ;
\draw    (253.54,68.6) -- (309.22,127.36) ;
\draw    (253.54,110.35) -- (327.69,109.93) ;
\draw    (318.45,74.24) -- (280.63,126.95) ;
\draw    (318.45,74.24) .. controls (331.65,87.11) and (329.85,96.76) .. (327.69,109.93) ;
\draw    (280.63,126.95) .. controls (290.75,129.85) and (295.58,133.59) .. (309.22,127.36) ;
\draw  [dash pattern={on 0.84pt off 2.51pt}]  (393,86) -- (393.19,166.25) ;
\draw [shift={(393.2,168.6)}, rotate = 89.86] [color={rgb, 255:red, 0; green, 0; blue, 0 }  ][line width=0.75]      (0, 0) circle [x radius= 3.35, y radius= 3.35]   ;
\draw  [dash pattern={on 0.84pt off 2.51pt}]  (291.2,145.6) -- (292.1,239.11) ;
\draw [shift={(292.12,241.46)}, rotate = 89.45] [color={rgb, 255:red, 0; green, 0; blue, 0 }  ][line width=0.75]      (0, 0) circle [x radius= 3.35, y radius= 3.35]   ;
\draw  [dash pattern={on 0.84pt off 2.51pt}]  (159.2,121.6) -- (160.17,190.25) ;
\draw [shift={(160.2,192.6)}, rotate = 89.19] [color={rgb, 255:red, 0; green, 0; blue, 0 }  ][line width=0.75]      (0, 0) circle [x radius= 3.35, y radius= 3.35]   ;
\draw    (488,95.6) -- (289.84,180.81) ;
\draw [shift={(288,181.6)}, rotate = 336.73] [color={rgb, 255:red, 0; green, 0; blue, 0 }  ][line width=0.75]    (10.93,-3.29) .. controls (6.95,-1.4) and (3.31,-0.3) .. (0,0) .. controls (3.31,0.3) and (6.95,1.4) .. (10.93,3.29)   ;
\draw    (199.2,247.34) .. controls (220.83,202.71) and (253.27,260.88) .. (274.89,216.25) ;
\draw    (203.51,205.91) .. controls (225.13,161.28) and (257.57,219.45) .. (279.2,174.82) ;
\draw    (199.2,247.34) -- (203.51,205.91) ;
\draw    (274.89,216.25) -- (279.2,174.82) ;
\draw    (256.2,207) ;
\draw [shift={(256.2,207)}, rotate = 0] [color={rgb, 255:red, 0; green, 0; blue, 0 }  ][fill={rgb, 255:red, 0; green, 0; blue, 0 }  ][line width=0.75]      (0, 0) circle [x radius= 3.35, y radius= 3.35]   ;
\draw    (201.2,221.34) .. controls (222.83,176.71) and (255.27,234.88) .. (276.89,190.25) ;
\draw [shift={(245.12,207.23)}, rotate = 193.25] [color={rgb, 255:red, 0; green, 0; blue, 0 }  ][line width=0.75]    (10.93,-3.29) .. controls (6.95,-1.4) and (3.31,-0.3) .. (0,0) .. controls (3.31,0.3) and (6.95,1.4) .. (10.93,3.29)   ;
\draw    (489,54) -- (547,48.6) ;
\draw    (489,54) .. controls (484,101.6) and (539,75.6) .. (545,113.6) ;
\draw    (547,48.6) .. controls (540,100.6) and (501,64.6) .. (495,116.6) ;
\draw    (495,116.6) -- (545,113.6) ;
\draw    (523,124.6) -- (523,205.6) ;
\draw [shift={(523,207.6)}, rotate = 270] [color={rgb, 255:red, 0; green, 0; blue, 0 }  ][line width=0.75]    (10.93,-3.29) .. controls (6.95,-1.4) and (3.31,-0.3) .. (0,0) .. controls (3.31,0.3) and (6.95,1.4) .. (10.93,3.29)   ;
\draw  [dash pattern={on 4.5pt off 4.5pt}]  (428,217) -- (500,218.56) ;
\draw [shift={(502,218.6)}, rotate = 181.24] [color={rgb, 255:red, 0; green, 0; blue, 0 }  ][line width=0.75]    (10.93,-3.29) .. controls (6.95,-1.4) and (3.31,-0.3) .. (0,0) .. controls (3.31,0.3) and (6.95,1.4) .. (10.93,3.29)   ;
\draw    (492,72.6) -- (545,59.6) ;
\draw    (504,83.6) -- (540,69.6) ;
\draw    (502,95.6) -- (538,99.6) ;
\draw    (498,105.6) -- (542,107.6) ;

\draw (161.07,167.14) node [anchor=north west][inner sep=0.75pt]   [align=left] {$\sfD^+$};
\draw (267.42,238.59) node [anchor=north west][inner sep=0.75pt]   [align=left] {$\sfD^-$};
\draw (372,174) node [anchor=north west][inner sep=0.75pt]   [align=left] {$\sfM^\circ$};
\draw (412,48) node [anchor=north west][inner sep=0.75pt]   [align=left] {smooth};
\draw (103,18) node [anchor=north west][inner sep=0.75pt]   [align=left] {mildest singularity};
\draw (226,47) node [anchor=north west][inner sep=0.75pt]   [align=left] {worse singularity};
\draw (362,274) node [anchor=north west][inner sep=0.75pt]   [align=left] {$\sfM$};
\draw (518,218) node [anchor=north west][inner sep=0.75pt]   [align=left] {$\shS$};
\draw (241,204) node [anchor=north west][inner sep=0.75pt]   [align=left] {$s$};
\draw (548,80) node [anchor=north west][inner sep=0.75pt]   [align=left] {$T$};
\draw (433,98) node [anchor=north west][inner sep=0.75pt]   [align=left] {$\pi$};
\draw (449,193) node [anchor=north west][inner sep=0.75pt]   [align=left] {$\rho$};
\draw (424,238) node [anchor=north west][inner sep=0.75pt]   [align=left] {(defined on $\wt{\sfM}^\circ$)};
\draw (73,71) node [anchor=north west][inner sep=0.75pt]   [align=left] {fibers:};
\draw (535,155) node [anchor=north west][inner sep=0.75pt]   [align=left] {$\tau$};
\end{tikzpicture}}
\end{center}

As in \cref{intro_main_thm}, we let $\sX \to \sfM$ be the family of natural models over a coarse moduli space, with $\sfD$ being the discriminant locus parametrizing singular fibers. The meaning of other symbols will be explained below.

\textbf{Extended Period Morphisms}
We first let $\sfM^\circ := \sfM \smallsetminus \sfD$ be the open smooth locus. 
As observed in the previous work of the second author with Hamacher and Zhao \cite{HYZ}, one can still construct a period morphism $\rho : \wt{\sfM}^\circ \to \shS$, where $\shS$ is a suitable integral orthogonal Shimura variety and $\wt{\sfM}^\circ$ is an \'etale cover of $\sfM^\circ$.

To analyze the singular fibers, the naive approach is to simultaneously resolve the family $f:\sX\to \sfM$ over some cover $\sfM'$ of $\sfM$ (which exists thanks to Artin--Brieskorn resolution \cite{Artin-Res}), and then extend $\rho$ over $\sfM'$. 
However, as discussed above, extending $\rho$ is a very delicate problem, which underscores the necessity of \ref{Problem_A} in \cite{MPTate}. 
As Artin--Brieskorn resolution does not offer any effective control on $\sfM'$, which can a priori be very singular or ramified over $\sfM$, it is unclear whether $\sfM'$ can be further resolved such that the purity of abelian schemes holds over it.

In order to remedy the above issues, we instead seek to partially extend the period morphism, and hence introduce the following notion.
\begin{definition}[cf. \cref{def: extended period}]
	Given a normal $\sfM$-scheme $T$ (with structure morphism denoted by $\pi$), a map $\tau : T \to \shS$ is called an \emph{extended period morphism} if the family $\sX|_T\to T$ admits a simultaneous resolution, and for some open dense subscheme $T^\circ \subseteq T$, the map $\tau$ fits into a commutative diagram as below 
	\begin{equation}
		\label{eqn: extended period in intro}
		\begin{tikzcd}
			T^\circ & T \\
			{\wt{\sfM}^\circ} & {\sfM} \\
			&& \shS.
			\arrow[hook, from=1-1, to=1-2]
			\arrow[from=1-1, to=2-1]
			\arrow["\pi", from=1-2, to=2-2]
			\arrow["\tau", curve={height=-12pt}, from=1-2, to=3-3]
			\arrow[from=2-1, to=2-2]
			\arrow["\rho", curve={height=12pt}, from=2-1, to=3-3]
		\end{tikzcd}
	\end{equation}
\end{definition}
The above notion is inspired by the idea of taking ``limiting Hodge structures'' in classical Hodge theory.
Let us use $\breve{X}$ to denote the minimal resolution of a surface $X$. 
Given $s \in \sfM$, our method to prove the Tate conjecture for $\breve{\sX}_s$ is to find a suitable arithmetic surface $S$ in $\sfM$ passing through $s$ such that:
\begin{itemize}
	\item after passing to a generically finite cover $T$ of $S$, there is an extended period morphism $\tau : T \to \shS$, and 
	\item the morphism $\tau$ is \textit{motivically aligned} at $s$, meaning that for each $t \in T$ that lifts $s$, there is a correspondence bewteen special endomorphisms of the Kuga-Satake abelian variety $\sA_{\tau(s)}$ and divisors on $\breve{\sX}_s = \breve{\sX}_t$ (see \cref{def: motivically aligned}).
\end{itemize}
With the above construction established, the Tate conjecture for $\breve{\sX}_s$ then follows quickly from a variant of Tate's theorem for special endomorphisms as in \cite{MPTate}. 

Below are two important aspects for extended period morphisms. 
\begin{enumerate}[label=\upshape{(\Roman*)}]
	\item \textit{Deformation techniques:} To prove that a point $t \in T$ is motivically aligned, the general method is to show that, for every special endomorphism $\zeta \in \LEnd(\sA_{\tau(t)})$, the pair $(t, \zeta)$ can be deformed to some $t' \in T$, such that $t'$ is motivically aligned. 
        This conceptualizes the key argument in \cite{MPTate} for K3 surfaces that eventually appeals to the Lefschetz $(1, 1)$-theorem. 
	In our more general situation, we shall extend this idea by not only considering mixed characteristic deformations, but also equi-characteristic ones. 
	\item \textit{Matching local systems:} For a point $t$ in  $T$ such that $\pi(t)$ is in the singular locus $\sfD$, we need to prove that the Shimura variety $\shS$ really ``knows'' the period of the resolution $\breve{\sX}_t$. 
	Precisely, this means that the pullbacks of various automorphic local systems $\bL_?$, with $?$ being Betti, de Rham, $\ell$-adic, or crystalline, admit natural isomorphisms with the $?$-cohomology of the resolution of $\sX$ over $T$.
	This is achieved by first doing this within the smooth locus $\sfM^\circ$, and then proving that this is also true over the boundary $T \times_\sfM \sfD$. 
	For $? = \ell, B$, this follows from the basic facts about local systems and results of \cite{HYZ}. 
	For $? = \crys$, we make an essential use of the rigidity theorem on extending F-(iso)crystals due to de Jong and Kedlaya--Drinfeld (see \cref{lem: rigidity of F-crystals}). 
\end{enumerate}

\textbf{The constructions for different loci.}
Now we explain how to construct the desired extended period morphisms $\tau : T \to \shS$. 
Considerations of enumerative geometry suggests that one should expect a rough correspondence between the singularities on $\sX_s$ and the singularity of $\sfD$ at $s$.\footnote{For plane curves, this is discussed in \cite[\S7.7.2]{3264}, but the principle is quite general.} 
In particular, the maximal smooth locus $\sfD^+$ of $\sfD$ should correspond to those fibers with mildest singularities. 
As we will see below, points on $\sfD^+$ and those on its complement $\sfD^-  \colonequals \sfD \smallsetminus \sfD^+$ will be treated separately via different constructions.\\\\
\textit{Case $s \in \sfD^+$:} By a standard Bertini-type argument, we can find a curve $C \subseteq \sfM_{\IF_p}$ that passes through the point $s$ and interacts $\sfD^+_{\IF_p}$ transversely.
We also choose a mixed characteristic scheme $\sC$ that lifts the curve $C$.
Then, up to passing to a generically \'etale cover $\sC_1 \to \sC$, we may find an extended period morphism over the two dimensional scheme $T = \sC_1$, where
the morphism $\sC_1 \times_\sfM \sfM^\circ \to \shS$ extends over $\sC_1 \times_\sfM \sfD$ thanks to the N\'eron--Ogg--Shafarevich criterion.
Here it is critical that the curve $\sC_1$ is delicately constructed with the aid of Abhyankar's lemma so that $\sC_1 \times_\sfM \sfD$ is \'etale over $\IZ_p$. 
In particular, the construction enables us to analyze the behavior of the monodromy using Grothendieck's specialization theorem for the tame fundamental groups at $\sC_1$.
Afterwards, the arguments of \cite{HYZ} for the smooth locus $\wt{\sfM}^\circ$ can be extended to the boundary $\sC_1 \times_\sfM \sfD$ (see \cref{prop: good points}):
The idea is that if there is a special endomorphism $\zeta \in \LEnd(\sA_{\tau(t)})$ such that the pair $(t, \zeta)$ does not lift to a point on $\sC_1$ in characteristic zero, then one can derive a contradiction regarding the dimension of monodromy invariants. \\\\
\textit{Case $s \in \sfD^-$:} This case is significantly more complicated and takes more inputs, as $\sfD^-$ parametrize those with worse singularities and $\sfD$ itself might be very singular.
In particular, the aforementioned mixed characteristic curve may not exist, so this time we let $T$ be an actual surface over $\IF_p$, where a combination of Artin--Brieskorn resolution of the family $\sX|_T\to T$ and Lipman's resolution of the base $T$ allow us to construct an extended period morphism $\tau:T\to \shS$.

At this point, we know that the morphism $\tau$ is motivically aligned at any point $t' \in T$, as long as the image $\pi(t')\in \sfM$ is contained in $\sfM\smallsetminus \sfD^-$. 
If for any given special endomorphism $\zeta\in \LEnd(\sA_{\tau(t)})$, the pair $(t, \zeta)$ can be deformed to any such $t'$, then we win. 
On the other hand, by a fact on special divisors for orthogonal Shimura variety, the pair $(t, \zeta)$ always deforms along some divisor on $T$.
However, it is a priori possible that this divisor is contracted by $\pi : T \to \sfM$, and in particular the motivic alignedness has not been established on any point of this divisor.
We also note that the periods of the point $s$, defined as the image subscheme $\tau(\pi^{-1}(s)) \subset \shS$, can in principle be positive dimensional. 

However, we make the following critical observation that prevents the above scenario:
\begin{equation*}
	\label{eqn: exceptional}
	\text{Any $\pi$-exceptional curve is also $\tau$-exceptional.} 
\end{equation*}
Heuristically, the above is saying that up to finite ambiguity, any point in $\sfD^-$ admits a unique ``period'' on $\shS$. 
To achieve the above observation, we apply the minimal model program for 3-folds in positive characteristic, a Beauville--Laszlo  type gluing theorem for projective schemes, as well as the matching of F-crystals established in (II) above. 
In the non-supersingular case, we also give a different argument using the positivity of Hodge bundle on $\shS$. 
Here we remark that the $\pi$-exceptional curves are not necessarily $(-1)$-curves, and in particular it is not clear if they can be blown down for the morphism $\tau$. 
Nonetheless, we can show that it is not possible for $(z, \zeta)$ to deform only along $\pi$-exceptional curves, as this would contradict the Hodge index theorem on some compactification of $T$. 

Below we also comment on some technical caveats and subtleties of this paper. 
\begin{remark}[Matching isomorphism for crystalline coefficients]
    \label{rmk: matching crystalline} In \cite{HYZ}, it was not necessary to match $\bL_\crys$ with the (primitive) crystalline cohomology of $\sX$ over $\wt{\sfM}^\circ$. But now it becomes important when dealing with points on $\sfD^-$, so we fill this in. As in \cite{HYZ}, much of the technicalities arise because we do not assume that the classical Kuga--Satake construction is \textit{absolute Hodge}, or that the motives of the varieties in question are abelian.
    So Madapusi's argument to do this step in \cite{MPTate}, which relies on Blasius' result on the de Rhamness of the Hodge cycles on abelian motives \cite{Blasius}, does not apply. 
    
    In this paper, we shall establish the matching of these F-(iso)crystals using the recent advances on the $p$-adic Riemann--Hilbert functor in \cite{DLLZ} and integral $p$-adic Hodge theory of Bhatt--Morrow--Scholze in \cite{BMS}. 
    Then we prove a compatibility statement for the outputs of the $p$-adic and classical Riemann--Hilbert functors using the analytic density of Noether--Lefschetz loci (see \cref{sub:match crystalline}). 
    This steps addresses a question with a similar flavor to \cite[Conjecture~1.4]{DLLZ} in our specific setting.
    Therefore, we think these techniques might be of independent interest. 
\end{remark}
\begin{remark}[Assumptions on $p$]
    The assumptions on $p$ in this paper only come from condition (i) in \cref{thm: general theorem}, together with the requirement that minimal resolutions of fibers of $\sX$ should have the correct Hodge numbers whenever they are supersingular. 
    In fact, the latter is not even required for points on $\sfM^\circ$ or $\sfD^+$. 
    Beyond the potential influence on Hodge numbers of resolutions, our method is insenstive to the \textit{type} of singularities of fibers over $\sfD^-$ with respect to $p$. 
    For example, we do \emph{not} assume that $p$ is coprime to the order of the Weyl group of the relevant RDPs (see \cref{exp: not good prime})---imposing such restrictions is often necessary for studying the local monodromy around RDPs over mixed characteristic discrete valuation rings. 
    So in this regard our method is quite robust. 
\end{remark}

\subsection{Notations and conventions} 
Let $S$ be a base scheme. 
Given a field-valued point $s$ on $S$, we write $k(s)$ for its residue field; we write $\bar{s}$ for a geometric point over $s$, when such a choice does not need to be specified. 
When the letter $k$ denotes a perfect field of characteristic $p > 0$, we simply use $W$ to denote its ring of Witt vectors $W(k)$. 
When $k$ is not necessarily perfect, we write $k^p$ for the perfect closure of $k$.
For a morphism $f : S \to T$ of finite type between schemes, we customarily use $\overline{S}$ to denote some relative compactification, i.e., a scheme $\overline{S}$ which is proper over $T$ and contains $S$ as an open dense subscheme.
We also remind the reader that a \emph{global functional field} is the fraction field of a curve over a finite field.

By a (relative) surface $f : X \to S$ we mean a separated morphism in the category of algebraic space whose fibers are of dimension $2$. 
Note that a smooth proper surface over a field is automatically a scheme (\cite[Thm~4.7]{Artin-AS}) and hence projective (\cite[p.~328, Cor.~4]{Kl66}). 
Given a normal surface $X$ over a field, we write $\breve{X} \to X$ for the minimal resolution of $X$, which is unique up to unique isomorphism (see for example \cite[\S9.3.4, Prop.~3.32]{LiuBook}).

\section{Geometric preparations}
To prepare for the proof of main theorems, we introduce various geometric ingredients in this section, which can be of independent interests.
These include results on F-crystals, resolutions of surface singularities, the minimal model program for threefolds, certain application of relative Abhyankar's lemma, and geometry of elliptic surfaces in positive characteristics.

\subsection{Rigidity of F-crystals}
\label{sub:F-crystal}
In this subsection, we prove that the restriction functor of F-crystals along an open immersion is fully faithful, and give a fiberwise criterion for the local freeness of crystalline cohomology.

\begin{definition}
    Let $S$ be a scheme such that $p \sO_S = 0$, $F_S : S \to S$ be the absolute Frobenius morphism, and $F_{S/\IZ_p}$ be the inducded morphism on the crystalline topos $(S/\IZ_p)_\crys$. An F-crystal $H$ on $S$ is a crystal of vector bundles on $(S/\IZ_p)_\crys$ together with a morphism $F : F_{S/\IZ_p}^* H \to H$. It is said to be \textit{non-degenerate} if for some $i \ge 0$, there exists a morphism $V : H \to F^*_{S/ \IZ_p} H$ such that $VF = p^i \mathrm{id}_H$. Denote the category of non-degenerate F-crystals on $S$ by $\FCrys(S)$.   
\end{definition}

The following theorem follows from a combination of results of \cite{deJong} and \cite{DK17}: 

\begin{theorem}
\label{lem: rigidity of F-crystals}
Let $k$ be a perfect field of characteristic $p > 0$ with $W \colonequals W(k)$. Let $S$ be a smooth $k$-variety and $U$ be a dense open subset of $S$. Then the restriction functor $\FCrys(X) \to \FCrys(U)$ is fully faithful. 
\end{theorem}
\begin{proof}
    We let $M$ and $N$ be two objects of $\FCrys(X)$. 
	For any thickening $(T,X) \in (X/W)_\crys$, denote the coherent $\sO_T$-module given by evaluating $M$ on $T$ by $M_T$, and similarly for $N$.
	As both $\FCrys(X)$ and $\FCrys(U)$ admit Zariski descent, we may assume both $U=\Spec(R[1/f])$ and $X=\Spec(R)$ are affine schemes, where $X$ admits an \'etale morphism to the affine space $\mathbb{A}^n_k$ and $U=X\backslash V(f)$ for a nonzerodivisor $f\in R$.
	By deformation theory, there is a smooth $p$-adic affine formal scheme $\tilde{X}=\mathrm{Spf}(\tilde{R})$ over $W$ together with a semilinear endomorphism $F_{\tilde{X}}:\tilde{X}\to \tilde{X}$ that lifts $X$ with its Frobenius morphism.
	Moreover, the object $(\tilde{X},X)$ forms a weakly final object in the crystalline site $(X/W)_\crys$.
	We let $\tilde{f}$ be any lift of $f$ in $\tilde{R}$, and let $\tilde{U}=\mathrm{Spf}((\tilde{R}[1/\tilde{f}])^\wedge_p)$ and the $p$-complete localization of $\tilde{R}$ at $\tilde{f}$.
	Then similarly the thickening $(\tilde{U},U)$ covers the crystalline site $(U/W)_\crys$.
	Here we notice that by construction, the map of rings $\sO(\tilde{X})\to \sO(\tilde{U})$ is injective.
	
	We first prove the faithfulness. By definition $\Hom_{\FCrys}(M,N)$ is the subgroup of $\Hom_{\mathrm{Crys}}(M,N)$ that preserves the Frobenius structures (and similarly for the restrictions on $U$), and it suffices to show that $\Hom_{\mathrm{Crys}}(M,N) \to \Hom_{\mathrm{Crys}}(M|_U,N|_U)$ is injective.
	Moreover, since the thickenings $(\tilde{X},X)$ and $(\tilde{U},U)$ are weakly final in the sites $(X/W)_\crys$ and $(U/W)_\crys$ respectively, 
	the maps $\Hom_{\mathrm{Crys}}(M,N) \to \Hom_{\sO_{\tilde{X}}} (M_{\tilde{X}}, N_{\tilde{X}})$ and $\Hom_{\mathrm{Crys}}(M|_U,N|_U) \to \Hom_{\sO_{\tilde{U}}} (M_{\tilde{U}}, N_{\tilde{U}})$  are injective.
	Thus it is left to show the injection of the map  $\Hom_{\sO_{\tilde{X}}} (M_{\tilde{X}}, N_{\tilde{X}}) \to \Hom_{\sO_{\tilde{U}}} (M_{\tilde{U}}, N_{\tilde{U}})$, which follows from the local freeness of $M$ and $N$ together with the injection of the rings $\tilde{R} \to (\tilde{R}[1/\tilde{f}])^\wedge_p$.
	Here we also notice that by the local freeness again, the group $\Hom_{\sO_{\tilde{X}}} (M_{\tilde{X}}, N_{\tilde{X}})$ is $p$-torsionfree, and thus so is the subgroup $\Hom_{\FCrys}(M,N)$.
	
	We then prove the fullness.
	Let $\varphi : M|_U \to N|_U$ be a morphism between their restrictions to $U$. 
	By \cite[Thm~2.2.3]{DK17}, there exists some $r \ge 0$ such that $p^r \varphi$ can be extended to a morphism $\psi : M \to N$. 
	Moreover,  continuing with the discussion in the last paragraph, an arrow in $\Hom_{\FCrys}(M,N)$ is equivalent to a map of $\sO_{\tilde{X}}$-modules $M_{\tilde{X}}\to N_{\tilde{X}}$ that is compatible with connections and Frobenius structures (\cite[Thm.~6.6]{BO}).
	Thus by the $p$-torsionfreeness of the hom groups, it suffices to show that $\psi \in p^r \Hom_{\sO_{\tilde{X}}}(M_{\tilde{X}}, N_{\tilde{X}})$.
	
	Now we may adapt the argument of \cite[Lem.~6.15]{Maulik}. By induction we may assume $r = 1$. 
	Up to possibly shrinking $X$ and choosing bases of $M$ and $N$, the map $\psi$ is given by a matrix $P$ with entries in $\sO_{\tilde{X}}$, and
	it suffices to show that the entries of the matrix are contained in $p\sO_{\tilde{X}}$. 
	Let $s \in X$ be any $k$-point. 
	We may choose an affine smooth irreducible curve $C \subseteq X$ which passes through $s$ and $C \cap U \neq \emptyset$. 
	Let $\sO_{C, s}$ be the localization of $\sO_C$ at $s$. 
	By \cite[Cor.~1.2]{deJong}, the restriction of $\varphi$ to $\Spec(K(C))$ extends uniquely to a map of $F$-crystals $\varphi_{\Spec(\sO_{C,s})}$ over $\Spec(\sO_{C, s})$. 
	In particular, similarly to the second paragraph above, the restriction along $\sO_{C,s} \to K(C)$ is injective on hom groups, and we have $\psi|_{\Spec(\sO_{C,s})} = p\varphi_{\Spec(\sO_{C,s})}$.
	Thus the reduction $\bar{P}$ of the matrix $P$ at the point $s\in \tilde{X}(k)$ is zero.
	
	Finally, we let $\fm$ be the maximal ideal of $\sO_{X}$ given by $s$.
	The above implies that the entries of the mod $p$-reduction $\bar{P}$ are contained in $\fm$. 
	Since the choice of $s$ is arbitrary, the entries of the reduction $\bar{P}$ lie in the Jacobson ideal of $\sO_X$, which is zero as $X$ is reduced and finite type over $k$. 
	Hence the matrix $P$ is divisible by $p$, which finishes the proof.
\end{proof}

For the later application, we also give a criterion on the local freeness of the integral crystalline cohomology.
\begin{lemma}
	\label{lem:local_free_of_crys_coh}
	Let $S$ be a smooth $k$-variety and $f : X \to S$ is a smooth proper morphism. 
	Suppose that for each $n$ and for each closed point $s \in S$, we have $\dim \H^n_\dR(X_s/k(s))$ is equal to the rank of the F-isocrystal $R^n f_{\crys *} \sO_{X/W}[1/p]$ on $\crys(S/W)$. 
	Then $R^n f_{\crys*} \sO_{X/W}$ is locally free for each $n$. 
\end{lemma}
\begin{proof}
	If $S = \Spec(k)$ is a point, this is a well known consequence of the universal coefficient theorem. The general case is actually not much different. 
	Up to shrinking $S$ by an open subvariety, let us choose a formally smooth lifting $\wt{S}$ of $S$ over $W$. 
	We may assume that $\wt{S} = \mathrm{Spf}(R)$ for some $p$-completely smooth $W$-algebra $R$. 
	
	Then it suffices to argue that the evaluation $\sH$ of $R^n f_{\crys*} \sO_{X/W}$ on $\wt{S}$, which a priori is a coherent sheaf, is locally free over $\wt{S}$. 
	By the fiberwise criterion of the local freeness (\cite[\href{https://stacks.math.columbia.edu/tag/0FWG}{Tag 0FWG}]{stacks-project}) it suffices to show that $\dim \sH \tensor k(s) = m$ for all closed point $s$. 
	
	Consider $R f_{\crys *} \sO_{X/W}$, which is an object in the derived category of $\sO_{S/W}$-modules on $(S/W)_\crys$. 
	We let $C^\bullet$ be the $R$-linear complex given by the evaluation of $R f_{\crys *} \sO_{X/W}$ on $\wt{S}$.
	Then we have $\H^n(C^\bullet) = \sH$. 
	On the other hand, by the base change formula of crystalline cohomology \cite[Cor.~7.12]{BOBook} and the crystalline--de Rham comparison theorem, we have $\H^n(C^\bullet \tensor_{R}^\IL k(s)) \simeq \H^n_\dR(X_s/k(s))$. 
	Now we recall that there is a spectral sequence \cite[\href{https://stacks.math.columbia.edu/tag/061Y}{Tag 061Y}]{stacks-project}
	$$ \mathrm{Tor}^R_{j} (\H^i(C^\bullet), k(s)) \Rightarrow \H^{i - j}(C^\bullet \tensor^\IL_R k(s)),$$
	inducing an injection $\H^n(C^\bullet) \tensor k(s) \into \H^n_\dR(X_s/k(s))$. 
	Note that by the semicontinuity of fiber dimensions of a coherent sheaf, we know that $\dim \H^n(C^\bullet) \tensor k(s) \ge \dim \H^n(C^\bullet) \tensor K(\wt{s})$, where $\wt{s}$ is any lift of $s$ on the generic fiber $\wt{S}_\eta$. 
	In addition, since $K(\wt{s})$ is in characteristic zero, the dimension $\dim \H^n(C^\bullet) \tensor K(\wt{s})$ is equal to the rank of the $F$-isocrystal.
	Hence by the assumption, the injection must in fact be an isomorphism, and we are done. 
\end{proof}

\subsection{Resolving rational double points}
\label{sub:RDP}
We recall the Artin--Brieskorn resolution of the surface singularity, together with some basic notions and observations on the minimal resolution of surfaces.

\begin{definition}
\label{def: resolution}
    Let $S$ be a separated noetherian scheme (or algebraic space) and $f : X \to S$ be a relative surface whose geometric fibers are normal. 
    \begin{itemize}
        \item A (minimal) \textbf{simultaneous resolution} of $X/S$ is a smooth surface $Y/S$ equipped with a proper and birational morphism $\pi : Y \to X$ over $S$ such that for each geometric point $s \to S$, the map of fibers $\pi_s : Y_s \to X_s$ is the minimal resolution. 
        \footnote{We remind the reader to not confuse the notion with the minimal model of a smooth surface. In particular, a minimal resolution of a (singular) surface is not necessarily a minimal model.}
        \item We say that $X/S$ is Zariski-locally \textbf{resolvable} if $S$ admits a Zariski open cover $\{ U_\alpha \}$ such that the restriction of $X$ to $U_\alpha$ admits a simultaneous resolution.\footnote{Note that simultaneous resolutions are usually non-unique when they exist. } 
    \end{itemize}
\end{definition}

We also record a useful fact about compactification of surfaces.
\begin{lemma}
\label{lem: smooth compactification}
    Let $S$ be an irreducible smooth algebraic surface over a field $k$. Then there is always a smooth compactification $\overline{S}$ of $S$, i.e., a smooth proper surface $\overline{S}$ over $k$ which contains $S$ as an open dense subvariety. 
\end{lemma}
\begin{proof}
    First, by Nagata's compactification theorem \cite[\href{https://stacks.math.columbia.edu/tag/0F41}{Tag 0F41}]{stacks-project}, we can find some integral proper algebraic surface $S'$ over $k$ which compactifies $S$. Then we obtain $\overline{S}$ by applying Lipman's theorem on resolution of surfaces \cite[\href{https://stacks.math.columbia.edu/tag/0BGP}{Tag 0BGP}]{stacks-project}. Note that condition (4) in \textit{loc. cit.} is satisfied simply because $S'$ is quasi-excellent (see the second paragraph of \cite[\href{https://stacks.math.columbia.edu/tag/0ADX}{Tag 0ADX}]{stacks-project}). 
\end{proof}

Now we recall an important consequence of Artin--Brieskorn resolution for a family of surfaces.
\begin{proposition}[Artin--Brieskorn resolution]
\label{lem: etale locally resolvable}
    Let $S$ be a connected smooth algebraic surface or curve over a field $k$, and let $f : X \to S$ be a generically smooth relative surface.
    Assume the geometric fibers of $f$ are normal and have at worst RDP singularities, and the singular locus of $f$ is finite over $S$.
    Then there exists a surjective and generically \'etale morphism $S' \to S$ between smooth connected algebraic surfaces over $k$, such that $X|_{S'}$ is Zariski-locally resolvable. 
\end{proposition}
\begin{proof}
    We prove the statement when $S$ is a surface, the curve case follows from the similar argument, but is simpler. 

    By the Artin--Brieskorn resolution \cite[p.~332]{Artin-Res}, there is an algebraic space $R$ parametrizing simultaneous resolutions of $X/S$ with the following properties. 
    \begin{enumerate}[label=\upshape{(\roman*)}]
        \item The induced map on geometric points is a bijection.
        \item For every geometric point $s$ on $S$, the induced map between Henselizations $\wt{R}_s \to \wt{S}_s$ is surjective.
        \item The map $\wt{R}_s \to \wt{S}_s$ is an isomorphism when $X_s$ is smooth.
    \end{enumerate}
  By taking elementary \'etale neighborhoods \cite[\href{https://stacks.math.columbia.edu/tag/0BGW}{Tag 0BGW}]{stacks-project}, we can find an \'etale cover $\{ S_\alpha \}$ of $S$, together with a finite surjective and generically \'etale morphism $R_\alpha \to S_\alpha$ for each $\alpha$, such that that each composition $R_\alpha \to S_\alpha \to S$ factors through the map $R\to S$. 
  Here the quasi-finiteness and the surjectivity follows from (i) and (ii), and the generic \'etaleness follows from (iii) and our assumption that $X/S$ is generically smooth. Note that this step is quite general and does not need that $S$ is a surface. 

    Next we resolve each $R_\alpha$ and globalize them.
    Assume that each $R_\alpha$ or $S_\alpha$ is irreducible, and let $U_\alpha$ be a resolution of singularities of $R_\alpha$ (\cite[\href{https://stacks.math.columbia.edu/tag/0BGP}{Tag 0BGP}]{stacks-project}).
    Let $U$ be a connected component of the fiber product of $U_\alpha$'s over $S$. 
    Then $U$ dominates each $U_\alpha$.
    By \cref{lem: smooth compactification}, we can find smooth compactifications $\overline{S}$, $\overline{U}_\alpha$ and $\overline{U}$ of $S, U_\alpha$ and $U$ respectively. Moreover, by applying \cite[\href{https://stacks.math.columbia.edu/tag/0C5H}{Tag 0C5H}]{stacks-project} repeatedly, we may arrange that $U_\alpha \to S$ extends to $\overline{U}_\alpha \to \overline{S}$, and $U \to U_\alpha$ extends to $\overline{U} \to \overline{U}_\alpha$.
    Then for every $\alpha, \beta$, the diagram
    \begin{equation}
        \begin{tikzcd}
	\overline{U} & {\overline{U}_\alpha} \\
	{\overline{U}_\beta} & {\overline{S}}
	\arrow[from=1-1, to=1-2]
	\arrow[from=1-1, to=2-1]
	\arrow[from=1-2, to=2-2]
	\arrow[from=2-1, to=2-2]
        \end{tikzcd}
    \end{equation}
commutes, because this is true when restricted to the open dense subvariety $U \subseteq \overline{U}$. 
Finally, we take $S'$ to be preimage of $S$ in $\overline{U}$. 
\end{proof}

Finally, we check that rational double points are invariant under deformations, even in positive and mixed characteristic. This seems well known to experts on singularities. However, many classical references restrict to considering characteristic $0$, so here we provide a proof.
\begin{proposition}[Deformation Invariance of RDP]
\label{prop: def RDP}
    Let $R$ be a strictly Henselian discrete valuation ring, and let $t$ and $\bar{\eta}$ be the special point and a geometric generic point of $T \colonequals \Spec(R)$ respectively. 
    
    Let $f : X \to T$ be a proper flat morphism of schemes of relative dimension $2$. 
    If $X_s$ has only at worst RDP singularties, then so does $X_{\bar{\eta}}$. 
\end{proposition}
\begin{proof}
    By \cite[\href{https://stacks.math.columbia.edu/tag/0BJJ}{Tag 0BJJ}]{stacks-project} and the assumption on $f$, the fact that $X_s$ is Gorenstein implies that the total space $X$ is also Gorenstein.
    In particular, the generic fiber $X_{\bar{\eta}}$ is Gorenstein. 
    Since a RDP singularity is equivalent to a Gorenstein rational singularity, it remains to show that the singularity on $X_{\bar{\eta}}$ is rational, i.e., the higher direct images of $\sO_{\breve{X}_{\bar{\eta}}}$ under the projection $\breve{X}_{\bar{\eta}} \to X_{\bar{\eta}}$ all vanish. To check this, we may replace $T$ by a finite ramified cover over which $f : X \to T$ admits a simultaneous resolution. This is allowed thanks to \cite[Thm.\ 3]{Artin-Res} (the deformation space of a RDP is smooth and irreducible). 
    Then the conclusion follows from the fact that the singularities on $X_s$ are rational, and the semi-continuity theorem applied to the resolution. 
\end{proof}

Finally we observe that there is a natural variation of Newton polygons for the minimal resolutions of a family of surfaces.

\begin{proposition}[Upper semi-continuity of Newton polygons]
	\label{prop:var_of_Netwon}
	Let $S$ be an integral algebraic variety over a field $k$ of characteristic $p$, and let $f : X \to S$ be a proper and generically smooth relative surface whose geometric fibers are normal with at worst rational double points singularities.
	For each integer $n\in \mathbb{N}$, the function
    \begin{equation}
        \psi : s \longmapsto \text{Newton polygon of}~\mathrm{H}^n_\crys(\breve{X}_{s^p} / W(k(s)^p))[1/p]
    \end{equation} 
    on $S$ is upper semi-continuous.
\end{proposition}
\begin{proof}
        Let $\eta$ be the generic point of $S$. Our goal is to prove that for every $s \in S$, $\psi(s) \ge \psi(\eta)$ and the subset 
        \begin{equation}
            \label{eqn: jumping locus}
            S_{> \eta} := \{ s \in S \mid \psi(s) > \psi(\eta) \}
        \end{equation}
        is a closed in $S$. By de Jong's theory of alterations \cite{deJongAlt}, there exists a smooth variety $S'$ together with a proper surjective morphism $S' \to S$. If the proposition holds for the pullback family $X|_{S'}$ over $S'$, then it holds for the original family $X$ over $S$. Therefore, we may assume that $S$ is smooth.

	Similarly, by comibining Artin--Brieskorn resolution \cite{Artin-Res} and de Jong's alterations \cite{deJongAlt}, we can find an open cover $\{S_\alpha\to S\}$, together with proper surjective morphisms $U_\alpha \to S_\alpha$ for each $\alpha$ such that the family $X|_{R_\alpha}$ admits simultaneous resolution $Y_\alpha$, and each $U_\alpha$ is smooth over $k$ (cf. the proof of \cref{lem: etale locally resolvable}).
	In particular, we obtain the identifications between the fibers $(Y_\alpha)_t$ and the minimal resolution $\breve{X}_t$, for each point $t\in U_\alpha$.

	Now, notice that since the closedness of a subset can be checked on an open cover, we may replace $S$ by one of the $S_\alpha$ for a fixed $\alpha$. 
        By the first paragraph, we may further replace 
    $S_\alpha$ by $U_\alpha$. 
    Therefore, we reduce to the case when $S$ is smooth over $k$, and $X$ is smooth over $S$. 
    The conclusion then follows from \cite[Thm.\ 3.12]{Ked22}(a), which is due to Grothendieck and Katz, plus a few additional arguments by Kedlaya. 
\end{proof}
We also comment on the purity of the Newton polygon.
\begin{remark}[Purity of Newton polygons]
\label{rmk: purity of Newton}
	It is natural to ask whether the variation of Newton polygons in \Cref{prop:var_of_Netwon} satisfies purity, namely, whether $S_{> \eta}$ defined in (\ref{eqn: jumping locus}) is of codimension $1$ in $S$ if non-empty. 

        The closest result in this direction is the purity theorem by de Jong--Oort \cite[Thm.\ 4.1]{dJO20}, which is about locally free F-crystals (not merely F-isocrystals). 
        Kedlaya provided an argument to show that the purity result also holds for F-isocrystals, at the expense of assuming that the base is a smooth variety (\cite[Thm.~3.12(b)]{Ked22}). 

        However, in the setting of \cref{prop:var_of_Netwon}, even when $S$ is smooth we do not see a way of proving the purity of the crystalline cohomology of the minimal resolutions using the above results: 
        If we use \cite{Artin-Res} to resolve the family $X$ over a quasi-finite cover $S'$ of $S$, then the problem is that we do not know if the crystalline cohomology of $X_{S'}$ is locally-free, or has a locally-free lattice; 
        if we further resolve this quasi-finite cover by a smooth cover using alterations, then the problem is that we lose control on the fiberwise relative dimensions, as alterations are only generically finite. 
        However, the situation of \cref{prop:var_of_Netwon} naturally appears when we consider a coarse moduli of surfaces, so it will be interesting to figure out whether purity still holds despite these obstructions. In particular, it is potentially useful for the Tate conjecture (cf. \cref{rmk: purity and supersingular}).  
\end{remark}

\subsection{Specialization of line bundles and cohomology}
\label{sub:specialization_of_lb}
For later use, we discuss some generalities on specializations of line bundles and cohomology. 

\begin{definition}
\label{def: path}
    Let $S$ be a noetherian scheme with geometric points $t$ and $s$. 
    A \textbf{path} from $t$ to $s$ is a tuple $(\gamma, t' \to t, s' \to s)$, consisting of a morphism of schemes $\gamma : T \to S$ together with morphisms of geometric points $t' \to t$ and $s' \to s$, where $T = \Spec(R)$ for a strictly Henselian discrete valuation ring $R$, and $t', s'$ are geometric generic and special points on $T$ respectively, satisfying the condition that the following diagrams commute:
        \[\begin{tikzcd}
	{t'} & T & {s'} & T \\
	t & S, & s & S.
	\arrow[from=1-1, to=2-1]
	\arrow[from=1-1, to=1-2]
	\arrow[from=1-2, to=2-2]
	\arrow[from=2-1, to=2-2]
	\arrow[from=1-3, to=2-3]
	\arrow[from=1-3, to=1-4]
	\arrow[from=1-4, to=2-4]
	\arrow[from=2-3, to=2-4]
    \end{tikzcd}\]
    We abbreviate the tuple $(\gamma, t' \to t, s' \to s,  t \rightsquigarrow s)$ as $\gamma : t \rightsquigarrow s$. 
    Note that a path induces an \'etale path from $t$ to $s$ in the sense of \cite[\href{https://stacks.math.columbia.edu/tag/03VD}{Tag 03VD}]{stacks-project}, which we shall denote by the same letter. 
\end{definition}
We recall some basic facts about the N\'eron--Severi group which we may often use implicitly. 

\begin{theorem}
\label{thm: specialization of line bundle}
    Suppose that $X \to S$ is a smooth and proper morphism between algebraic spaces with irreducible geometric fibers. Then we have the following. 
    \begin{enumerate}[label=\upshape{(\alph*)}]
        \item The relative Picard functor $\Pic_{X/S}$ is representable by a separated algebraic space locally of finite type over $S$. 
        \item Each finite type closed subspace of $\Pic_{X/S}$ is proper over $S$. 
        \item If $S$ is the spectrum of a discrete valuation ring with the generic point $\eta$, then there is a natural isomorphism $\Pic(X_\eta) \simeq \Pic(X)$. 
        \item If $S = \Spec(k)$ and $k \subseteq k'$ is an extension of separably closed fields, then the base change map $\NS(X_k) \to \NS(X_{k'})$ is an isomorphism. 
        \item Every path $\gamma : t \rightsquigarrow s$ on $S$ naturally determines a specialization map 
        \[
        \mathrm{sp}_\gamma  : \NS(X_t) \to \NS(X_s). 
        \] 
        which is injective after tensoring with $\IQ$.
    \end{enumerate}
\end{theorem}
\begin{proof}
    Part (a) and Part (b) are stated in \cite[\S8.3, Thm.~3]{BLR}, whose proof implies Part (c) (cf. \cite[Eqn.(3.4)]{MPJumping}):
    The idea is that a Weil divisor on $X_\eta$ gives rise to a Weil divisor on $X$ by taking the Zariski closure, and any line bundle on $X$ is trivial if its restriction on $X_\eta$ is trivial. 
    Part (d) follows from Part (a) because the N\'eron--Severi group $\NS(X_k)$ is identified with the set of connected components of $\Pic_{X/k}$, and hence does not increase under a separably closed extension (cf. \cite[Prop.~3.1]{MPJumping}). 
    Part (e) follows from Part (c) and Part (d). 
    Note that unlike Prop.~3.3 of \textit{loc. cit.}, our specialization map is canonically defined, because by \Cref{def: path} a path is defined over a strictly Henselian ring. 
    For the injectivity statement in (e), see \cite[Prop.~3.6]{MPJumping}.
\end{proof}

We remark that the specialization maps also exist for endomorphisms of abelian schemes. In this context, Part (c) of \Cref{thm: specialization of line bundle} is a consequence of the N\'eron extension property. 

Recall that simultaneous resolutions of a family of surfaces are often non-unique, even over a discrete valuation ring (cf. page 1 of \cite{Artin-Res}). However, the specializations of N\'eron--Severi classes or cohomology groups do not depend on the resolution, due to the following observation.
\begin{lemma}
\label{lem: specialization resolution}
    Let $S \colonequals \Spec(R)$ for a strictly Henselian discrete valuation ring $R$, with $s$ the special point, $\eta$ the generic point, and $\bar{\eta}$ a geometric point over $\eta$. 
    Let $X$ be a relative surface over $S$ with normal geometric fibers such that $X_\eta$ is smooth and $X_s$ is not ruled. Suppose that $\pi : Y \to X$ and $\pi' : Y' \to X$ are two simultaneous resolutions. Let $f_* : Y_* \sto Y'_*$ be the unique isomorphism such that $ \pi'_* \circ f_* = \pi_{*}$ for $* = s, \eta$. Then the following diagram commutes: 
    \begin{equation}
    \label{diag: specialization of graph}
        \begin{tikzcd}
	{\H^2_\et(Y_{\bar{\eta}}, \IQ_\ell)} & {\H^2_\et(Y_s, \IQ_\ell)} \\
	{\H^2_\et(Y'_{\bar{\eta}}, \IQ_\ell)} & {\H^2_\et(Y'_s, \IQ_\ell)},
	\arrow["\sim", from=1-1, to=1-2]
	\arrow["{f_{\eta}}"', from=1-1, to=2-1]
	\arrow["{f_s}", from=1-2, to=2-2]
	\arrow["\sim", from=2-1, to=2-2]
    \end{tikzcd}
    \end{equation}
    where the horizontal isomorphisms are given by the smooth and proper base change theorem. \footnote{Or in other words, the isomorphisms  are given by the \'etale path $\bar{\eta} \to s$ induced by $T$.} 
\end{lemma}
We remark that the map $f_{\eta}$ exists simply because $\pi_{\eta}$ and $\pi'_{\eta}$ are isomorphisms.
The map $f_s$ exists because $\pi_s$ and $\pi'_s$ are birational, and $Y_s$ and $Y'_s$ are both the minimal  resolution of $X_s$, so that any birational map $Y_s \dashrightarrow Y'_s$ extends to an isomorphism. 
\begin{proof}
    We recall the basic fact of \'etale cohomology that both of the vertical arrows are induced from the cycle classes of the correspondences (cf. \cite[Lem.\ 25.3]{milneLEC}). 
    Let $\Gamma_* \subseteq Y_* \times_* Y'_*$ be the graph of $f_*$ for $* = s, \eta$, which we shall view as a correspondence between $Y_*$ and $Y'_*$. Let $\Gamma_{\eta, s}$ denote the specialization of $\Gamma_\eta$ (i.e., the special fiber of the Zariski closure of $\Gamma_\eta$ in $Y \times_S Y'$). 
    Then the diagram (\ref{diag: specialization of graph}) commutes if the map $\H^2_\et(Y_s, \IQ_\ell) \sto \H^2_\et(Y_s, \IQ_\ell)$ is induced by the correspondence $\Gamma_{\eta, s}$ because cycle class maps are compatible with specialization. 
    However, by \cite[Thm~1]{MM64} we know that $\Gamma_{\eta, s} = \Gamma_s + \Gamma_s'$, where the projections of $\Gamma_s'$ to $Y_s'$ and $Y_s$ are both $0$ (as an algebraic cycle). 
    Therefore, $\Gamma_{\eta, s}$ and $\Gamma_s$ induce the same map $\H^2_\et(\sY_s, \IQ_\ell) \sto \H^2_\et(\sY_s, \IQ_\ell)$. 
    Technically, \cite[Thm~1]{MM64} is phrased in terms of schemes, but it does hold for algebraic spaces in the generality which is applicable in the current context, see \cite[Thm~A.2]{FLTZ} (see also \cite[Thm~5.4]{Matsumoto}).
\end{proof}

At the end of the subsection, we record an extension result for the cohomology sheaves of a family of surfaces, which is often used together with the above observations.
\begin{lemma}
\label{lem: coh of Zariski-resolvable}
    Let $S$ be a normal integral Noetherian scheme, and $f : X \to S$ be a proper relative surface whose geometric fibers are normal and nonruled, such that the generic fiber $X_\eta$ is smooth. 
    Let $S^\circ \subseteq S$ be the maximal open subscheme such that $f^\circ \colonequals f|_{S^\circ}$ is smooth. 
    Assume that $X \to S$ is Zariski-locally resolvable.

    Then for each $i \in \IN$ and each prime $\ell\in \sO_S^\times$, the $\ell$-adic sheaf $R^i f^\circ_{\et *} \IZ_\ell$ over $S^\circ$ extends uniquely to an $\ell$-adic sheaf $\breve{R}^i f_{\et *} \IZ_\ell$ over $S$.
    Moreover, the extension has the following properties: 
    \begin{enumerate}[label=\upshape{(\alph*)}]
        \item For every open $U \subseteq S$ and a simultaneous resolution $g : Y \to U$ of $f|_U$, the identification of $\breve{R}^i f_{\et *} \IZ_\ell$ and $R^i g_{\et *} \IZ_\ell$ over $U \cap S^\circ$ extends to $U$. 
        \item For every geometric point $\bar{s}$ over a point $s \in S$, there is a canonical $\Gal_{k(s)}$-equivariant isomorphism $(\breve{R}^i f_{\et *} \IZ_\ell)_{\bar{s}} \simeq \H^i_\et(\breve{X}_{\bar{s}}, \IZ_\ell)$, where $\breve{X}_{\overline{s}}$ is the minimal resolution of $X_{\overline{s}}$. 
    \end{enumerate}
\end{lemma}
\begin{proof}
    Recall that for a normal integral Noetherian scheme $T$ with open dense $V \subseteq T$, the restriction functor bewteen the categories of finite \'etale covers $\textsf{F{\'E}t}(T) \to \textsf{F{\'E}t}(V)$ is fully faithful (\cite[\href{https://stacks.math.columbia.edu/tag/0BQI}{Tag 0BQI}, \href{https://stacks.math.columbia.edu/tag/0BQ3}{Tag 0BQ3}]{stacks-project}) and the natural map $\pi_1(V)\to \pi_1(T)$ is surjective (\cite[\href{https://stacks.math.columbia.edu/tag/0BN6}{Tag 0BN6}]{stacks-project}) for any choice of base points. 
    In particular, To construct $\breve{R}^i f_{\et *} \IZ_\ell$, we find an open cover $\{ U_\alpha \}$ of $S$ such that on each $U_\alpha$, $X$ admits a simultaneous resolution $g_\alpha : Y_\alpha \to U_\alpha$. 
    As $Y_\alpha |_{U_\alpha \cap S^\circ} = X|_{U_\alpha \cap S^\circ}$, we have identifications bewteen $R^i g_{\alpha, \et *} \IZ_\ell$ and $R^i g_{\beta, \et*} \IZ_\ell$ over $U_\alpha \cap U_\beta \cap S$, and hence over $U_\alpha \cap U_\beta$ by the full-faithfulness above. 
    Therefore, we can glue these sheaves. 
    This proves (a), and (b) additionally follows from \cref{lem: specialization resolution}. 
\end{proof}
  
\begin{construction}
\label{const: VHS and crystals resolved}
    The full-faithfulness used above is also true for $\IZ$-VHS (resp. F-isoscrystals) when $S$ is a smooth variety over $\IC$ (resp. over a perfect field $k$ of characteristic $p > 0$), as in \cite[Cor.~12]{Peters} (resp.  \cite[Thm.~2.2.3]{DK17}). 
    Therefore, we likewise define $\IZ$-VHS $(\breve{R}^i f_{*} \IZ, \breve{R}^i f_* \Omega^\bullet_{X/S})$ (resp. F-isocrystal $\breve{R}^i f_{\crys *} \sO_{X/W}[1/p]$) to be the one satisfying the analogous properties of \cref{lem: coh of Zariski-resolvable}. 
    Moreover, in the latter situation, if $\breve{R}^i f_{\crys *} \sO_{X/W}$ is locally free, then we apply \cref{lem: rigidity of F-crystals} to define an F-crystal $\breve{R}^i f_{\crys *} \sO_{X/W}$ over $S$. 
\end{construction}

\subsection{A Beauville--Laszlo gluing lemma for projective schemes}

In the following, for a given scheme $S$, we let $\sC_S$ be the category of pairs $(Z,\sL)$, where $Z$ is a projective scheme over $S$ and $\sL$ is a relative ample line bundle for $Z/S$.
We also let $\Proj_S$ be the category of projective $S$-schemes.
Then we have a variant of the Beauville--Laszlo gluing but for projective schemes.
\begin{proposition}
	\label{BL gluing}
	Let $S$ be a smooth variety over $k$, let $s\in S$ be a closed point, and let $U$ be its complement open subvariety $S\backslash \{s\}$,
	We temporarily denote $S_s$ to be the affine scheme $\Spec(\sO^\wedge_{S,s})$, and let $W\colonequals U\times_S S_s$ be the fiber product.
	\begin{enumerate}[label=\upshape{(\roman*)},leftmargin=*]
		\item There is a natural functor of categories 
		\[
		\sC_U \times_{\sC_W} \sC_{S_s} \longrightarrow  \Proj_S,
		\]
		whose composition with the pullback functor $\Proj_S \to \Proj_U\times_{\Proj_W} \Proj_{S_s}$ is identified with the forgetful functor.
		\item Assume $X\to S$ is a qcqs algebraic space such that there are ample line bundles $\sL_U$ and $\sL_s$ over $X_U$ and $X_s$ respectively that coincide after base changing to $X_W$.
		Then the map $X\to S$ is projective, and in particular $X$ is an algebraic scheme.
	\end{enumerate}
\end{proposition}

\begin{proof}
	As the statement is Zariski local with respect to $S$, we may and do assume that $S$ is an affine scheme.
	For (i), let $\bigl( (X_U, \sL_U), (X_s,\sL_s), \alpha_W: (X_U,\sL_U)|_W \simeq (X_s, \sL_s)|_W)$ be an object in the fiber product category $\sC_U \times_{\sC_W} \sC_{S_s}$. 
	We first show that there is a projective scheme $\mathbb{P}$ over $S$ whose base changes to $U$, $S_s$ and $W$ admit compatible closed immersions from $X_U, X_s$ and $X_W$ respectively. 
	We let $\sA_U$ and $\sA_s$ be the graded rings $\oplus_{i\in \mathbb{N}} \H^0(X_U, \sL_U^i)$ and $\oplus_{i\in \mathbb{N}} \H^0(X_s, \sL_s^i)$ respectively.
	Then we get the following closed immersions into projective spaces
	\[
	X_U \to \Proj( \sA_U), \quad X_s \to \Proj( \sA_s).
	\]
	On the other hand, by assumption the base changes at $X_W$ satisfy $\alpha_W: (\sL_U)|_W \simeq (\sL_s)|_W$.
	Moreover, since both $S_s$ and $U$ are flat over $S$, by flat base change for coherent cohomology (\cite[\href{https://stacks.math.columbia.edu/tag/02KH}{Tag 02KH}]{stacks-project}), we get
	\[
	\H^0(X_U, \sL_U^i) \otimes_{\sO_U} \sO_{W} \simeq \H^0((X_U)|_W, (\sL_U)|_W^i), \quad \H^0(X_s, \sL_s^i) \otimes_{\sO_{S_s}} \sO_{W} \simeq \H^0((X_s)|_W, (\sL_s)|_W^i).
	\]
	As a consequence, the closed immersions $X_U \to \Proj( \sA_U)$ and  $X_s \to \Proj( \sA_s)$ are identified via $\alpha$ after base changing to $W$, and we get a commutative diagram of closed immersions
	\[
	\begin{tikzcd}
		\Proj(\sA_U) && \Proj(\sA_U\otimes_{\sO_U} \sO_W) \simeq \Proj(\sA_s\otimes_{\sO_{S_s}} \sO_W) \ar[ll] \ar[dd]\\
		X_U \ar[u] & X_W \ar[l] \ar[d] \ar[ru] &\\
		& X_s \ar[r] & \Proj(\sA_s),
	\end{tikzcd}
	\]
	where both squares are cartesian.
	
	Then we recall from \cite[Prop.~5.6.(4)]{Bha16} that the base change functors induce a symmetric monoidal equivalence of the categories
	\begin{equation}\label{BL gluing formula}
			\mathrm{Qcoh}_S \simeq \mathrm{Qcoh}_U \times_{\mathrm{Qcoh}_W} \mathrm{Qcoh}_{S_s}, 
	\end{equation}
	which restricts to an equivalence of coherent sheaves by \cite[Prop.~5.6.(2)]{Bha16} and the regularity of $S$.
	Apply this at the graded quasi-coherent algebras $\bigl(\sA_U, \sA_s, \alpha :(\sA_U)\otimes_{\sO_U} \sO_W \simeq (\sA_s)\otimes_{\sO_{S_s}} \sO_W \bigr)$,
	we see there is a graded quasi-coherent algebra $\sA$ over $S$ that specializes to $\sA_U$ and $\sA_s$ and is compatible with transition isomorphism $\alpha$.
	Moreover, as both $\sA_U$ and $\sA_s$ is generated by first finite amount of graded components, so is $\sA$.
	Furthermore, notice that since each $\sA_{U,i}$ and $\sA_{s,i}$ are coherent, we see each $\sA_i$ is coherent over $\sO_S$ as well.
	As a consequence, $\sA$ is a finitely generated graded $\sO_S$-algebra,
	and we can take $\mathbb{P}$ to be the projective scheme $\Proj(\sA)$ over $S$.
	Then the base changes of $\mathbb{P}$ to $U$ and $S_s$ coincide with projective spaces $\Proj(\sA_U)$ and $\Proj(\sA_s)$, which are compatible with the transition isomorphism $\alpha$.
	For simplicity, we will use $\mathbb{P}_U$, $\mathbb{P}_s$ and $\mathbb{P}_W$ to denote $\Proj(\sA_U)$, $\Proj(\sA_s)$ and $\Proj(\sA_U\otimes_{\sO_U} \sO_W) \simeq \Proj(\sA_s\otimes_{\sO_{S_s}} \sO_W)$ from now.
	
	Now we proceed to produce a projective scheme $X$ over $S$ as a closed subscheme of $\mathbb{P}$, which specializes to the given data at the beginning.
	Let $V\subset \mathbb{P}$ be an affine open subscheme, and let $\mathbb{V}_\ast$ for $\ast\in \{U,s, W\}$ be its base changes,
	which are affine open inside $\mathbb{P}_\ast$ by the affineness of the morphisms $U\to S$, $S_s\to S$ and $W\to S$.
	Moreover, the fiber products $V_\ast\times_{\mathbb{P}_\ast} X_\ast$ are affine open subspaces of $X_\ast$ for $\ast\in \{U,s, W\}$.
	So by Beauville--Laszlo gluing (\ref{BL gluing formula}) above, we can produce an affine closed scheme $T_V\subset V$ that specializes to $V_\ast\times_{\mathbb{P}_\ast} X_\ast$ for $\ast\in \{U,s, W\}$ under natural base changes.
	Finally, by \cite[Prop.~5.6.(6)]{Bha16}, we know the equivalence (\ref{BL gluing formula}) is compatible with flat base changes of $V$ (in particular open immersions).
	So by ranging over all affine open $V\subset \mathbb{P}$, we can glue $T_V$ to a closed scheme $X$ inside $\mathbb{P}$, which specializes to $X_U$, $X_s$ and $X_W$ under base changes.
	
	To prove (ii), we first notice that by (i), we can find a projective scheme $X'$ over $S$ such that its image under the functor
	\[
	\Proj_S \longrightarrow \Proj_U \times_{\Proj_W} \Proj_{S_s}
	\]
	is $(X_U, X_s, \alpha:(X_U)|_W \simeq (X_s)|_W)$.
	So it is left to show that as qcqs algebraic spaces, $X$ and $X'$ are isomorphic.
	By construction of $X'$, we know it satisfies the formula that $X'\times_S \ast \simeq X_\ast$, which also holds true for $X$ by assumption.
	In particular, the assumption of \cite[Prop.~5.6]{Bha16} applies to both $X$ and $X'$.
	As a consequence, by \cite[Prop.~5.6.(6)]{Bha16}, both $X$ and $X'$ are isomorphic to the pushout of the diagram below in the category of qcqs algebraic spaces
	\[
	\begin{tikzcd}
		X_U & X_W \ar[l] \ar[d] \\
		& X_s.
	\end{tikzcd}
    \]
    Hence by the uniqueness of the pushout, we see $X\simeq X'$.
\end{proof}

\begin{remark}
    Recently, Achinger and Youcis proved that a Beauville-Laszlo gluing datum is always effective in the category of algebraic spaces (\cite{AY24}). As illustrated by Matsumoto, such a statement is false for projective schemes because there may not be a pair of ample line bundles that glue (\cite[Ex.~5.2]{Matsumoto}, cf. \cite[Ex.~3.4]{AY24}).
    Our \cref{BL gluing}.(ii) above is saying that this is the only obstruction for the statement to be true for projective schemes. 
\end{remark}

\subsection{An application of MMP for threefolds}
\label{sub:MMP}

\begin{theorem}
\label{prop: connected by flops}
    Let $k$ be an algebraically closed field of characteristic $0$ or $p \ge 5$.
    Suppose that $X$ and $X'$ are smooth $k$-threefolds and $f : X \to S$ and $f' : X' \to S$ are projective morphisms between smooth quasi-projective varieties such that $K_X$ and $K_{X'}$ are both relatively nef over $S$. Then every birational map $\alpha : X \dashrightarrow X'$ can be decomposed into a sequence of flops. 
\end{theorem}
\begin{proof}
	When $k$ is of characteristic $0$, this is a special case of \cite[Thm.~1]{Kawamata} when we are simply taking the $\IQ$-divisors $B$ and $B'$ therein to be $0$. 
	The proof only relies on the existence of minimal model program (MMP), including cone theorem, the based point free theorem, and termination of MMP with scaling, which all hold true for threefolds in characteristic $p \ge 5$  by \cite{HX-MMP, HW-MMP}.
	For the reader's convenience, we briefly sketch the proof of \cite{Kawamata} below for general $B$ and $B'$, explaining the necessary modifications in characteristic $p$.

	Without any change, the first two paragraphs of the proof of \cite[Thm.~1]{Kawamata} shows that the map $\alpha$ is an isomorphism in codimension one, and there are only flips in this MMP, since the other operation (which is a divisor contraction) changes the picard group (\cite[\S 3.31]{KM98}).
	In particular, for a Cartier divisor $L$ on $X$ with its strict transform $L'$ on $X'$, we have $\mathrm{H}^i(X,\sO(L))\simeq \mathrm{H}^i(X',\sO(L'))$, and $L$ is big if and only if $L'$ is so (\cite[Def.~2.59]{KM98}).
	We then choose $L'$ an effective $f'$-ample divisor on $X'$, with $L$ its strict transform on $X$ and $l\ll 1$ a positive rational number such that $(X,B+lL)$ is klt (which exists since being klt is an open condition).
	In this case, the divisor $L$ is big and can be written as the sum $A+E$, where $A$ is $f$-ample and $E$ is $f$-nef.
	Moreover, for a small enough $\epsilon >0$, by the openness of the klt condition again, the pair $(X, B+ (l-\epsilon) L +\epsilon E)$ is klt as well.
	In particular, the base point free theorem in characteristic $p$ \cite[Cor.~1.5]{ABL22} (applied at the log pair $(X, B+(l-\epsilon)L+\epsilon E)$) implies that the sum $m(K_X+B+lL)$ is $f$-free for $m\in \mathbb{N}$ large enough.
	As a consequence, the global sections of a positive multiple of the $\mathbb{Q}$-divisor $K_X+B+lL$ defines a projective $S$-morphism and coincides with that of the $f'$-ample $\mathbb{Q}$-divisor $K_{X'}+B'+lL'$.
 The latter implies that the map $\alpha$ is an isomorphism.
	Hence we may assume $K_X+B+l'L$ is not $f$-nef over $S$ for $0<l' \leq l$.
  Furthermore, since $X$ is $\mathbb{Q}$-factorial and $B$ is a $\mathbb{Q}$-Cartier divisor, by \cite[Def.~6.10]{KM98} we may assume $l$ is small enough so that a flop for $(X,B)$ is the same as a flip for $(X,B+l'L)$.

	To proceed, we let $H$ be an effective divisor on $X$ such that $(X,B+lL+tH)$ is klt and $K_X+B+lL+tH$ is $f$-nef for some $t>0$.

	our goal is to run the minimal model program for the pair $(X,B+l'L)$ with the scaling $H$, so that the index $l'$ is decreasing each time by a ratio $e$ that only depends on $X$ and $B$.
	To this end, we let $$e=\frac{1}{2m \dim(X)+1}$$ for a positive integer $m$ with $m(K_X+B)$ being a Cartier divisor.
	Then \cite[Lem.~2]{Kawamata} shows that there exists an extremal ray $R$ for $(X,B+elL)$ over $S$ such that $K_X+B+elL+et_0H$ is $f$-nef, and
	\[
	((K_X+B+elL+et_0H)\cdot R) = ((K_X+B)\cdot R) =0,
	\]
	where 
	\begin{align*}
	     t_0=\min\{t\in \mathbb{R} \mid &((K_X+B+lL+tH)\cdot R)\geq 0 \text{ for all extremal rays } R \text{ for } \\ &(X,B+lL) \text{ over } S \text{ such that } ((K_X+B)\cdot R)=0\}.
	\end{align*}
    
	Here we note that the only input of the above existence in \textit{loc.\ cit.} is the inequality $0>((K_X+B+lL)\cdot C) \geq -6$, which is proved in \cite[Thm.~1.4]{HW-MMP}.

	Finally, we may run MMP for the pair $(X,B+elL)$ with scaling of $H$.
	By \cite[Thm.~1.1]{HW-MMP}, the flip for the contraction morphism $\pi: X\to Z$ with respect to the extremal ray $R$ exists.
	Moreover, by the equalities above, since  $(K_X+B+lL+t_0H)\cdot R=0$, the $\mathbb{Q}$-Cartier divisor $(K_X+B+lL+t_0H)$ and its strict transform on the flip are relatively numerically trivial with respect $Z$ (and thus are anti-nef and nef with respect to $Z$ separately). 
	So thanks to the negativity lemma (\cite[Lem.~3.38]{KM98}), we know  the pair $(X,B+lL+t_0H)$ remains klt after the flip.
	Furthermore, the equalities above imply that the $\mathbb{Q}$-Cartier divisor $k(K_X+B)-k(K_X+B+elL)=(-elkL)$ is $\pi$-big.
	We can thus apply the base point free theorem  \cite[Cor.~1.5]{ABL22} at the $\pi$-nef Cartier divisor $K_X+B$ and the log pair $(X, B+elL)$, to see that there exists a large $m\in \mathbb{N}$ such that both $mk(K_X+B)$ and $(m+1)k(K_X+B)$ (and so is their difference $k(K_X+B)$) lift to ample Cartier divisors of some projective morphism $Y\to Z$.
	But notice that since the map $X\to Z$ is defined by contracting the extremal ray $R$, the equation $k(K_X+B)\cdot R=0$ enforces that $Y=Z$ and $k(K_X+B)$ is defined over $Z$.
	Hence the divisor $k(K_X+B)$ remains a Cartier divisor after the flip.
	In this way, we can continue the process for $(X,B+elL)$ with the scaling $H$, which terminates after finite amount of steps by \cite[Thm.~1.2]{HW-MMP}.
	So we are done.\end{proof}

\begin{corollary}[Resolution of a constant family is constant]
	\label{cor: res of constant family}
	Let $k$ be an algebraically closed field of characteristic $p \ge 5$. 
	Let $X_0$ be a projective surface over $k$ with only RDP singularities and $X = X_0 \times C$ for some smooth curve $C$ over $k$. 
	Let $Y_0$ be the minimal resolution of $X_0$, and let $(Y/C, \pi : Y \to X)$ be any minimal simultaneous resolution of $X/C$. 
	Then $Y/C$ is isomorphic to the constant family $Y_0 \times C$.
\end{corollary}
\begin{proof}
    Let us first prove the statement assuming that the total space $Y$ is a quasi-projective over $k$. 
	We first notice that the assumption that $X_0$ only has RDP singularities implies that $K_Y$ is $\pi$-nef. 
	To see this, it suffices to show that for any $\pi$-exceptional curve $D$, the intersection $K_Y\cdot D$ is non-negative. 
	As $\pi$ is fiberwise the minimal resolution, $D$ lies in a single fiber, i.e., $D \subseteq Y_c$ for some $c \in C(k)$. 
	Moreover, by the RDP assumption, the curve $D$ is a rational $(-2)$-curve on $Y_c$ and in particular we have $D^2=-2$.  
	So by adjunction formula of surfaces, we get $(K_Y+D)\cdot D=\deg (K_D) = -2$, and thus $K_Y\cdot D=0$.
	
	Note that as both $Y$ and $Y_0 \times C$ are minimal resolutions of $X$, they are birational equivalent and are relative nef over $C$. 
	By the algebraicity assumption of $Y$ and by \Cref{prop: connected by flops}, any birational morphism of varieties $Y_0 \times C \dashrightarrow Y$ decomposes into a sequence of flops. 
	However, there are no flops that can be defined on $Y_0 \times C$. 
	Indeed, in order to do so one must choose a numerical class of relative exceptional curves (namely a class of those curves $D\subset Y_0\times C$ that are mapped onto single points in $C$. 
	But since $Y_0 \times C$ is a constant family, every exceptional curve $D_c$ is contained in a single fiber $Y_0\times c$ and in particular is algebraic equivalent to a curve $D_{c'}\subset Y_0\times c'$, where $c'\subset C$ is any point that is in the same connected component of $c$.
	As the consequence, the contraction of the relative numerical class $[D_c]$ will simultaneously contracting a codimension $1$ locus in $Y_0\times C$, which in particular would not be a flop (\cite[Def.6.10]{KM98}).

    Now we consider the general case that $Y$ is an algebraic space, and we want to prove that $Y$ is in fact a scheme.
    By \cref{lem: spreading out} below, there exists an open dense subscheme $U \subseteq C$ such that $Y|_U$ is a scheme, so that by the previous paragraphs we know $Y|_U \simeq Y_0 \times U$. 
    By induction we may assume that $C \smallsetminus U$ is a single point $c$. 
    Let $C_c$ (resp. $\what{C}_c$) be the spectrum of the local ring $\sO_{C, c}$ of $C$ at $c$ (resp. formal completion of $C$ at $c$), and similarly let  $Y|_{C_c}$ (resp. $Y_{\what{C}_c}$ be the base change of $Y$ at $C_c$ (resp. formal completion of $Y$ at $c$).
    Then $Y|_{C_c}$ (resp. $Y|_{\what{C}_c}$) is a smooth proper algebraic space over $C_c$ (resp. formal scheme over $\what{C}_c$), and there is a natural map $\Pic(Y_{C_c}) \to \Pic(Y_{\what{C}_c})$. 
    Let $\eta$ be the generic point of $C$ and thus for $C_c$. 
    There is then a specialization map $\Pic(Y_\eta) \to \Pic(Y|_{C_c}) \to \Pic(Y_c)$ (\cref{thm: specialization of line bundle}). 
    Moreover, by the assumption that $Y|_U = Y_0 \times U$, we know $Y_\eta = Y_0 \tensor_k k(\eta)$, and we can form a diagram
    \begin{equation}
        \label{eqn: triangle of specialization}
        \begin{tikzcd}
	{\Pic(Y_0) } \\
	{\Pic(Y_\eta)} & {\Pic(Y_c),}
	\arrow["{\tensor_k k(\eta)}"', from=1-1, to=2-1]
	\arrow["{\textrm{specialize}}", from=2-1, to=2-2]
	\arrow["{\textrm{identify }Y_0 \simeq Y_c}"', "\simeq", from=2-2, to=1-1]
    \end{tikzcd}
    \end{equation}
    which a priori may not commute.
    However, we can tell that the injectivity the morphism $\NS(Y_\eta)_\IQ \to \NS(Y_c)_\IQ$ (which is true thanks to \Cref{thm: specialization of line bundle}.(e)) implies its surjectivity, as the source and the target of the map are at least abstractly isomorphic, thanks to \Cref{thm: specialization of line bundle}.(d).
    So as the ampleness is a numerical condition, the latter implies that there exists some $\zeta \in \Pic(Y_\eta)$ which specializes to an ample line bundle $\zeta \in Y_c$. 
    Furthermore, since the effectivity and the intersection numbers of line bundles on $Y_\eta$ are preserved under specialization, by the Nakai--Moishezon criterion, we know $\zeta$ must also be ample.
    Therefore, we may now apply \cref{BL gluing} to conclude that $Y$ is projective over $C$ and hence is trivial by the first two paragraphs above. 
\end{proof} 

\begin{lemma}
\label{lem: spreading out}
    Let $S$ be an integral scheme with generic point $\eta$ and let $f : X \to S$ be a smooth and proper morphism in the category of algebraic spaces. 
    If the generic fiber $X_\eta$ is a (projective) scheme, then up to shrinking $S$ (i.e., replacing $S$ by an open dense subscheme $U \subseteq S$) such that $X$ is also a (projective) scheme. 
\end{lemma}
\begin{proof}
    By spreading out the scheme $X_\eta$, and by shrinking $S$ if necessary, we can find a smooth and proper scheme $X'$ over $S$ such that there is an isomorphism $\psi_\eta : X'_\eta \sto X_\eta$. 
    It then suffices to show that up to shrinking $S$ again, the isomorphism $\psi_\eta$ can be extended to an isomorphism $\psi : X' \sto X$ in the category of algebraic spaces over $S$. Then by \cite[Thm~1.1]{Olsson} the fibered categories $\underline{\Hom}_S(X', X)$, $\underline{\Hom}_S(X, X')$ and $\underline{\Hom}_S(X, X)$ are in fact algebraic spaces locally of finite presentation over $S$. Now consider the fiber $Z$ of the morphism given by composition 
    $$ \underline{\Hom}_S(X', X) \times \underline{\Hom}_S(X, X') \to \underline{\Hom}_S(X, X) $$
    over $\mathrm{id}_X \in \underline{\Hom}_S(X, X)$. Then $Z$ is also an algebraic space locally of finite presentation and the pair $(\psi_\eta, \psi^{-1}_\eta)$ defines a point in $Z(k(\eta))$. Now by \cite[\href{https://stacks.math.columbia.edu/tag/05N0}{Tag 05N0}]{stacks-project}, up to shrinking $S$, we may extend this $k(\eta)$-point to an $S$-point, which justifies that we can spread out $\psi_\eta$ to an isomorphism $\psi : X' \sto X$. 
\end{proof}

\begin{remark}
\label{rmk:short_proof_on_constancy}
In \cref{cor: res of constant family}, as the minimal resolutions of the geometric fibers are unique, it is easy to see that $Y/C$ is isotrivial (i.e., has isomorphic geometric fibers over the same residue field). 
The content here is that $Y/C$ must in fact be trivial on the nose.
Below we sketch another proof of \cref{cor: res of constant family} using Artin's representability statement \cite[Thm.~1]{Artin-Res}, where assuming $p \ge 5$ is not needed.

Consider the resolution functor $R \colonequals \mathrm{Res}_{X_0/k}$ defined in page $1$ of \textit{loc. cit.}. 
Any simultaneous resolution of $X_0 \times C$ would define a morphism $C \to R$. 
Now, by \cite[Thm.~1]{Artin-Res}, we know $R$ is represented by a \textit{quasi-separated} algebraic space over $k$ and has a single geometric point.
On the other hand, a theorem of Rydh implies that algebraic space with a single point is a scheme if and only if it is decent \cite[\href{https://stacks.math.columbia.edu/tag/047Z}{Tag 047Z}, \href{https://stacks.math.columbia.edu/tag/086U}{086U}]{stacks-project}. 
As a consequence, since quasi-separated algebraic spaces are indeed decent \cite[\href{https://stacks.math.columbia.edu/tag/03I7}{Tag 03I7}]{stacks-project}, we know $R$ is a scheme such that $R_{\mathrm{red}} = \Spec(k)$. 
As $C$ is reduced, $C \to R$ factors through $R_{\mathrm{red}}$ and the result follows. 
\end{remark}

\subsection{Hodge diamond of elliptic surfaces in positive characteristic}

Let $C$ be a smooth projective curve over an algebraically closed field $k$. Let $X$ be a smooth surface over $k$ equipped with an elliptic fibration $\pi : X \to C$. 
Assume that there exists a zero section $\sigma : C \to X$. 
In particular, $X$ has no multiple fibers. 
Let $L := (R^1 \pi_* \sO_X)^\vee$ be the fundamental line bundle (cf. \cite[II.4]{Miranda}). 
Set $h := \mathrm{deg\,}L$.

\begin{proposition} \label{prop: Liu's lemma} If the $j$-invariant map $C \to \IP^1$ is finite and separable, then $\Ohm_{C/k} \iso \pi_* \Ohm_{X/ k}$. 
\end{proposition}
In fact, here it is unnecessary for $C$ to be projective---it works for an affine curve as well. The argument below is essentially an expansion of Liu's answer to a MathOverflow question (\cite{LiuAnswer}), which works verbatim in positive characteristics as well as long as the $j$-invariant map is separable.
\begin{proof} Consider the canonical short exact sequence 
    $$ 0 \to \pi^* \Omega_{C/k} \to \Omega_{X/k} \to \Omega_{X/C} \to 0,  $$
    which induces a long exact sequence 
    $$ 0 \to \Omega_{C/k} \to \pi_* \Ohm_{X/k} \to \pi_* \Ohm_{X/C} \stackrel{\theta}{\to} (R^1 \pi_* \sO_X) \tensor \Ohm_{C/k} \to \cdots. $$
    To see that $\pi^* \Omega_{C/k} \to \Omega_{X/k}$ is injective, note that over some open dense $U \subseteq C$, the map $\pi|_U$ is smooth, so that $(\pi|_U)^* \Omega_{U/k} = \pi^* \Omega_{C/k} |_{X_U} \to \Omega_{X_U/k}$ is smooth.     
    Let $\iota : U \into C$ and $\wt{\iota} : X|_U \into X$ be the natural inclusions. Then there is a commutative diagram
    \[\begin{tikzcd}
	{\pi^* \Omega_{C/k}} & {\Omega_{X/k}} \\
	{\pi^*(\iota_* \Omega_{U/k})} & {\wt{\iota}_* \Omega_{X_U/k}}
	\arrow[from=1-1, to=1-2]
	\arrow[hook, from=1-1, to=2-1]
	\arrow[hook, from=1-2, to=2-2]
	\arrow[hook, from=2-1, to=2-2]
    \end{tikzcd}.\]

    The morphism $\theta$ is generically injective, as the $j$-invariant map $C \to \IP^1$ is a finite and separable. Indeed, let $s \in C(k)$ be a general point, the fiber $\theta_s$ is nothing but the Kodaira--Spencer map 
    $$ \H^0(\Ohm_{X_s}) \to \H^1(\sO_{X_s}) \tensor (T_s C)^\vee.  $$
    
    To finish the proof, it suffices to show that $\pi_* \Ohm_{X/C}$ is torsion-free (or equivalently a vector bundle, since $C$ is a curve), because then $\ker(\theta)$, being a rank $0$ subsheaf of a torsion-free, has to vanish. 
    This then implies that $\Ohm_{C/k} \to \pi_* \Ohm_{X/k}$ is an isomorphism.   

    Write $\Omega$ for $\Omega_{X/C}$ and $\w$ for the relative dualizing sheaf $\w_{X/C}$. 
    Now consider the decomposition $0 \to \Omega_{\mathrm{tor}} \to \Omega \to \sF \to 0$ where $\sF$ is the torsion-free quotient. By the flat base change theorem and \cite[Prop.~1]{LiuSaito}, $\H^0(\Omega_{\mathrm{tor}}) = 0$. 
    Here we note that the assumption of \textit{loc. cit.} is satisfied thanks to the existence of the zero section: For any $s \in C(k)$, the fiber $X_s$ has at least an irreducible component of multiplicity $1$, i.e., the one which meets the zero section. 
    It remains to show that $\H^0(\sF)$ is torsion-free. Since $\Omega_{\mathrm{tor}} = \ker(\Omega \to \w)$ (see the second paragraph of \cite[\S 1]{LiuSaito}), $\sF$ is a subsheaf of $\w$, so that it suffices to observe that $\H^0(\w)$ is torsion-free: This follows easily from the facts that the formation of the relative dualizing sheaf commutes with arbitrary base change \cite[\href{https://stacks.math.columbia.edu/tag/0E6R}{Tag 0E6R}]{stacks-project}, and that for any point $s \in C$, $\dim \H^0(\w|_{X_s})$ is the arithmetic genus of $X_s$ \cite[\href{https://stacks.math.columbia.edu/tag/0BS2}{Tag 0BS2(5)}]{stacks-project}, which is constant by flatness.
\end{proof}
\begin{theorem}\label{thm: Hodge diamond of elliptic surface} If $h > 0$ and the $j$-invariant map $C \to \IP^1$ is finite and separable, then the following holds: 
    \begin{enumerate}[label=\upshape{(\alph*)}]
        \item The Hodge--de Rham spectral sequence of $X$ degenerates at $E_1$-page. 
        \item The Hodge diamond of $X$ is independent of $k$: $h^{0, 1} = h^{1, 0} = h^{1, 2} = h^{2, 1} = g$, $h^{0, 2} = h^{2, 0} = h - 1 + g$, and $h^{1,1} = 10h + 2g$ (cf. the case over $\IC$ \cite[Lem.~IV.1.1]{Miranda}). 
        \item If $\mathrm{char\,} k = p > 0$, then $\H^k_\crys(X/W(k))$ is torsion-free for every $k$. 
    \end{enumerate} 
\end{theorem}

\begin{proof} 
First of all, we have $h^{0, 1} = g(C)$. 
Indeed, the Leray spectral sequence gives us an exact sequence 
$$ 0 \to \H^1(C, \pi_* \sO_X) \to \H^1(X, \sO_X) \to \H^0(C, R^1 \pi_* \sO_X) \to 0. $$
As $h > 0$ (so that $R^1 \pi_* \sO_X = L^\vee$ has no global sections), and $\pi_* \sO_X \iso \sO_C$, we have $\H^1(X, \sO_X) \iso \H^1(C, \sO_C)$. 
The equality $h^{1,0} = g(C)$ follows from Prop.~\ref{prop: Liu's lemma}. 

By the universal coefficient theorem, for any $i$ we have 
$$ 0 \to \H^i_\crys(X/W) \tensor k \to \H^i_\dR(X/k) \to \mathrm{Tor}_1^W(\H^{i + 1}_\crys(X/W), k) \to 0. $$
And by the Hodge--de Rham spectral sequence we always have $h_\dR^i \le \sum_j h^{j, i - j}$, with equality taken only when all arrows in or out of the terms $\H^{i - j}(\Ohm_X^i)$ on the $E_1$-page vanish. 

Recall that $\H^1_\crys(X/W)$ is always torsion-free. In fact, this is true for any smooth proper variety over $k$, as the Albanese morphism $X \to \mathrm{Alb}(X)$ induces an isomorphism between their $\H^1_\crys(-/W)$ (cf. \cite[II, Remarque~3.11.2]{Illusie}), and $\H^1_\crys$ of an abelian variety is always torsion-free. 
By \cite[Prop.~2.1(1)]{Liedtke}, the Picard scheme of $X$ is smooth, so we have $\mathrm{rank\,} \H^1_\crys(X/W) = 2 h^{0, 1}$. This forces $h^1_\dR = h^{0, 1} + h^{1, 0}$, and hence $\H^2_\crys(X/W)$ is torsion-free. To see that $\H^3_\crys(X/W)$ is torsion-free, we first note that by Poincar\'e duality $\H^1_\crys(X/W)$ and $\H^3_\crys(X/W)$ have the same rank, and by Serre duality we have $h^{2, 1} = h^{0, 1}$ and $h^{1, 2} = h^{1, 0}$. Then we have $$2 h^{0, 1} = \mathrm{rank\,} \H^3_\crys(X/W) \le h^3_\dR \le h^{2, 1} + h^{1, 2}.$$
Hence the equality must be taken, which forces $\H^3_\crys(X/W)$ to be torsion-free. 
In particular, by the universal coefficient theorem, this forces the natural map $\H^2_\crys(X/W) \tensor k \to \H^2_\dR(X/k)$ to be an isomorphism.

At this point we can easily tell the Hodge--de Rham spectral sequence $\{E_n\}$ degenerates at $E_1$-page.
Note that we have argued that $h^{k}_\dR = \sum_{i + j = k} h^{i, j}$ for $k \neq 1$, and hence the terms $\H^{j}(\Omega_X^i) = E_\infty^{i, j}$ for $i + j \neq 1$. 
This forces the differentials involving the $E_n^{i, j}$ terms with $n \ge 1$ and $i + j = 1$ to vanish.
That $h^{0, 2} = h - 1 + g$ now follows from the definition of $\chi(\sO_X)$ and the isomorphism $\H^1(X,\sO_X)\simeq \H^1(C,\sO_C)$, and $h^{2, 0} = h^{0, 2}$ follows by Serre duality. 
The only undetermined Hodge number is $h^{1, 1}$. 
However, the crystalline Euler number (defined as the alternating sum of ranks) is the same as the $\ell$-adic Euler number, which is $12 h$.
Hence $h^{1,1}$ is $10h + 2g$. 
\end{proof}

\subsection{Inseparable $j$-invariant maps}
We analyze the $j$-invariant map of an elliptic surface in terms of its Weierstrass normal form, assuming that the base field is algebraically closed and has characteristic $p \ge 5$. 
We follow the notations as in last subsection.
Recall that the Weierstrass normal form of $X$ is given by the equation (cf. \cite[Thm.~1]{Kas})
\begin{equation}
    \label{eqn: Weierstrass normal form}
    Y^2 Z = X^3 - a_4 XZ^2 - a_6 Z^3, 
\end{equation}
where $[X: Y : Z]$ is the relative coordinate of the $\IP^2$-bundle $\IP(L^2 \oplus L^3 \oplus \sO_C)$ over the curve $C$, and $a_4, a_6$ are elements of $\H^0(C, L^4), \H^0(C, L^6)$ respectively. 
In terms of the elements $a_4$ and $a_6$, the $j$-invariant of the generic fiber of $\pi : X \to C$, as an element in the fraction field $k(C)$, is given by $1728 \cdot (4a_4^3 / \Delta)$, where $\Delta = 4a_4^3 - 27 a_6^2 \in \H^0(C, L^{12})$ is the discriminant. 
Note that in order for $\pi : X \to C$ to have a smooth generic fiber, $\Delta$ does not vanish identically on $C$; in particular, the element $a_4 \neq 0 \in \H^0(C, L^4)$. 
Notice that being an element of $k(C)$, the $j$-invariant can be simultaneously viewed as a map $C \to \IP^1$, which we referred to as the \emph{$j$-invariant map}. 
Denote the degree of this map by $\deg j$.

For \cref{prop: inspearable j} and \cref{cor: bound inseparable} below, assume that $p \ge 5$, and $g = h = 1$, where $g$ is the geometric genus and $h$ is the degree of the fundamental line bundle $L$. 
Let $e \in C$ be the unique point such that $L \simeq \sO_C(e)$. 
We shall view $C$ as an elliptic curve with origin at $e$.
Below for any nonzero global section of a line bundle on $C$, we view its divisor simultaneously as a multiset of closed points on $C$. 
Here, to avoid confusing the addition of points on the elliptic curve $C$ with the addition of divisors, we denote the former by the symbol $\oplus$ and the latter by the symbol $+$.
Note that a line bundle on $C$ of degree $0$ admits a nonzero global section if and only if it is trivial (i.e., $\simeq \sO_C$). 
Therefore, if $s \in \H^0(C, L^r)$ is any nonzero section, then $s$ can be viewed as as unique (up to scalar) section in $\H^0(C, L^r(-\mathrm{div}(s)))$, and we have the following induced formula for the summation of the points under the group law of the elliptic curve $C$:
\begin{equation}
    \label{eqn: sum to zero}
    \oplus_{x \in \supp(\mathrm{div}(s))} m_x(\div(s)) \cdot x = e.
\end{equation}
Here for a multiset $A$ in $C$, we use $m_x(A)$ to denote the multiplicity of a given point $x \in \supp(A)$, and in particular we have $\val_x(s) = m_x(\div(s))$. 

\begin{proposition}
\label{prop: inspearable j}
    If the $j$-invariant map is nonconstant and inseparable, then $p = 5$ or $7$. When $p = 5$, there are two possibilities: 
    \begin{enumerate}[label=\upshape{(\roman*)}]
        \item $\deg j = 10$, and there are points $y, z \neq x$ in $C$, such that $\mathrm{div}(a_4) = \{4x \}$, $\mathrm{val}_x(a_6) = 1$, and $\mathrm{div}(\Delta) = \{2x, 5y,  5z\}$. 
        \item $\deg j = 5$, and there are points $x \neq y$ in $C$, such that $\div(a_4) = \{ 3x, y \}$, and $\div(a_6) = \{ 2x, 4y \}$.  
    \end{enumerate}
    When $p = 7$, there is one possibility: 
    \begin{enumerate}
        \item[\upshape{(iii)}] $\deg j = 7$, and there are points $x \neq y$ in $C$, such that $\div(a_4) = \{ 3x, y \}$, and $\div(a_6) = \{ x, 5y \}$. 
    \end{enumerate}
\end{proposition}
\begin{proof}
    Note that the multiset of zeroes of $j$, which is a meromorphic function on $C$, is given by $3 \div(a_4) \smallsetminus 2\div(a_6)$. 
    Therefore, the order of this multiset is $\deg j$. 
    It is easy to see that under the assumption, we have $\deg j \le 12$, so $p = 5, 7$ or $11$. 

    For simplicity, let us write $A, B$ for the multisets $\div(a_4)$ and $\div(a_6)$ respectively. 
    In order for the $j$-invariant map to be inseparable, the following conditions needs to be satisfied:
    \begin{itemize}
        \item[(P)] The multiplicity of $j$ (as a meromorphic function) at any point on $C$ is divisible by $p$.
        \item[(P')] $p \mid m_x(3A \smallsetminus 2B)$ for any $x \in 3A \smallsetminus 2B$.
        \item[(P'')] $p \mid \deg j = \#(3A \smallsetminus 2 B) = 12 - \#(3A \cap 2B)$.  
    \end{itemize}
    Clearly, we have the implications $P \Rightarrow P' \Rightarrow P''$. 
    The case $p = 11$ cannot happen, because then by (P') we must have $3 A \smallsetminus 2 B = \{ 11 x \}$, but either $x \in B$ or $x \not\in B$ would lead to a contradiction. 
    Hence we are left with three possibilities: (i) $p = 5$ and $\deg j = 10$, (ii) $p = \deg j = 5$ and (iii) $ p = \deg j = 7$. 
    
    (i) As $\deg j = 10$, $\#(3A \cap 2B) = 2$. 
    Hence we must have $3A \cap 2B = \{ 2 x \}$ for some $x \in B$. 
    By (P) we know that $(3A \smallsetminus 2B) = \{ 5x', 5x''\}$ for some $x', x'' \in C$. It is clear then that $A = \{ 2x', 2x'' \}$. If $x' \neq x''$, then $3A \smallsetminus 2B = \{ x', x'' \}$ cannot be of the form $\{ 2 x \}$. Hence we must have that $A = \{ 4x \}$. As $2B$ cancels out exactly two copies of $x$, $x$ has multiplicity exactly $1$ in $B$. The statement that $\div(\Delta)= 2x + 5y + 5z$ for some $y, z \neq x$ then follows by applying (P) to the poles of $j$. 
    
    (ii) By (P') we must have $3A \smallsetminus 2B = \{ 5x \}$ for some $x$.  The only integer solution to the equation $3 m_x(A) - 2 m_x(B) = 5$ in the range $1 \le m_x(A) \le 4$ is $(m_x(A), m_x(B)) = (3, 2)$. 
        Therefore, $ A= \{ 3x , y \}$ for some $y \neq x$. As $y \not\in 3A \smallsetminus 2B$, we must have $m_y(B) \ge 2$.
        A quick inspection of Kodaira's table \cite[p.41, IV.3.1]{Miranda}\footnote{It is well known that Kodaira's classification works verbatim in characteristic $p \ge 5$ as the case over $\IC$. 
        The $J$ used in \textit{loc. cit.} is the reduced $j$-invariant, i.e., $j / 1728 = 4a_4^3/ \Delta.$.  } tells us that the Weirstrass normal form of $X$ has to have a type III fiber at $y$.
        But then the multiplicity of $j$ at $y$ is $2 m_y(B) - 3$ is divisible by $5$, which forces $m_y(B) = 4$. 
        
        (iii) The analysis is similar to that of (ii), so we omit the details. 
\end{proof}

\begin{corollary}
\label{cor: bound inseparable}
    View the vector space $\H^0(C, L^4) \times \H^0(C, L^6)$ as the affine space $\IA^4 \times \IA^6 \simeq \IA^{10}$ over $k$. The set $I$ of $k$-points $(a_4, a_6)$ such that the corresponding elliptic surface $X$ has finite inseparable $j$-invariant map is contained in a subscheme of dimension at most $2$. 
\end{corollary}
\begin{proof}
     
    Let $\{ \Delta = 0 \}$ be the subscheme of $\IA^{12}$ such that $\Delta(a_4, a_6)$ vanishes identically, and set $\IA' := \IA^{12} \smallsetminus \{ \Delta = 0 \}$. 
    For every $r \in \IN$, let $C^{[r]}$ denote the Hilbert scheme of $r$-points on $C$, which is well known to be the same as the $r$th symmetric power as $C$ is a smooth curve. 
    Then the association $(a_4, a_6) \mapsto \mathrm{div}(\Delta)$ defines a morphism $\IA' \to C^{[12]}$. 
    Note that a point on $C^{[r]}$ can also be viewed as a degree $r$ effective divisor. 

    Let us start with case (i) in \cref{prop: inspearable j}. As $a_4 \neq 0 \in L^4(-4 x)$, $x$ is a $4$-torsion point under the group law by (\ref{eqn: sum to zero}). Hence there are at most $16$ different choices for $x$. Let us fix one. Since $h^0(L(-4x)) = 1$, there is a line in $\IA^4$ parametrizing those $a_4$ such that $\mathrm{div}(a_4) = 4x$. 
    Let $\IA''$ be the fiber of $\IA'$ over this line.  
    Let $C' \in C^{[2]}$ be the subscheme parametrizing those pairs $y + z$ such that $2x \oplus 5y \oplus 5z = e$.
    Let $Z \subseteq C^{[12]}$ be the image of the morphism $C' \to C^{[12]}$ given by $y + z \mapsto 2 x + 5y + 5z$. 
    Since two sections $s, s' \in \H^0(C, L^{12})$ define the same divisor if and only if $s \in k^\times s'$, every fiber of the morphism $\IA' \to C^{[12]}$ (and hence the composition $\IA'' \into \IA' \to C^{[12]}$) has dimension at most $1$.
    Now, $I$ is contained the set of $k$-points of $\IA'' \times_{C^{[12]}} C'$. 
    As $\dim C' = 1$, $\dim \IA'' \times_{C^{[12]}} C' \le 1 + 1 = 2$. 

    Now consider case (ii). The equations $3 x \oplus y = e$ and $2x \oplus 4y = e$ imply that $x$ is a $10$-torsion point. Hence there are finitely many of them. Fixing each $x$, each $a_4$ and $a_6$ are well defined up to scalars, so we obtain the dimension bound. Case (iii) is entirely similar. 
\end{proof}

\subsection{Ramified covers of curves over discrete valuation rings}
\label{sec: ramified covers}
We start by recalling the relative Abhyankar's lemma from \cite{SGA1}.
\begin{proposition}\cite[Expos\'e XIII, Prop.\ 5.5]{SGA1}
	\label{prop:relative_Abhy}
	Let $X\to S$ be a smooth map of schemes, and let $D$ be a reduced relative normal crossing divisor relative to $S$.
	Assume $U'\to U=X \smallsetminus D$ is a finite \'etale cover such that its fiber over each geometric generic point of $S$ is tamely ramified along $D$.
	Then \'etale locally around any point of $U$, the map $U'\to U$ can be extended to a finite disjoint union of standard tamely ramified covers of $X$.
	In particular, $U'\to U$ is tamely ramified (along $D$) relative to $S$.
\end{proposition}
Here we recall from \cite[Expos\'e XIII, 5.3.0]{SGA1} (see also \cite[\href{https://stacks.math.columbia.edu/tag/0EYF}{Tag 0EYF}]{stacks-project}) that a standard tamely ramified cover of $X=\Spec(R)$ that is ramified at $D=\Spec(R/f))$ is of the form $\Spec(R[t]/t^n-f)$, where the integer $n$ is invertible in $R$.
In particular, the above result is stable under any base change of $S$.

\begin{proposition}
\label{prop:etaleness_of_pullback}
	Let $g : C' \to C$ be a finite flat morphism between smooth proper curves over $\sO_K$ such that the special fiber $C'_k \to C_k$ is separable. 
	Suppose that $g$ is \'etale over $\sC \smallsetminus D$ for a relative Cartier divisor $D$ that is \'etale over $\sO_K$. 
	Then the reduced subscheme of $g^{-1}(D)$ is also \'etale over $\sO_K$. 
\end{proposition}
\begin{proof}
	As the claim is \'etale local with respect to $\sO_K$, we replace $\sO_K$ by its maximal unramified extension and assume that $D$ is a disjoint union of $\sO_K$-points $D_1, \cdots, D_n$. 
	Suppose that the reduction $g_k^{-1}(D_i)_k$ is supported on $k$-points $P_{i, 1}, \cdots, P_{i, n_i}$ and has multiplicity $m_{i,j}$ at $P_{i, j}$.
	Then by the separatedness of the map $g_k$ and Riemann--Hurwitz formula (\cite[\href{https://stacks.math.columbia.edu/tag/0C1F}{Tag 0C1F}]{stacks-project}), we have 
	\begin{equation}
		\label{eq:Riemann-Hurwitz_k}
		\chi(C'_k) - (\deg \, g) \cdot \chi(C_k) = \sum_{i = 1}^n \sum_{j = 1}^{n_i} (m_{i, j} - 1). 
	\end{equation}
    Here we implicitly use the fact that the map $g_k$ is tamely ramified over each $D_k$, thanks to the relative Abhyankar's Lemma \Cref{prop:relative_Abhy}.
    
    On the other hand, for each $i, j$, we set $\mathcal{Q}_{i, j}$ to be the set of closed points $Q$ on $C'_K$ such that $Q$ specializes to $P_{i, j}$. 
    Let $m_Q$ be the multiplicity of the map $g_K$ at $Q$, and let $e_Q$ be the ramification index of the field extension $K(Q)/K$ (which by assumption is also the degree of $K(Q)/K$).
    Then by Riemann--Hurwitz formula again, we get
	\begin{equation}
		\label{eq:Riemann-Hurwitz_K}
		\chi(C'_K) - (\deg \, g) \cdot \chi(C_K) = \sum_{i = 1}^n \sum_{j = 1}^{n_i} \sum_{Q \in \mathcal{Q}_{i, j}}(m_Q - 1)\cdot e_Q. 
	\end{equation}
    Notice that by the flatness of the map $g$ and the flat base change theorem of the coherent cohomology (\cite[\href{https://stacks.math.columbia.edu/tag/02KH}{Tag 02KH}]{stacks-project}), the left hand side of (\ref{eq:Riemann-Hurwitz_k}) is equal to that of (\ref{eq:Riemann-Hurwitz_K}).
    Hence we get the equality
    \begin{equation}
    	\label{eq:ram_deg_k_and_K}
    	\sum_{i = 1}^n \sum_{j = 1}^{n_i} (m_{i, j} - 1) = \sum_{i = 1}^n \sum_{j = 1}^{n_i} \sum_{Q \in \mathcal{Q}_{i, j}}(m_Q - 1)\cdot e_Q. 
    \end{equation}

    Moreover, by the flatness of the map $g$, the preimage $g^{-1}(D)=D\times_{C} C'$ is finite flat over $\sO_K$.
    In particular, we get
    \begin{equation}
    	\label{eq:equality_of_dim_of_divisors}
    	\dim_k g^{-1}(D)_k = \dim_K g^{-1}(D)_K.
    \end{equation}
    Here we note that by summing over the connected components of $g^{-1}(D)_K$ and $g^{-1}(D)_k$ individually, we have
    \begin{align}
    	\label{eq:expansion_of_g^{-1}D}
    	&\dim_k g^{-1}(D)_k  = \sum_{i = 1}^n \sum_{j = 1}^{n_i} m_{i, j};\\
    	&\dim_K g^{-1}(D)_K = \sum_{i = 1}^n \sum_{j = 1}^{n_i} \sum_{Q \in \mathcal{Q}_{i, j}}m_Q \cdot e_Q. 
    \end{align}
    Hence we get another equality on the multiplicities and the ramification indexes
    \begin{equation}
    	\label{eq:ram_deg_k_and_K_2}
    	\sum_{i = 1}^n \sum_{j = 1}^{n_i} m_{i, j} = \sum_{i = 1}^n \sum_{j = 1}^{n_i} \sum_{Q \in \mathcal{Q}_{i, j}} m_Q \cdot e_Q. 
    \end{equation}

    Now, by substituting (\ref{eq:ram_deg_k_and_K_2}) into (\ref{eq:ram_deg_k_and_K}) and cancel the terms with multiplication indexes, and by the positivity of each $e_Q$, we get
    \begin{align*}
    	    \sum_{i = 1}^n n_i = \sum_{i = 1}^n \sum_{j = 1}^{n_i} \sum_{Q \in \mathcal{Q}_{i, j}} e_Q  \geq \sum_{i = 1}^n \sum_{j = 1}^{n_i} \sum_{Q \in \mathcal{Q}_{i, j}} 1  = \sum_{i = 1}^n \sum_{j = 1}^{n_i} |Q_{i,j}| \geq \sum_{i = 1}^n \sum_{j = 1}^{n_i} 1 = \sum_{i = 1}^n n_i,
    \end{align*}
    where the last equality follows from the surjectivity of the specialization map.
    Hence each set $Q_{i,j}$ has only one point (which we shall also denote by $Q_{i,j}$), and the residue field extension $K(Q_{i,j})/K$ is trivial and hence unramified.
    Thus the claim follows from \Cref{lem:finite_flat_O_K_scheme} below.
\end{proof}

\begin{lemma}
	\label{lem:finite_flat_O_K_scheme}
	Let $D=\Spec(R)$ be a connected finite flat $\sO_K$-scheme.
	Assume $(R[1/p])_\mathrm{red}$ is equal to $\Spec(K)$.
	Then $R_\mathrm{red}$ is isomorphic to (and hence \'etale over) $\sO_K$.
\end{lemma}
Here we note that by the flatness of $R$ over $\sO_K$, \Cref{lem:finite_flat_O_K_scheme} in particular implies that in the setting of \Cref{prop:etaleness_of_pullback}, the multiplicity of $P_{i,j}\in C_k'$ coincides with the multiplicity of $Q_{i,j}\in C_K'$, i.e. $m_{i,j}=m_{Q_{i,j}}$.
\begin{proof}
We let $\overline{R}$ be the image of the composition $R\to R[1/p] \to (R[1/p])_\mathrm{red}=K$, which is a finite $\mathcal{O}_K$-subalgebra in $K$.
If $\overline{R}\neq \sO_K$, the $p$-adic valuation of any element $x\in \overline{R}\smallsetminus \sO_K$ is negative.
However, the $\sO_K$-submodule $\sO_K[x]$ is not finitely generated over $\sO_K$, contradicting to the assumption that $\overline{R}$ is finite over $\sO_K$.
Hence $\overline{R}=\sO_K$.
Finally, since $\ker(R\to \overline{R})$ is contained in $\ker(R[1/p]\to (R[1/p])_\mathrm{red}=K)$, we see $\ker(R\to \overline{R})$ is nilpotent, and hence $R_{\mathrm{red}}=\sO_K$.
\end{proof}

\begin{proposition}
	\label{prop:lift_closed_immersion}
	Let $X$ be a smooth projective scheme over $\sO_K$, and let $D \subseteq X$ be a relative Cartier divisor whose special fiber $D_k$ is generically reduced on each connected component. 
	Let $C \subseteq X$ be a closed subscheme which is smooth and proper of relative dimension one over $\sO_K$, such that $C \cap D$ is \'etale over $\sO_K$. 
	
	Assume $X_1^\circ$ (resp. $C_1^\circ$) is a connected finite \'etale cover of $X^\circ \colonequals X \smallsetminus D$ (resp. $C^\circ \colonequals C \smallsetminus C \cap D$). 
	Then the closed immersion $C^\circ \into X^\circ$ lifts to a map $C_1^\circ \to X_1^\circ$ if and only if this is true on the special fiber.  
\end{proposition}
In fact, we show below that the reduction map induces a natural bijection between the lifts $C_1^\circ \to X_1^\circ$ and the lifts $C_{1,k}^\circ \to X_{1, k}^\circ$.
\begin{proof}
	By assumption, we have the following diagram of schemes
		\[\begin{tikzcd}
		{C_1^\circ} & {X_1^\circ} \\
		{C^\circ} & {X^\circ},
		\arrow[from=1-1, to=2-1]
		\arrow[from=1-2, to=2-2]
		\arrow["i", hook, from=2-1, to=2-2]
	\end{tikzcd}\]
    where the both vertical arrows are finite \'etale covers.
    In particular, the fiber product $C_2\colonequals C^\circ\times_{X^\circ} X_1^\circ$ is a finite \'etale cover of $C^\circ$.
    So to show the statement, it suffices to show that the reduction map induces an isomorphism of hom groups, namely
    \[
\Hom_{C^\circ_{\et}}(C_1^\circ , C_2) \simeq \Hom_{C^\circ_{k, \et}}(C_{1,k}^\circ , C_{2,k}^\circ),
    \]
    where $(-)_\et$ is the \'etale site of the given scheme.
    
    Moreover, we notice that since the complement divisor $D=C\smallsetminus C^\circ$ is \'etale over $\sO_K$ by assumption, it is in particular a reduced and normal crossing divisor in $C$ that is relative to $\sO_K$.
    So by the relative Abhyankar's lemma in \Cref{prop:relative_Abhy}, we see both of the \'etale covers $C_i^\circ$ for $i=1,2$ and their reductions are tamely ramified along the complements.
    To proceed, we apply \cite[Expos\'e XIII, Cor.\ 2.8]{SGA1} at the diagram
    \[
    \begin{tikzcd}
    	C^\circ \ar[r] \ar[d] & C \arrow[dl] \\
    	\Spec(\sO_K),&
    \end{tikzcd}
    \]
    where the $1$-acyclicity of the map $C\to \Spec(\sO_K)$ with respect to the set $\mathbb{L}=\{\text{prime numbers\,} \ell~|~\ell\neq p\}$ follows from \cite[Expos\'e XV, Thm.\ 2.1]{SGA4}.
    The loc.\ cit. implies that for each finite group $G$ of order prime to $p$, the reduction map induces a bijection of groups
    \[
    \mathrm{H}^1_\mathrm{t}(C^\circ, G) \simeq \mathrm{H}^1_\mathrm{t}(C^\circ_k, G),
    \]
    where $\mathrm{H}^1_\mathrm{t}(-,G)$ is the equivalent classes of $G$-torsors that are tamely ramified along the normal crossing complement above.
    In particular, we obtain the isomorphism of tame fundamental groups
    \footnote{The original definition of the tame fundamental group in \cite[Expor\'e XV, 2.1.3]{SGA1} is for the prescribed compactification with normal crossing divisors. 
    	However, it is shown in \cite[Thm.\ 1.1]{KS10} the notion is independent of such choices when the compactifications are regular and proper with normal crossing divisors.}
    \[
    \pi_1^\mathrm{t}(C^\circ, \overline{x}) \simeq \pi_1^\mathrm{t}(C^\circ_k, \overline{x}_k),
    \]
    where $\overline{x}$ is a geometric point of $C^\circ$.
    We let $(-)_{\et}^\mathrm{t}$ be the full subcategory of the \'etale site $(-)_\et$ consists of the \'etale objects that are tame along the normal crossing complement.
    Then by the general result on the Galois category \cite[Expor\'e V, Prop.\ 6.9]{SGA1}, the isomorphism above implies that the reduction map induces a fully faithful functor of Galois categories
    \[
    (C^\circ)_\et^\mathrm{t} \longrightarrow (C^\circ_k)_\et^\mathrm{t}, 
    \]
    which concludes the proof.
\end{proof}

\section{Period morphisms}
\label{sec:period_mor}
In this section, we discuss the period morphism for a \textit{smooth} family of projective varieties with $h^{2, 0} = 1$. 
The key property of the period morphism is that there are matching isomorphisms between the pullback of the automorphic sheaves on Shimura varieties and the (primitive) cohomology sheaves of the varieties in question. 
For Betti and $\ell$-adic cohomology, this has been done in \cite{HYZ} by extending the construction in \cite{MPTate} using the inputs from \cite{Moonen}. 
Here we fill in the treatment of crystalline and de Rham cohomology.

The main difficulty to overcome is that we cannot assume that the motives of the varieties in question are abelian, so that the matching isomorphism of de Rham cohomology over $\IC$ does not automatically descent over $\IQ$ as in \cite[Prop.~5.6(2)]{MPTate}; moreover, its compatibility with the $p$-adic matching isomorphism does not follow from Blasius' result \cite{Blasius}.
Instead, we shall achieve this by making use of the $p$-adic Riemann-Hilbert functor in \cite{DLLZ} and the analytic density of Noether--Lefschetz loci (cf. proof of \cref{lem: Voisin}) to conclude that there is a natural descent over $\IQ_p$, which is enough for our applications. 

\subsection{General set-up}
\label{sub: period_setup}
Let $p > 2$ be a prime and $\sS$ be a smooth connected scheme over $\IZ_{(p)}$ of finite type with generic point $\eta$ and $S \colonequals \sS_\IQ$ be its generic fiber. 
Let $f : \sX \to \sS$ be a smooth projective morphism with geometrically connected fibers of relative dimension $d \colonequals \dim \sX / \sS \geq 2$. 
Let $\Theta$ be a relatively ample line bundle on $\sX/\sS$ and $\Lambda \subseteq \NS(\sX_\eta)_{\mathrm{tf}}$ be a sublattice which contains the class of $\Theta_\eta$. 
Here we are taking an ad hoc definition for the N\'eron--Severi group when the base field is not alegebraically closed: $\NS(\sX_\eta) \colonequals \mathrm{im}(\Pic(\sX_\eta) \to \NS(\sX_{\bar{\eta}}))$. 
Note that $\Theta$ defines symmetric bilinear pairings on $(R^2 f_{\IC *} \IZ_{(p)})_\tf$ over $S_\IC$, $R^2 f_{\IQ, \et *} \IQ_\ell$ over $\sS$, $(R^2 f_{\et *} \IZ_p)_\tf$ over $S$, and $R^2 f_{*} \Ohm^\bullet_{\sX/\sS}$ individually, via the formula $\< x, y \> = x \cup y \cup c_1(\Theta)^{d - 2}$, as each element in $\NS(\sX_\eta)_{\mathrm{tf}}$ gives rise to a global section $c_1(\Theta)$ of these sheaves. 
Let $\bP_{B}, \bP_\ell, \bP_p$ and $\bP_\dR$ be the orthogonal complement of $\Lambda$ in these sheaves. Over $\IF_p$, we similarly define the F-isocrystal $\bP_\crys[1/p]$ on $\crys(\sS_{\IF_p}/\IZ_p)$. When the second crystalline cohomology of $\sX_{\IF_p}$ is locally free (e.g., when \cref{lem:local_free_of_crys_coh} is applicable), we also define $\bP_\crys$. For each fiber, write $\PNS(-)$ for the orthogonal complement of $\Lambda$ in $\NS(-)$.

\begin{definition}
    \emph{(Algebraic Monodromy)} Let $T$ be a normal integral notherian scheme and $\sfW_\ell$ be an \'etale $\IQ_\ell$-local system on $T$ ($\ell \in \sO_T^\times$). For any geometric point $t \to T$, we define the \textit{algebraic monodromy group} $\mathrm{Mon}(\sfW_\ell, t)$ to be the Zariski closure of the image of the monodromy representation $\pi_1^\et(T, t) \to \mathrm{GL}(\sfW_{\ell, t})$. Denote by $\mathrm{Mon}^\circ(\sfW_\ell, t)$ its identity component. 
\end{definition}

\begin{assumption}\label{assump: period} We shall consider families $f : \sX \to \sS$ under the following conditions:

\begin{enumerate}[label=\upshape{(\roman*)}]
    \item For some (and hence every) $s \in S(\IC)$, $h^{2,0}(\sX_s) = 1$. 
    \item For some $s \in S(\IC)$, the Kodaira--Spencer map 
    $$ \mathrm{KS}_s : T_s S_\IC \to \Hom(\H^1(\Ohm^1_{\sX_s}), \H^2(\sO_{\sX_s})) $$
    is nonzero. 
    \item For some (and hence every) geometric point $s \to \sS$, $\bP_{p, s} = \PH^2_\et(\sX_s, \IZ_p)_{\mathrm{tf}}$ is self-dual. 
    \item For some (and hence every) geometric point $s \to \sS$, $\mathrm{Mon}(\bP_2, s)$ is connected. 
\end{enumerate}
Additionally, assume one of the following: 
\begin{enumerate}
    \item[(v)] For some $s \in \sS(\IC)$, the map $\mathrm{KS}_s$ in (ii) has rank $\ge 2$. 
    \item[(vi)] The motive $\fh^2(\sX_s) \in \mathsf{Mot}_{\mathrm{AH}}(\IC)$ determined by $\H^2(\sX_s, \IQ)$ is abelian for every $s \in \sS(\IC)$, where $\mathsf{Mot}_{\mathrm{AH}}(\IC)$ is the category of motives with absolute Hodge cycles over $\IC$ (cf. \cite[\S2]{Pan94}\footnote{Here we are following the notations of \cite[\S2]{MPTate}. The notation $\fh^i(X)$ means the submotive of $\fh(X) \colonequals (X, 0, \mathrm{id})$ whose cohomological realizations give $\H^2$.}). 
\end{enumerate}
\end{assumption}

\begin{construction}
\label{const: sheaves on Sh}
Choose a base point $b \in S(\IC)$ which lies above the generic point $\eta$ of $\sS$ (or $S$). 
Let $S^\dagger$ be the connected component of $S_\IC$ which contains $b$. 
Let $L$ be a quadratic form over $\IZ_{(p)}$ such that $L \iso \PH^2(\sX_b, \IZ_{(p)})_{\mathrm{tf}}$ and fix such an isomorphism. Let $V \colonequals L_\IQ$ and $G$ (resp. $\wt{G}$) be the reductive $\IQ$-group $\SO(V)$ (resp. $\CSpin(V)$). 
By a slight abuse of notation we also write $G$ (resp. $\wt{G}$) for its $\IZ_{(p)}$-model $\SO(L)$ (resp. $\CSpin(L)$). 
Denote by $\Omega$ the Hermitian symmetric space $\{ \w \in \IP(V \tensor \IC) \mid \< \w, \bar{\w} \> > 0 \}$. 
Let $\sfK_p \colonequals G(\IZ_p)$ and $\wt{\sfK}_p \colonequals \wt{G}(\IZ_p)$. 
Let $\sfK^p$ (resp. $\wt{\sfK}^p$) be a sufficiently small compact open subgroup of $G(\IA_f^p)$ (resp. $\wt{G}(\IA^p_f)$) such that $\sfK^p$ contains the image of $\wt{\sfK}^p$. 
Set $\sfK := \sfK_p \sfK^p$ and $\wt{\sfK} := \wt{\sfK}_p \wt{\sfK}^p$. 

We briefly recall some basic facts and commonly used notations of Shimura varieties and various sheaves: 
The reflex field of $(\wt{G}, \Ohm)$ is $\IQ$ and by \cite{KisinInt} there is an canonical integral model $\shS_{\wt{\sfK}}(\wt{G})$ over $\IZ_{(p)}$. 
There is a suitable sympletic space $(H, \psi)$ and a Siegel half space $\sH^\pm$ such that there is an embedding of Shimura data $(\wt{G}, \Ohm) \into (\GSp(\psi), \sH^\pm)$ which eventually equips $\shS_{\wt{\sfK}}(\wt{G})$ with a universal abelian scheme $\sA$.\footnote{Technically, in \cite{KisinInt} and \cite{CSpin}, $\sA$ is only defined as a sheaf of abelian schemes up to prime-to-$p$ quasi-isogeny. 
However, for $\wt{\sfK}^p$ sufficiently small, we can take $\sA$ to be an actual abelian scheme (cf. \cite[(2.1.5)]{KisinInt}).} 
Let $a : \sA \to \shS_{\wt{\sfK}}(\wt{G})$ be the structural morphism. 
Define the sheaves $\bH_B \colonequals R^1 a_{\IC *} \IZ_{(p)}$, $\bH_\ell \colonequals R^1 a_* \underline{\IQ}_\ell$ ($\ell \neq p$), $\bH_p \colonequals R^1 a_{\IQ *} \underline{\IZ}_p$, $\textbf{H}_\dR \colonequals R^1 a_* \Ohm^\bullet_{\sA / \shS_{\wt{\sfK}}(\wt{G})}$ and $\bH_\crys \colonequals R^1 \bar{a}_{\crys *} \sO_{\sA_{\IF_p}/ \IZ_p}$ (where $\bar{a} \colonequals a \tensor \IF_p$). 
The abelian scheme $\sA$ is equipped with a ``CSpin-structure'': a $\IZ / 2 \IZ$-grading, $\Cl(L)$-action and an idempotent projector $\bpi_? : \End(\bH_?) \to \End(\bH_?)$ for $?\in \{B, \ell, p, \dR, \crys\}$ on (various applicable fibers of) $\shS_{\wt{\sfK}}(\wt{G})$. 
We use $\bL_?$ to denote the images of $\bpi_?$. 
Given any $\shS_{\wt{\sfK}}(\wt{G})$-scheme $T$, an endomorphism $f \in \End(\sA_T)$ is called a \textit{special endomorphism} (\cite[Def.~5.2, see also Lem.~5.4, Cor.~5.22]{CSpin}) if for some (and hence all) $\ell \in \sO_T^\times$, the $\ell$-adic realization of $f$ lies in $\bL_\ell|_T \subseteq \End(\bH_\ell|_T)$; if $\sO_T = k$ for a perfect field $k$ in characteristic $p$, then equivalently $f$ is called a \textit{special endomorphism} if the crystalline realization of $f$ lies in $\bL_{\crys, T}$. We write the submodule of $\End(\sA_T)$ consisting of special endomorphisms as $\LEnd(\sA_T)$. 
The sheaves $\bL_?$ and special endomorphisms on $\shS_{\wt{\sfK}}(\wt{G})$ descend to its \'etale quotient $\shS_{\sfK}(G)$, and the descents are customarily denoted by the same letters. 
\end{construction}

Below we assume that $\sfK^p$ is of the form $\prod_{\ell \neq p} \sfK_\ell$ for $\sfK_\ell \subseteq G(\IQ_\ell)$ and is contained in the stabilizer of the lattice $\PH^2(\sX_b, \what{\IZ}^p)_{\mathrm{tf}} \subseteq V \tensor \IA^p_f$. 

\begin{theorem}
\label{thm: period}
    Under Assumption~\ref{assump: period}, up to replacing $\sS$ by a connected finite \'etale cover, there exists a period morphism $\rho : \sS \to \shS_\sfK(G)$ with the following properties: 
   \begin{enumerate}[label=\upshape{(\alph*)}]
       \item There exists an isometry of $\IZ_{(p)}$-VHS $(\alpha^\dagger_B, \alpha^\dagger_{\dR, \IC}) : \rho_\IC^* (\bL_B, \bL_{\dR, \IC})|_{S^\dagger} \sto (\bP_B, \bP_{\dR, \IC})$ over $S^\dagger$.\footnote{In \cite{HYZ} we used $S^\circ$ instead of $S^\dagger$ to denote the connected component of $S_\IC$ that contains the point $b$. 
       Here we changed it to avoid a clash of notations with \cref{sec: proofs of theorems}.}
       \item For every prime $\ell \neq p$, there exists an isometry of $\ell$-adic local systems $\alpha_\ell : \rho^* \bL_\ell \sto \bP_\ell$ over $\sS$ whose restriction to $S^\dagger$ agrees with $\alpha^\dagger_B \tensor \IQ_\ell$. 
       \item There exists an isometry of $p$-adic local systems $\alpha_p : \rho_\IQ^* \bL_p \sto  \bP_p$ over $S$ whose restriction to $S^\dagger_\IC$ agrees with $\alpha_B^\dagger \tensor \IZ_p$. 
   \end{enumerate}
\end{theorem}
\begin{proof}
        This can be extracted from the proof of \cite[Prop.~5.3.3(a)]{HYZ}. We give a sketch of the arguments for readers' convenience. By an $\ell$-independence of the group of connected components of algebraic monodromy groups \cite[Prop.~6.14]{Larsen-Pink}, the assumption (iv) is used to make sure that $\mathrm{Mon}(\bP_\ell, s)$ for every prime $\ell \neq p$. Therefore, up to replacing $\sS$ be a \textit{finite} \'etale cover, the image of $\pi_1^\et(S, b)$ in $G(\IA^p_f)$ is contained in $\sfK^p$. The VHS $(\bP_B, \bP_{\dR, \IC})$ on $S^\dagger$, together with suitable level structures, defines a morphism $\rho^\dagger : S^\dagger \to \Sh_\sfK(G)_\IC$. Let $F$ be the field of definition of $S^\dagger$. Then \cite[Thm.~4.1.9]{HYZ} implies that $\rho^\dagger$ descends to $F$, so that by applying the forgetful functor from the category of $F$-schemes to that of $\IQ$-schemes, we obtain a $\IQ$-morphism $\rho_\IQ : S \to \Sh_\sfK(G)$. Moreover, we obtain isomorphisms $\alpha_\ell : \rho_\IQ^* \bL_\ell \sto \bP_\ell|_{S}$ and $\alpha_p : \rho_\IQ^* \bL_p \sto \bP_p$ over $S$ which are compatible with $\alpha_B^\dagger$ over $S^\dagger$. Our assumption (vi) corresponds to assumption (a) in \textit{loc. cit.}, and our assumption (v) ensures that, if (a) in \textit{loc. cit.} is not satisfied, then (b) is, thanks to \cite[Lem.~2.2.5]{HYZ}. The condition denoted by ``$\#$'' in \cite[(4.1.6)]{HYZ} is automatically satisfied when $\sfK$ is small enough. The fact that $\rho_\IQ$ extends (necessarily unique) to a $\rho$ over $\IZ_{(p)}$ is a consequence of the extension property of the canonical model (\cite{KisinInt}, cf. \cite[Thm.~3.4.4]{HYZ}). Then $\alpha_\ell$ extends over $\sS$ as the natural map $\pi_1^\et(S, b) \to \pi_1^\et(\sS, b)$ is surjective. 
\end{proof}

\subsection{Recollections of the $p$-adic Riemann-Hilbert functor}

We briefly recall the $p$-adic Riemann-Hilbert functor. Let $K$ be a $p$-adic field with perfect residue field and fix an algebraic closure $\bar{K}$. Let $T$ be a smooth rigid analytic variety over $K$. Let $\Loc_{\IQ_p}(T)$ denote the category of (lisse) $p$-adic \'etale local systems. 
In \cite[Thm.~3.9]{LiuZhu}, Liu and Zhu defined a functor $\ID_\dR$ which sends a $\IL \in \Loc_{\IQ_p}(T)$ to a vector bundle $\ID_\dR(\IL)$ with integrable connection $\nabla_\IL$ and filtration $\Fil^\bullet$ on $T$, 
satisfying the following properties.
\begin{itemize}
    \item The formation of $(\ID_\dR(\IL), \nabla_\IL)$ commutes with base change. That is, if $f : T' \to T$ is a morphism between smooth rigid-analytic varieties over $K$, then there is a canonical isomorphism 
\begin{equation}
    \label{eqn: DdR commutes with base change}
    f^* (\ID_\dR(\IL), \nabla_\IL) \simeq (\ID_\dR(f^* \IL), \nabla_{f^* \IL}).
\end{equation}
\item The filtration satisfies Griffiths transversality.
\item The restriction of $\ID_\dR$ onto the full subcategory of \textit{de Rham} local systems (in the sense of Scholze \cite[Def.~8.3]{Sch13}) is a tensor functor and preserves the rank.
\item When $T = \mathrm{Spa}(K)$ is just a point, the functor $\ID_\dR$ recovers the classical Fontaine's $\ID_\dR$ functor (which we denote by the same symbol), i.e. $\ID_\dR(\IL) = \bigl(\IL(\mathrm{Spa}(\IC_p)) \tensor_{\IQ_p} \mathrm{B}_\dR \bigr)^{\Gal_K}$. 
\end{itemize}

One can of course apply $\ID_\dR$ to $\IZ_p$-local systems after inverting $p$. 
We recall (a special case of) Scholze's result \cite[Thm.~8.8]{Sch13} in the following form: 
\begin{theorem}
\emph{(Scholze)}
    Let $f : X \to T$ be a smooth proper morphism between smooth rigid-analytic varieties over $K$. Then (the torsion-free part of) $R^i f_{\et *} \IZ_p$ is de Rham and there is a functorial isomorphism bewteen filtered flat vector bundles 
    \begin{equation}
    \label{eqn: relative comparison}
        \ID_\dR(R^i f_{\et *} \IZ_p) \simeq R^i f_* \Omega_{X/T}^\bullet. 
    \end{equation}
\end{theorem}
\begin{proof}
    By \cite[Thm.~8.8]{Sch13} and its proof, there is a natural filtered isomorphism of pro-\'etale sheaves 
    $$ R^i f_{\et *} \IZ_p \tensor_{\IZ_p} \sO \IB_{\dR, T} \simeq R^i f_* \Omega_{X/T}^\bullet \tensor_{\sO_T} \sO \IB_{\dR, T}. $$
    Recall that on $T$, $\ID_\dR(-) = \nu_*(- \tensor_{\IZ_p} \sO \IB_{\dR, T}) $, where $\nu$ is the projection of topoi $T_{\textrm{pro{\'et}}} \to T_\et$. By the projection formula of finite projective modules, and the fact that $\nu_* \sO \IB_{\dR, T} = \sO_{T}$ (\cite[Cor.~6.19]{Sch13}), we obtain (\ref{eqn: relative comparison}) by applying $\nu_*$ to the above isomorphism. 
\end{proof}

Now suppose that $T$ is a smooth algebraic variety over $K$ and let $T^{\mathrm{an}}$ be its analytification. Then there is an algebraic de Rham functor $\ID^\alg_\dR$ which sends de Rham $p$-adic local systems on $X$ to a filtered flat vector bundle over $X$, whose analytification is identified with $\ID_\dR(\IL)$ (\cite[Thm~1.1]{DLLZ}). Using this functor, we obtain an algebraic version of (\ref{eqn: relative comparison}). 

\begin{theorem}
\label{thm: alg. compare}
    Let $f : X \to T$ be a smooth proper morphism between smooth $K$-varieties. Then (the torsion-free part of) $R^i f_{\et *} \IZ_p$ is de Rham and there is a functorial isomorphism $\mathsf{c}_\dR : \ID_\dR^{\mathrm{alg}}(R^i f_{\et *} \IZ_p) \simeq R^i f_* \Omega_{X/T}^\bullet$ bewteen filtered flat vector bundles over $\sO_T$ whose analytification is given by (\ref{eqn: relative comparison}).  
\end{theorem}
\begin{proof}
It suffices to produce an isomorphism between the underlying flat vector bundles, as the property of respecting the filtration can be checked after analytification. 
As explained in the proof of \cite[Lem.~4.1.2]{DLLZ}, by several results of Andr\'e and Baldassarri, the analytification functor from the category of algebraic regular connections on $T$ whose exponents contain no Liouville numbers to the category of analytic ones on $T^{\mathrm{an}}$ is fully faithful. Therefore, it suffices to show that both $\ID_\dR^{\mathrm{alg}}(R^i f_{\et *} \IZ_p)$ and $R^i f_* \Omega_{X/T}^\bullet$ are regular and have non-Liouville exponents. 
For the former, this has been explained in \textit{loc. cit.} that $\ID_\dR^{\mathrm{alg}}(\IL)$ is regular and has rational exponents for any local system $\IL$. 
For the latter, the same is true by the local monodromy theorem (e.g., see \cite[\S13]{Kat70}). 
\end{proof}


\begin{remark}
\label{rmk: Chern class}
    In the above theorem, if $L$ is a relative line bundle on $X/T$, then $\mathsf{c}_\dR$ sends the $\ID_\dR$ of its $p$-adic Chern class to its de Rham Chern class. 
    Indeed, by compatibility of Chern classes with base change it suffices
    to observe this when $T$ is a point. 
    For a reference in the latter case, see \cite[Prop.~11.2]{IIK}.
\end{remark}

\begin{corollary}
\label{cor: DdR on Sh}
    Consider the Shimura variety $\Sh_\sfK(G)$ in \cref{sub: period_setup}. Let $\bH_{\dR, \IQ_p}$ and $\bH_{p, \IQ_p}$ be the restrictions of $\bH_\dR$ and $\bH_{p}$ to $\Sh_\sfK(G)_{\IQ_p}$ respectively. Define $\bL_{\dR, \IQ_p}$ and $\bL_{p, \IQ_p}$ similarly.  The natural isomorphism $\ID_\dR^{\mathrm{alg}}(\bH_{p, \IQ_p}) \simeq \bH_{\dR, \IQ_p}$ induces an isomorphism of filtered flat vector bundles $\ID^{\mathrm{alg}}_\dR(\bL_{p, \IQ_p}) \simeq \bL_{\dR, \IQ_p}$. 
\end{corollary}
\begin{proof}
    Recall that $\bH_{p, \IQ_p}$ and $\bH_{\dR, \IQ_p}$ are given by the first $p$-adic and de Rham cohomology of the universal abelian scheme on $\Sh_\sfK(G)_{\IQ_p}$, so that the natural isomorphism $\ID_\dR^{\mathrm{alg}}(\bH_{p, \IQ_p}) \simeq \bH_{\dR, \IQ_p}$ is given by \cref{thm: alg. compare}. As $\bL_{?}$ is the image of the idempotent projector $\bpi_? : \End(\bH_?) \to \End(\bH_?)$ for $? = p, \dR$, it suffices to show that $\ID_\dR^{\mathrm{alg}}$ sends the restriction of $\bpi_p$ over $\Sh_\sfK(G)_{\IQ_p}$ to that of $\bpi_\dR$. This can be checked at closed points, for which the statement follows from Blasius' result. See \cite[Prop.~4.7]{CSpin}.\footnote{Note that here we are only talking about abelian varieties. The usage of Blasius result is to be avoided only for matching isomorphisms.}
\end{proof}

\subsection{Matching de Rham and crystalline realizations}
\label{sub:match crystalline}
\begin{lemma}
\label{lem: Voisin}
Let $B$ be a smooth connected $\IC$-variety and $\IV = (\IV_B, \IV_\dR)$ be a polarized $\IQ$-VHS of pure weight $2$. Assume that there is no locally constant rational class which is of type $(1, 1)$ everywhere. Suppose that at a point $t_0 \in B$, there exists some $\lambda \in \IV_{\dR, t_0}^{(1, 1)}$ such that the Kodaira-Spencer map induces a surjection $T_{t_0} B \twoheadrightarrow \IV_{\dR, t_0}^{(0, 2)}$. 
Let $U \subseteq B$ be any complex analytic disk containing $t_0$. 

Then some non-empty open ball $\IB$ of $\IV_{B, t_0}$ has the following property: For every $\beta \in \IB$, there exists a point $t \in U$ such that $\beta_t \in \IV_{B, t}$ obtained by $\beta$ via parallel transport within $U$ is of Hodge type $(1, 1)$.  
\end{lemma}
\begin{proof}
This is implied by the proof of the density of Hodge loci \cite[Thm~3.5]{Voisin}. We sketch the argument for readers' convenience. Up to replacing $B$ by an \'etale cover and splitting off a direct summand, we may assume that there is no local section of $\IV$ which is constant of type $(1, 1)$. Denote by $\IV_\IR$ the real bundle $\IV_B \tensor \IR$ over $S$. Let $F^1 \sV$ (resp. $\sV$, $\sV_\IR$) denote the total space of the vector bundle $\Fil^1 \IV_\dR$ (resp. $\IV_\dR$, $\IV_\IR$) restricted to $U$. The connection on $\sV$ defines a canonical trivialization of holomorphic vector bundle $\IV_\dR|_U = \IV_{t_0} \times U$, so there is a natural projection $\sV \to \sV_{t_0}$. 
Let $\phi$ denote its restriction to $F^1 \sV$. 

It follows from the assumption that for some $\wt{\lambda} \in (\sV_\IR \cap F^1 \sV)_{t_0}$, $\nabla_{t_0} \wt{\lambda} \colon T_{B, t_0} \to \IV^{0, 2}_{t_0}$ is surjective. By Lem.~3.6 in \textit{loc. cit.}, $\phi_\IR \colon F^1 \sV \cap \sV_\IR \to \sV_{\IR, t_0}$ is a submersion, in particular an open map, near $\wt{\lambda}$. Take $\IB$ to be the intersection of $\sV_{B, t_0}$ with a small open neighborhood of $\wt{\lambda}$. Then the projection of $\phi^{-1}_\IR(\IB)$ to $U$ is the Hodge locus. 
\end{proof}

\begin{lemma}
\label{lem: alpha is de Rham}
    In the setting of \cref{thm: period}, let $\hat{\alpha}_\dR : \bL(-1)_\dR|_{S_{\IQ_p}} \sto \bP_{\dR}|_{S_{\IQ_p}}$ be the isometry
    of flat vector bundles given by $\ID^{\alg}_\dR(\alpha_p)$. 
    Then for any isomorphism $\tau : \bar{\IQ}_p \simeq \IC$, the restriction of $\hat{\alpha}_\dR$ to $S^\dagger \subseteq S_\IC \stackrel{\tau}{\simeq} S_{\bar{\IQ}_p}$ is equal to $\alpha^\dagger_{\dR}$. 
\end{lemma}
First, we remark that by \cref{thm: alg. compare} and \cref{rmk: Chern class}, we may identify $\ID^{\mathrm{alg}}_\dR(\bP_p|_{S_{\IQ_p}})$ with $\bP_\dR |_{S_{\IQ_p}}$. Together with \cref{cor: DdR on Sh}, $\ID^{\alg}_\dR(\alpha_p)$ gives rise to an isomorphism $\hat{\alpha}_\dR$ as above. 
\begin{proof}
    Note that $\hat{\alpha}_\dR|_{S_\IC}$ and $\alpha^\dagger_\dR$ are both isomorphisms $\bL_\dR (-1) |_{S^\dagger_\IC} \sto \bP_\dR |_{S^\dagger_\IC}$. 
    Our goal is to show that $\gamma : (\hat{\alpha}_\dR|_{S_\IC})^{-1} \circ \alpha^\dagger_\dR$ is the identity on $\bP_{\dR}|_{S^\dagger_\IC}$. Let $s \in S^\dagger_\IC$ be any $\IC$-point. 
    Since $\gamma$ is horizontal with respect to the Gauss-Manin connection and $S^\dagger_\IC$ is connected, it suffices to show that the fiber $\gamma_s$ of $\gamma$ at $s$ is the identity on $\bP_{\dR, s} = \PH^2_\dR(\sX_s/\IC)$. 

    We first claim that, if $z \in \bP_{B, s} \cap \Fil^1 \bP_{\dR, s}$, then $\gamma_s(z) = z$. 
    Lefschetz theorem on $(1, 1)$ classes tells us that $z$ comes from a line bundle $\xi \in \NS(\sX_s)_\IQ$.
    On the other hand, $\alpha_B(z) \in \bL_{B, s}$ is a Hodge class, and correponds to an element of $\zeta \in \LEnd(\sA_s)_\IQ$. 
    Under the isomorphism $\tau$, the triple $(s, \xi, \zeta)$ descends to a finite extension $K$ of $\IQ_p$, i.e., there exists $s_0 \in S(K)$, $\xi_0 \in \NS(\sX_{s_0})_\IQ$ and $\zeta_0 \in \LEnd(\sA_{s_0})_\IQ$ such that $(s, \xi, \zeta) = (s_0, \xi_0, \zeta_0) \tensor_K \IC$. 
    Under the de Rham comparison isomorphism, $c_{1, p}(\xi)$ is sent to $c_{1, \dR}(\xi)$ \cite[Prop.~11.2]{IIK}. 
    The same is true for the cycle class of $\zeta_0$. 
    Note that there is a commutative diagram \[\begin{tikzcd}
	{\bL_{p, s}(-1) \tensor B_\dR} & {\bP_{p, s} \tensor B_\dR} \\
	{\bL_{\dR, s_0}(-1) \tensor_K B_\dR} & {\bP_{\dR, s_0} \tensor_K B_\dR}
	\arrow["{\alpha_{p, s} \tensor 1}", from=1-1, to=1-2]
	\arrow["{\sc_\dR}"', from=1-1, to=2-1]
	\arrow["{\hat{\alpha}_{\dR, s_0} \tensor 1}"', from=2-1, to=2-2]
	\arrow["{\sc_\dR}", from=1-2, to=2-2].
    \end{tikzcd}\]
    As the isomorphism $\alpha_{B, s}$ sends $c_{1, p}(\xi)$ to $\mathrm{cl}_{p}(\zeta)$, by the compatibility $\alpha_{p, s} = \alpha_{B, s} \tensor \IZ_p$, we know that the map $\alpha_{p, s}$ sends $c_{1, p}(\xi)$ to $\mathrm{cl}_{p}(\zeta)$ as well.
    Therefore, we must have that $\hat{\alpha}_{\dR, s_0}$ sends $c_{1, \dR}(\xi_0)$ to $\mathrm{cl}_{\dR}(\zeta_0)$. This implies that $\hat{\alpha}_{\dR, s}$ sends $c_{1, \dR}(\xi)$ to $\mathrm{cl}_\dR(\zeta)$, which does the the same as $\alpha^\dagger_{\dR, s}$. This proves the claim. 

    Suppose that $s' \in S_\IC$ is another $\IC$-point, and $z' \in \bP_{\dR, s'}$ is obtained from $z$ via parallel transport along some path $[0, 1] \to S^\dagger$ connecting $s$ and $s'$. 
    Then $\gamma_{s}$ fixes $z$ if and only if $\gamma_{s'}$ fixes $z'$. 
    Since $\bP_{B, s} \tensor \IC = \bP_{\dR, s}$, there exists a basis $z_1, \cdots, z_m \in \bP_{\dR, s}$ such that each $z_i \in \bP_{B, s}$. 
    By \cref{assump: period}.(ii), Lem.~\ref{lem: Voisin} is applicable to $\sX|_{S^\dagger}$.  
    Hence for every $z_i$, there exists a point $s_i$ near $s$ such that parallel transport along some path connecting $s$ and $s_i$ sends $z$ to a Hodge $(1, 1)$-class. Therefore, we may conclude by the observation in the preceeding paragraph.  
\end{proof}

\begin{proposition}
\label{prop: crystalline period} 
    In the setting of Thm~\ref{thm: period}, there exist an isometry of F-isocrystals $\alpha_\crys : \bL_\crys(-1)[1/p]|_{\sS_{\IF_p}} \sto \bP_\crys[1/p]|_{\sS_{\IF_p}}$ which satisfies the following properties: 
    \begin{enumerate}[label=\upshape{(\alph*)}]
       \item The isomorphism between flat vector bundles given by evaluating $\alpha_\crys$ on the rigid-generic fiber of the $p$-adic completion of $\sS_{\IZ_p}$ agrees with $\hat{\alpha}_\dR$ under the crystalline-de Rham comparison isomorphism.  
       \item For any isomorphism $\tau : \bar{\IQ}_p \iso \IC$, the restriction of $\alpha_\dR$ to $S_\IC \stackrel{\tau}{\simeq} S_{\bar{\IQ}_p}$ agrees with $\alpha^\dagger_{\dR, \IC}$ on $S^\dagger$. 
       \item If for every closed point $s \in \sS_{\IF_p}$ and $i = 2, 3$, $\H^i_\crys(\sX_s/W(k(s)))$ is $p$-torsion-free, then $\alpha_\crys$ restricts to an isometry of locally free F-crystals $\bL_\crys(-1)|_{\sS_{\IF_p}} \sto \bP_\crys|_{\sS_{\IF_p}}$.
   \end{enumerate}
\end{proposition}
\begin{proof}
    To construct $\alpha_\crys$, one can apply the crystalline Riemann--Hilbert functor $\sE_{\crys}$ defined recently in \cite[Thm.~1.10]{GY24} (or the $\ID_\crys$ functor in \cite[Def.~3.12]{TT19} as here the $p$-adic base field is unramified). 
    By \cite[Thm.~1.10.(iv)]{GY24} these functors are compatible with Liu--Zhu's $\ID_\dR$ functor, which gives (a). Part (b) was proved in \cref{lem: alpha is de Rham}.

    It remains to show (c). 
    Before the argument, we note that by \cref{lem:local_free_of_crys_coh}, $R^2 f_{\IF_p, \crys*} \sO_{\sX_{\IF_p}/\IZ_p}$ is locally free. 
    Hence we can define the \emph{integral} crystalline F-crystal $\bP_\crys$.
    By taking the relative de Rham cohomology sheaf for the morphism $f$, we also obtain the integral flat connection $\bP_\dR$.
    Now we first show that the isomorphism $\hat{\alpha}_\dR$ extends over $\sS_{\IZ_p}$. 
    In other words, there exists an (necessarily unique) isomorphism $\alpha_\dR : \bL_\dR(-1) |_{\sS_{\IZ_p}} \sto \bP_\dR |_{\sS_{\IZ_p}}$ whose restriction to $S_{\IQ_p}$ is $\hat{\alpha}_\dR$. 
    Let $U = \Spec(A) \subseteq \sS_{W}$ be any affine open on which both $\bP_\dR|_{\sS_{\IZ_p}}$ and $\bL_\dR(-1)|_{\sS_{\IZ_p}}$ are trivialized. 
    Then over $U$, $\hat{\alpha}_\dR$ is given by an element $M \in \mathrm{GL}_m(A[1/p])$. 
    Our goal is to show that the matrix coefficients are integral, i.e., $M \in \mathrm{GL}_m(A)$. 
    By the same argument as in \cite[Prop.~5.11]{MPTate}, it suffices to show the following: For every closed point $s \in \sS_{\IF_p}$ with residue field $k \colonequals k(s)$, if we set $K_0 = W[1/p]$, then for every $\wt{s} \in \sS(W(k))$ lifting $s$ and $t \colonequals \wt{s}[1/p]$, the isomorphism $\alpha_{t} : \bL_\dR(-1)|_{t} \sto \bP_\dR|_{t}$ sends the $W$-lattice $\bL_\dR(-1)|_{\wt{s}}$ isomorphically onto $\bP_\dR|_{\wt{s}}$.

    The above claim follows from the recent advance in integral $p$-adic Hodge theory: Let $\mathfrak{S}$ be the ring $W[\![u]\!]$ and view $W$ as a $\fS$-algebra via $u \mapsto 0$. 
    Let $\mathrm{BK}(-)$ denote the Breuil--Kisin functor associates to each $\IZ_p$-linear crystalline $\Gal_K$-representation a finite free $\mathfrak{S}$-module in \cite[Thm~4.4]{BMS}. 
    Then the F-isocrystal $\mathrm{BK}(\H^2_\et(\sX_{\bar{t}}, \IZ_p)) \tensor_{\fS} K$ is naturally identified with $(\H^2_\et(\sX_{\bar{t}}, \IZ_p) \tensor B_\crys)^{\Gal_K}$, i.e., $\ID_\crys(\H^2_\et(\sX_{\bar{s}}, \IQ_p))$. 
    By the torsionfree assumption on crystalline cohomology, we know from \cite[Thm.\ 14.6.(iii)]{BMS} that under $\sc_\crys : \ID_\crys(\H^2_\et(\sX_{\bar{s}}, \IQ_p)) \simeq \H^2_\crys(\sX_{s_0}/W) \tensor K$, the image of the $W$-lattice $\mathrm{BK}(\H^2_\et(\sX_{\bar{t}}, \IZ_p)) \tensor_{\fS} W$ is precisely $\H^2_\crys(\sX_{s}/W)$. 
    By restriction, $\bP_{\crys, s} \subseteq \H^2_\crys(\sX_s/W)$ is the image of $\mathrm{BK}(\bP_{p, \bar{t}}) \tensor_{\fS} W$. 
    Likewise, $\bL_{\crys, s}$ is recovered by $\mathrm{BK}(\bL_{p, \bar{t}}) \tensor_\fS W$ inside $\ID_\crys(\bL_{p, \bar{t}})$ (where the torsion-freeness of $\bL_\crys$ follows from \cite[Prop.~4.7(iv)]{CSpin}).
    Since $\hat{\alpha}_\dR|_{t} = \ID_\dR(\alpha_p|_t)$, where the latter by assumption is further isomorphic to $\ID_\crys(\alpha_p|_t) \simeq \mathrm{BK}(\alpha_p|_t) \otimes_{\fS} W [1/p]$, and since $\alpha_p|_t$ sends the $\IZ_p$-lattice $\bL_{p, \bar{t}}(-1)$ isomorphically onto $\bP_{p, \bar{t}}$, we have that $\hat{\alpha}_{\dR}|_t \simeq \ID_\crys(\alpha_p|_t)$ sends $\bL_{\crys, s}(-1)$ to $\bP_{\crys, s}$. 

    Via the crystalline-de Rham comparison isomorphism, there exists globally an isometry of crystals $\alpha_{\crys}$ whose evaluation on the $p$-adic completion of $\sS_{W}$ agrees with that of $\alpha_\dR$. After inverting $p$, it agrees with the isometry of F-isocrystals defined in the first paragraph. 
\end{proof}

\section{Proofs of theorems}
\label{sec: proofs of theorems}
In this section, we state and prove the precise form of \cref{intro_main_thm} (see \cref{thm: general theorem}) and then specialize it to prove \cref{thm: BSD} and \cref{thm: general type}. 

To deal with the singular fibers of the given family of surfaces $f:\mathcal{X}\to \sfM$, we introduce the notion of \emph{extended period morphism} in \Cref{def: extended period}, extending the Kuga--Sakate construction which only exists within the smooth locus.
Then we prove that for any point $t\in \sfM$ and a special endomorphism $\zeta$ of the Kuga--Satake abelian variety of $t$, there is always an extended period morphism through $t$ such that $\zeta$ deforms to the locus of milder singularities.

\subsection{A general statement}
\label{sec: general}

Recall that for a proper morphism $f : S \to T$ between schemes, the discriminant locus $\mathrm{Disc}(f)$ is the closed subset $\{ t \in T \mid S_t \textrm{ is singular} \}$, equipped with reduced subscheme structure. 
Note that the formation of the underlying closed subset of $\mathrm{Disc}(f)$ commutes with arbitrary base change, but it may become non-reduced.

\begin{set-up} \label{set-up: surfaces}
    Let $p > 2$ be a prime and $\sfM$ be a connected scheme which is smooth over $\IZ_{(p)}$. 
    Let $f : \sX \to \sfM$ be a projective and generically smooth morphism with geometrically irreducible fibers of dimension $2$. Assume the following conditions:   
    \begin{enumerate}[label=\upshape{(\alph*)}]
        \item \emph{Generic fiber:}  There exists a lattice $\Lambda \subseteq \NS(\sX_\eta)_{\mathrm{tf}}$, where $\eta$ is the generic point of $\sfM$, such that $p \nmid |\Lambda^\vee / \Lambda|$ and $\Lambda$ contains the class of an ample line bundle. 
        \item \emph{Compactification of the base:} There exists a smooth scheme $\sfB$ over $\IZ_{(p)}$ and a smooth projective morphism $\pi : \overline{\sfM} \to \sfB$ such that $\sfM$ embeds as an open subscheme into $\overline{\sfM}$. 
        \item \emph{Complement of the base:} The complement $\sfH$ of $\sfM$ in $\overline{\sfM}$, equipped with the reduced scheme structure, has an open subscheme $\sfH^+ \subseteq \sfH$ which is smooth over $\IZ_{(p)}$ such that for every $b \in \sfB$ and every irreducible component $H$ of $\sfH_b$ of codimension $1$ in $\overline{\sfM}_b$, the intersection $H \cap \sfH^+_{b}$ is dense in $H$. 
        \item \emph{Geometric genus and Kodaira--Spencer map:} 
        There exists a point $s \in \sfM(\IC)$ such that $h^{2,0}(\sX_s) = 1$, and the Kodaira--Spencer map $\nabla_s : T_s \sfM_\IC \to \Hom(\H^1(\Ohm^1_{\sX_s}), \H^2(\sO_{\sX_s}))$ has rank $> 1$ and does not vanish on the subspace $T_{s} \sfB_{\pi(s)} \subseteq T_s \sfM_\IC$ (e.g., when $\mathrm{rank\,} \nabla_s > \mathrm{min}(1, \dim \sfB_\IC)$). 
        \item \emph{Discriminant locus:} The discriminant locus $\sfD \colonequals \mathrm{Disc}(f) \subseteq \sfM$ is a proper subscheme of $\sfM$ and there exists an open subscheme $\sfD^+ \subseteq \sfD$ which is smooth over $\IZ_{(p)}$ such that for every $b \in \sfB$ and for every irreducible component $D$ of $\sfD_b$ of codimension $1$ in $\sfM_b$, the intersection $D \cap \sfD^+_{b}$ is dense in $D$. 
        \item \emph{Singularity:} For every geometric point $s \to \sfD$, the fiber $\sX_s$ has at worst rational double points.
    \end{enumerate}
\end{set-up}
Here we emphasize that most of the items in \cref{set-up: surfaces} are quite general for moduli of surfaces, except for condition (d), which is pivotal to the Kuga--Satake construction. Moreover, let us remark that by the semi-continuity theorem, any geometric fiber has $h^{2,0} \ge 1$, and hence cannot be ruled (cf. \cref{lem: specialization resolution}). 

We define $\sf{D}^-$ to be the closed subscheme $\sfD \smallsetminus \sfD^+$ with the reduced structure.
By construction, for each point $b\in B$, the intersection $\sfD^-\cap \overline{\sfM}_b$ is of codimension $\geq 2$ in $\overline{\sfM}_b$. 
In addition, we let $\sfM^\circ$ be the open subscheme $\sfM \smallsetminus \sfD$, namely the smooth locus of the given family.
Summarizing the geometric objects mentioned above, we get the following diagram of schemes:

\begin{equation}
\label{eq:diagram_of_geometry}
    \begin{tikzcd}[scale cd=.8]
	& {\sX|_{\sfM^\circ}} \\
	{\sX|_{\sfD^+}} & {\sfM^\circ=\sfM\smallsetminus \sfD} & \sX \\
	{\sfD^+} & {\sX|_\sfD} & \sfM \\
	{\sX|_{\sfD^-}} & \sfD && {\overline{\sfM}} & {\sfB/\IZ_{(p)}} \\
	{\sfD^- = \sfD \smallsetminus \sfD^+} && \sfH
	\arrow["{{\text{proper smooth}}}"', from=1-2, to=2-2]
	\arrow[hook, from=1-2, to=2-3]
	\arrow["{{\text{``mildest singularity''}}}"', from=2-1, to=3-1]
	\arrow[hook, from=2-1, to=3-2]
	\arrow[hook, from=2-2, to=3-3]
	\arrow["f"', from=2-3, to=3-3]
	\arrow[hook, from=3-1, to=4-2]
	\arrow[tail, from=3-2, to=2-3]
	\arrow["{{\text{RDP}}}"', from=3-2, to=4-2]
	\arrow[hook, from=3-3, to=4-4]
	\arrow[tail, from=4-1, to=3-2]
	\arrow["{{\text{``worse singularity''}}}"', from=4-1, to=5-1]
	\arrow[tail, from=4-2, to=3-3]
	\arrow["{{\textrm{smooth}}}", from=4-4, to=4-5]
	\arrow[tail, from=5-1, to=4-2]
	\arrow[tail, from=5-3, to=4-4]
\end{tikzcd}
\end{equation}
Here in the diagram, every $\rightarrowtail$ arrow is a closed embedding, and every $\hookrightarrow$ arrow is the open immersion of the complement of the $\rightarrowtail$ arrow right below.
The objects in the ground level are base schemes and various loci, while the objects in the second level corresponds to the restrictions of the family in question.

This section is devoted to proving the following theorem. Recall that we use $\breve{X}$ to denote the minimal resolution of a normal surface over a field, which is unique up to unique isomorphism. 
\begin{theorem}[Tate conjecture for the minimal resolution of $\sX_s$]
\label{thm: general theorem}
    Assume Set-up \ref{set-up: surfaces}. 
    Let $s \in \sfM$ be a (not necessarily closed) point with residue field $k := k(s)$, let $\bar{s}\in \sfM(\bar{k})$ be a geometric point over $s$, and assume one of the following conditions.
    \begin{enumerate}[label=\upshape{(\roman*)}]
        \item The point $s$ lies in $\sfD^+$ or $\sfM^\circ$. 
        \item\label{item: non-supersingular} $\breve{\sX}_s$ is not supersingular (i.e., the Newton polygon of $\H^2_\crys(\breve{\sX}_{s^p}/W(k^p))[1/p]$ is not a straight line.) 
        \item\label{item: equality of Hodge numbers} There is an equality of Hodge numbers $h^{i, j}(\breve{\sX}_s) = h^{i, j}(\sX_\eta)$ for every $i, j$. 
    \end{enumerate}
    Then the map 
    $$ c_1 : \NS(\breve{\sX}_{\bar{s}}) \tensor \IQ_\ell \sto \varinjlim_{H} \H^2_\et(\breve{\sX}_{\bar{s}}, \IQ_\ell(1))^{H} $$
    is an isomorphism, where $H$ runs through open subgroups of $\Gal_k$. 
\end{theorem}

\subsection{Extended period morphism}
\label{subset: extended_period_mor} 
We start with the period morphism defined on the smooth locus $\sfM^\circ$, by applying various inputs from \cref{const: sheaves on Sh}. 
We choose small enough level structures $\wt{\sfK}^p \subseteq \wt{G}(\IA^p_f)$ and $\sK^p \subseteq \mathrm{GSp}(\IA^p_f)$ such that $\Sh_{\sfK}(\wt{G}) \to \Sh_{\sK}(\GSp)$ is a closed embedding.
We then apply \cref{thm: period} to obtain the Kuga--Satake period morphism 
\[
\rho : \wt{\sfM}^\circ \to \shS_{\wt{\sfK}}(\wt{G}),
\]
where $\wt{\sfM}^\circ$ is a connected, finite \'etale cover over $\sfM^\circ=\sfM\smallsetminus \sfD$. 
Let $\bP_?$ for $?\in \{B, \ell, p, \dR, \crys\}$ denote the coefficients sheaves on various fibers of $\sfM^\circ$ given by primitive cohomologies as in \cref{sub: period_setup}. 
Then by \Cref{thm: period}, \Cref{cor: DdR on Sh} and \Cref{prop: crystalline period}, we know the coefficients sheaves $\bP_?$ are naturally isomorphic to the pullbacks $\rho^*\bL(-1)_?$ for corresponding automorphic sheaves on the Shimura variety $\shS_{\wt{\sfK}}(\wt{G})$.
We will simply write $\shS$ for $\shS_{\wt{\sfK}}(\wt{G})$.

To eventually deal with the singular fiber in the given family, we will need to extend the period morphism to treat the points in the singular locus.
This leads to the following notion that will play a key role in the entire analysis.

\begin{definition}
\label{def: extended period}
We define an \textbf{extended period morphism} to be a map $\tau:T\to \shS$, where
\begin{itemize}
    \item $T$ is a normal integral $\sfM$-scheme such that the induced family $\sX|_T\to T$ is Zariski-locally resolvable;
    \item the map $\tau$ fits into the following commutative diagram
    \begin{equation}
        \label{eqn: extended period}
        \begin{tikzcd}
	T^\circ & T \\
	{\wt{\sfM}^\circ} & {\sfM} \\
	&& \shS,
	\arrow[hook, from=1-1, to=1-2]
	\arrow[from=1-1, to=2-1]
	\arrow[from=1-2, to=2-2]
	\arrow["\tau", curve={height=-12pt}, from=1-2, to=3-3]
	\arrow[from=2-1, to=2-2]
	\arrow["\rho", curve={height=12pt}, from=2-1, to=3-3]
        \end{tikzcd}
    \end{equation}
    where $T^\circ \into T$ is the open inclusion map of an open dense subscheme.
\end{itemize}
\end{definition}

One of the crucial properties of the extended period morphism is that various cohomology sheaves over $\wt{\sfM}^\circ$ and their matching isomorphisms with automorphic sheaves can be extended to the entire scheme $T$.
Concretely, we have the following constructions.
\begin{construction}[Extension of cofficients and matching isomorphisms]
\label{const: extend comparison}
Let $\tau:T\to \shS$ be an extended period morphism.
By applying the extension results from \cref{lem: coh of Zariski-resolvable} and \cref{const: VHS and crystals resolved} to the primitive cohomology over $T^\circ$, 
we can obtain the matching isomorphism of coefficients over the entire scheme $T$. 
\begin{itemize}
    \item ($\ell$-adic coefficients) By \textit{loc. cit.} the pullback $\bP_\ell|_{T^\circ}$ uniquely extends to $T$ and we denote the extension by $\breve{\bP}_{\ell, T}$. 
By the full-faithfulness of the restriction functor from the category of \'etale local systems on $T$ to those on $T^\circ$, the isomorphism $\alpha_{\ell}|_{T^\circ} : \bL_\ell(-1)|_{T^\circ} \sto \bP_\ell|_{T^\circ}$ extends to an isomorphism $\tau^* \bL_\ell(-1) = \bL_\ell(-1)|_T \sto \breve{\bP}_{\ell, T} $, which we shall denote by $\alpha_{\ell, T}$.  
\item (Variations of Hodge structures) Similarly, when $T$ is a smooth variety over $\IC$, we obtain the canonically extended variation of Hodge structures $(\breve{\bP}_{B, T}, \breve{\bP}_{\dR, T})$ over $T$.
If additionally $T^\circ$ maps to $(\wt{\sfM}^\circ)^\dagger$ (the connected component of $b$ in $\wt{\sfM}^\circ$), then we can extend the isomorphisms $(\alpha^\dagger_B, \alpha^\dagger_\dR)|_{T^\circ}$ to $(\alpha^\dagger_{B, T}, \alpha^\dagger_{\dR, T}) : \tau^*(\bL_B(-1), \bL_{\dR, \IC}(-1)) \sto  (\breve{\bP}_{B, T}, \breve{\bP}_{\dR, T})$.

\item (F-(iso)crystals) If $T$ is a smooth variety over a perfect field $\kappa$ of characteristic $p$, the primitive crystalline cohomology of the resolved family defines an F-isocrystal $\breve{\bP}_{\crys, T}[1/p]$, 
and $\alpha_{\crys}[1/p]|_{T^\circ}$ extends to $\alpha_{\crys, T}[1/p]$.
If in addition the crystalline cohomology sheaf $\breve{R}^2 (f_T)_{\crys, *} \sO_{\sX_T/W(\kappa)}$ is locally free, then we may also define the F-crystal $\breve{\bP}_{\crys, T}$ over $T$, and extend the isomorphism $\alpha_{\crys}|_{T^\circ}$ to $\alpha_{\crys, T}$ (\Cref{lem: rigidity of F-crystals}). 
\end{itemize}
\end{construction}

\begin{lemma}
\label{lem: NOS}
    Let $F$ be any discrete valuation field and $\ell$ be a prime invertible in $\sO_F$. Then every $s \in \shS(\wt{G})(F)$ extends to an $\sO_F$-point if and only if the inertia subgroup of $\Gal_F$ has finite image in $\mathrm{SO}(\bL_{\ell, \bar{s}})$.   
\end{lemma}
\begin{proof}
    This follows from an argument of Deligne (\cite[\S6.6]{DelK3Weil}, cf. \cite[Lem.~9.3.1]{Andre}). 
    By considering the Siegel embedding of $\shS$ and by the N\'eron--Ogg--Shafarevich criterion, it suffices to show that $\sA_s$ has good reduction. 
    Let $I \subseteq \Gal_F$ be the inertia group. 
    As the local monodromy is quasi-unipotent, the fact that $I$ has finite image in $\SO(\bL_{\ell, \bar{s}})$ implies that its image in $\CSpin(\bL_{\ell, \bar{s}})$ is also finite. 
    However, since we assumed level structure for $\shS$ is sufficiently small (in particular it is neat), this image has to be trivial. Hence we may conclude by the N\'eron--Ogg--Shafarevich criterion. 
\end{proof}

\begin{corollary}
\label{cor: extend DVR}
    Let $T = \Spec(R)$ for a discrete valuation ring $R$ and let $\eta$ be its generic point. Whenever there is a commutative diagram
    \[\begin{tikzcd}
	\eta & T \\
	{\wt{\sfM}^\circ} & \sfM,
	\arrow[hook, from=1-1, to=1-2]
	\arrow[from=1-1, to=2-1]
	\arrow[from=1-2, to=2-2]
	\arrow[from=2-1, to=2-2]
\end{tikzcd}\]
    the composite $\eta \to \wt{\sfM}^\circ \to \shS$ extends to a morphism $\tau : T \to \shS$. 
\end{corollary}
\begin{proof}
    By the existence of Artin--Brieskorn resolution of $\sX_T$ over a finite ramified extension of $T$ (\cite[Thm.~2]{Artin-Res}), the action of $\Gal_{k(\eta)}$ on $\H^2_\et(\sX_{\bar{\eta}}, \IQ_\ell)$ factors through a finite quotient, and hence the Galois action on $\bL_{\ell, \tau(\bar{\eta})}$ is finite as well. 
    Then we conclude by \Cref{lem: NOS}. 
\end{proof}
Another useful observation is that for a relative curve with \'etale boundary in a compactification, to check the period morphism extends, it suffices to do so on the mod $p$ reduction.
This will be used in conjunction with \Cref{prop:lift_closed_immersion}.
\begin{proposition}[Extension criterion in mixed characteristic]
\label{lem: extend morphism from curve}
    Let $k$ be a perfect field of characteristic $p$. 
    Let $\overline{T}$ be a smooth proper curve over $W=W(k)$, and let $D, H \subseteq \overline{T}$ be disjoint relative effective Cartier divisors that are \'etale over $W$. 
    Set $T \colonequals \overline{T} \smallsetminus H$ and $T^\circ \colonequals T \smallsetminus D$. 
    Then a morphism $\tau^\circ : T^\circ \to \shS_W$ extends over $T$ if and only if its special fiber $\tau^\circ_k$ extends over $T_k$. 
\end{proposition}
\begin{proof}
    Set $K \colonequals W[1/p]$. 
    Let $a$ and $b$ be geometric points on the special and generic fibers of $T^\circ$ respectively such that $b$ specializes to $a$ along a path. 
    As $\overline{T}$ is a compactification of both $T$ and $T^\circ$ whose boundaries are normal crossing divisors, we know by the relative Abhyankar's lemma in \cref{prop:relative_Abhy} that $(\tau^\circ_k)^* \bL_\ell$ is tamified ramified along $(\overline{T}\smallsetminus T)_k$ and $(\overline{T}\smallsetminus T^\circ)_k$.
    Moreover, there are specialization maps of tame fundamental groups $\pi_1^t(T^\circ_{\bar{K}}, b) \to \pi_1^t(T^\circ_{\bar{k}}, a)$ and  $\pi_1^t(T_{\bar{K}}, b) \to \pi_1^t(T_{\bar{k}}, a)$ (\cite[XIII, \S2.10]{SGA1}). 
    The fact that $\tau_k^\circ$ extends over $T_k$ implies that $(\tau^\circ_k)^* \bL_\ell$ is unramified over $T_k$. 
    This in turn implies that $(\tau^\circ_{\bar{K}})^* \bL_\ell$ is unramified over $T_{\bar{K}}$. 
    Indeed, the situation is that we have a commutative diagram below with solid arrows 
    \begin{equation}
        \begin{tikzcd}
	{\pi_1^t(T^\circ_{\bar{K}}, b)} & {\pi_1^t(T_{\bar{K}}, b)} & {\mathrm{SO}(\bL_{\ell, \tau^\circ(b)})} \\
	{\pi_1^t(T^\circ_{\bar{k}}, a)} & {\pi_1^t(T_{\bar{k}}, a)} & {\mathrm{SO}(\bL_{\ell, \tau^\circ(a)})}
	\arrow[from=1-1, to=1-2]
	\arrow[curve={height=-18pt}, from=1-1, to=1-3]
	\arrow[from=1-1, to=2-1]
	\arrow[dashed, from=1-2, to=1-3]
	\arrow[from=1-2, to=2-2]
	\arrow["\simeq"', from=1-3, to=2-3]
	\arrow[from=2-1, to=2-2]
	\arrow[curve={height=18pt}, from=2-1, to=2-3]
	\arrow[from=2-2, to=2-3]
    \end{tikzcd},
    \end{equation}
    and we want the dashed arrow to exist, which is true thanks to the commutativity of the rest of the diagram. 
    
    Now, by \cref{lem: NOS}, the map $\tau^\circ_{\bar{K}}$ extends to the entire generic fiber $T_{\bar{K}}$, by the equality $\pi_1^t(T^\circ_{\bar{K}}, b) = \pi_1(T^\circ_{\bar{K}}, b)$ and the existence of the dotted arrow. 
    Since the extended map $\tau_{\bar{K}} : T_{\bar{K}} \to \shS_{\bar{K}}$ descends to $K$ when restricted to an open dense subscheme $T^\circ_{\bar{K}}$, we must have that $\tau_{\bar{K}}$ itself descends to $K$ as well, giving rise to an extension of the original morphism $\tau^\circ_K : T^\circ_K \to \shS_K$ over $T_K$. 
    This in turn implies that $(\tau^\circ_K)^* \bL_\ell$ over $T_K^\circ$ is unramified over $T_K$. As $(\tau_K)^* \bL_\ell$ extends to the generic point of the special fiber $T_k$, the purity of $\ell$-adic local systems tells us that $(\tau_K)^* \bL_\ell$ is unramified over $T$. 
    Hence by the extension property of the canonical integral model $\shS$ (\cite{KisinInt}, cf. \cite[Thm.~3.4.4]{HYZ}) again, the map $\tau_K$ extends to a map of schemes $\tau : T \to \shS$,
    which agrees with the map $\tau^\circ$ when restricted to $T^\circ$.
\end{proof}

\subsection{Motivic alignment}
\label{sub: mot_align}
In this subsection, we introduce the notion of motivic alignment for extended period morphisms.
It relates the divisors of the given surface $\sX_t$ with the special endomorphisms of the Kuga--Satake abelian variety $\sA_{\tau(t)}$, for points $t\in T$ in the given family.
\begin{definition}
\label{def: motivically aligned}
    Let $\tau : T \to \shS$ be an extended period morphism and $t \to T$ be a geometric point. For every $\zeta \in \LEnd(\sA_{\tau(t)})_\IQ$, we say that the pair $(t, \zeta)$ is \textbf{motivically aligned} if for every $\ell \neq \mathrm{char\,}k(t)$ the isomorphism $\alpha_{\ell, T}|_t :   \bL_{\ell, \tau(t)}(-1) \sto \bP_{\ell, t}$ sends the cycle class of $\zeta$ into the the first Chern class of some element in $\PNS(\breve{\sX}_t)_\IQ$; if $\mathrm{char\,}k(t) = p$, then we in addition require that (after replacing $k(t)$ by its perfection) $\alpha_{\crys, t}$ sends the crystalline cycle class of $\zeta$ to the crystalline first Chern class of some element in $\PNS(\breve{\sX}_t)_\IQ$. 
    Moreover, 
    \begin{itemize}
        \item we say that $t$ is motivically aligned if the pair $(t, \zeta)$ is motivically aligned for every $\zeta$; 
        \item we say that a field-valued point $t'$ is  motivically aligned if some (and hence every) geometric point over $t'$ is motivically aligned;
        \item we say that $\tau$ is motivically aligned if it is motivically aligned at every point. 
    \end{itemize}
\end{definition}
\begin{remark}[Diagram of motivic alignment]
    Let $\tau : T \to \shS$ be an extended period morphism, let $t\in T$ be a motivically aligned point, and let $\bar{t}$ be a geometric point over $t$.
    By the Kummer sequence of the $\ell$-adic cohomology, we deduce that the Chern class map defines an injection $c_1:\PNS(\breve{\sX}_t)_\IQ \to \PH^2_\et(\breve{\sX}_{\bar{t}},\IQ_\ell(1))=\breve{\bP}_{\ell,\bar{t}}$.
    Similarly, the cycle class map defines an injection $\LEnd(\sA_{\tau(t)})_\IQ \to \bL_{\ell,\tau(\bar{t})}$.
    Consequently, the injectivity and motivic alignment implies a commutative diagram
    \begin{equation}
        \label{eq:diagram_of_mot_aligned}
       \begin{tikzcd}
	{\LEnd(\sA_{\bar{t}})_\IQ} & {\PNS(\breve{\sX}_t)_\IQ} \\
	{\bL_{\ell, \tau(\bar{t})}} & {\PH^2_\et(\breve{\sX}_{\bar{t}}, \IQ_\ell(1)) = \breve{\bP}_{\ell, \bar{t}}}(1),
	\arrow[from=1-1, to=1-2]
	\arrow[hook, from=1-1, to=2-1]
	\arrow["c_1",hook, from=1-2, to=2-2]
	\arrow["\sim","\alpha_{\ell, \bar{t}}"', from=2-1, to=2-2]
        \end{tikzcd}
   \end{equation}
   where the bottom arrow is the matching isomorphism for $\ell$-adic cohomology.
   We call this the \textbf{diagram of motivic alignment}. 
\end{remark}

The main reason of introducing the above notion is that, in order to extend the method of \cite{MPTate}, we translate the Tate conjecture into a problem of finding motivically aligned extended period morphism.
This is illustrated by the following statement.
\begin{lemma}
	\label{lem: reduction}
	To prove \cref{thm: general theorem}, it suffices to construct an extended period morphism $\tau : T \to \shS$ such that up to replacing $k$ by a finite extension, the point $s$ lifts to a motivically aligned $k$-point on $T$. 
\end{lemma}
\begin{proof}
	Assume such $\tau : T \to \shS$ exists, and consider the diagram (\ref{eq:diagram_of_mot_aligned}) with $t = s$.
	Then the conclusion follows from the fact that $\alpha_{\ell, \bar{s}}$ is $\Gal_k$-equivariant, and that 
	$$ \LEnd(\sA_{\bar{s}}) \tensor \IQ_\ell \to \varinjlim_H \bL_{\ell, \tau(\bar{s})}^H $$
	is an isomorphism, thanks to the Tate theorem for special endomorphisms by Madapusi Pera \cite[Thm.~6.4]{MPTate}. 
\end{proof}

For the rest of the subsection, we give several useful observations on the motivically aligned extended period morphism.

We first notice that by the Lefschetz $(1,1)$-theorem, the motivic alignedness is automatically true in characteristic zero.
This encapsulates the argument of \cite[Thm.~5.17(4)]{MPTate}.
The germ of this idea can be traced back to \cite{Nygaard} and \cite{NO}. 

\begin{proposition}
\label{prop: char 0 motivically aligned}
    Let $\tau:T\to \shS$ be an extended period morphism, and let $t\in T$ be a geometric point $t$ in characteristic $0$.
    Then $t$ is motivically aligned. 
\end{proposition}
\begin{proof}
    We may reduce to the case when $k(t) = \IC$. By resolution of singularities, we can always find a smooth $\IC$-variety $T'$ admitting a proper and birational morphism to $T_\IC$. 
    For the purpose of proving the proposition, there is no harm in replacing $T$ by $T'$ and assume that $T$ is a smooth $\IC$-variety. 
    Note that $\Aut(\IC)$ acts transitively on the set of connected components of $\wt{\sfM}^\circ_\IC$. 
    Therefore, by take the base change along automorphisms of $\IC$, we may assume that $T$ maps into $(\wt{\sfM}^\circ)^\dagger$. Recall that in \cref{const: extend comparison}, we extended the isomorphism $(\alpha^\dagger_B, \alpha^\dagger_\dR)|_{T^\circ}$ to the isomorphism $(\alpha^\dagger_{B, T}, \alpha^\dagger_{\dR, T}) : \tau^*(\bL_B, \bL_{\dR, \IC})(-1) \sto (\breve{\bP}_{B, T}, \breve{\bP}_{\dR, T})$ on the scheme $T$. 
    By the Lefschetz $(1, 1)$-theorem, $\PNS(\breve{\sX}_t)_\IQ$ can be identified with the $(0, 0)$-classes in $\PH^2_{B}(\breve{\sX}_{t}, \IQ(1))$. 
    Similarly, $\LEnd(\sA_{\tau(t)})$ can be identified with those in $\bL_{B, \tau(t)}$. 
\end{proof}


Our next observation is that the motivic alignedness is invariant under deformation. 

\begin{lemma}[Motivic alignedness is closed under specialization]
\label{lem: def invariance of MA}
    Suppose that $t', t$ are geometric points on $T$ with a chosen path $\gamma: t' \rightsquigarrow t$, 
    and $\zeta' \in \LEnd(\sA_{t'})$ specializes to $\zeta \in \LEnd(\sA_t)$.
    \begin{enumerate}[label=\upshape{(\alph*)}]
        \item If the pair $(t', \zeta')$ is motivically aligned, then so is $(t, \zeta)$.
        \item\label{item: def invariance of MA converse} The converse is also true provided that if $\mathrm{char\,} k(t') = p$, we additionally assume that for the scheme theoretic image $T' \subseteq T$ of $t'$, $\sX|_{T'}$ admits a simultaneous resolution whose total space is a projective scheme near $t$. 
    \end{enumerate} 
\end{lemma}
\begin{proof}
    The forward direction follows from the existence of the following diagrams for various specialization morphisms and their compatibility with each other. For $\ell$-adic ($\ell \in \sO_T^\times$) cohomology we have a diagram 
    \begin{equation}
        \begin{tikzcd}
	{\LEnd(\sA_{t'})_\IQ} & {\bL_{\ell_0, t'}} & {\bP_{\ell_0, t'}(1)} & {\PNS(\breve{\sX}_{t'})_\IQ} \\
	{\LEnd(\sA_{t})_\IQ} & {\bL_{\ell_0, t}} & {\bP_{\ell_0, t}(1)} & {\PNS(\breve{\sX_{t}})_\IQ}
	\arrow[from=1-1, to=1-2]
	\arrow["{\sp_\gamma}"', from=1-1, to=2-1]
	\arrow["{{{{\alpha_{\ell, t'}}}}}", from=1-2, to=1-3]
	\arrow["\gamma"', from=1-2, to=2-2]
	\arrow["\gamma", from=1-3, to=2-3]
	\arrow[from=1-4, to=1-3]
	\arrow["{{{\sp_\gamma}}}", from=1-4, to=2-4]
	\arrow[from=2-1, to=2-2]
	\arrow["{{{{\alpha_{\ell, t}}}}}", from=2-2, to=2-3]
	\arrow[from=2-4, to=2-3]
\end{tikzcd}
    \end{equation}
    Here the left and the right diagrams commute because the specialization of cycles is compatible with that of $\ell$-adic cohomology. That the central diagram commutes is a tautological consequence of the existence of a global matching isomorphism $\alpha_\ell$.  
    When $\mathrm{char\,} k(t') = p$, the analogous statement for crystalline cohomology follows from a similar argument and we omit the details. 
    If $\mathrm{char\,} k(t') = 0$ and $\mathrm{char\,} k(t) = p$, then apply the crystalline comparison theorem and its compatibility with cycle class maps (cf. \cite[Cor.~11.6]{IIK}). 

    Proving the converse amounts to deforming the line bundles. 
    By \cref{prop: char 0 motivically aligned}, it suffices to consider the case when both $t$ and $t'$ are in characteristic $p$. 
    As the statement would not change if we replace the discrete valuation ring by its completion, we reduce to the case when the path $\gamma$ is defined by the ring $k[\![x]\!]$. 
    Let $\zeta \in \LEnd(\sA_{t})_\IQ$ be an element such that $(t, \zeta)$ is motivically aligned, i.e., there exists an element $\xi \in \PNS(\sX_t)_\IQ$ whose cycle class matches that of $\zeta$ via $\alpha_{\ell, t}$. 
    We need to show that $\xi$ lifts to an element in $\PNS(\sX_{t'})_\IQ$. 
    At this point we use the isomorphism of F-isocrystals $\breve{\bP}_{\crys, T}[1/p] \simeq \tau^* \bL_\crys(-1)[1/p]$.  
    The assumption that $\zeta$ lifts to $\zeta'$ implies that the cycle class of $\zeta$ gives rise to a global section of the F-crystal $\gamma^* \bL_\crys(-1)$. This implies that $c_{1, \crys}(\xi)$ also lifts to a global section of the F-isocrystal $\breve{\bP}_{\crys}[1/p]|_\gamma$. 
    Then we conclude by \cite[Thm.~3.6]{Morrow}.\footnote{We can drop the algebraizability assumption in \ref{item: def invariance of MA converse} and use the more elementary \cite[Cor.~1.13]{OgusCrystals} when there are sufficient torsion-freeness assumptions on cohomology.}
\end{proof}

\subsection{Admissible curves}
\label{sub:adm_cur}
As we will see later, for points away from the locus $\sfD^-$, we will extend period morphisms along curves by ``taking limits'', and then deform these curves over a $p$-adic discrete valuation ring. To guarantee that the curves we choose are sufficiently amenable to deformation, we recall the following key notion from \cite[Def.~6.1.1]{HYZ}.

\begin{definition}
    \label{def: strongly dominant}
    \emph{(strongly dominant family of curves)}
    Let $S$ and $T$ be two smooth irreducible $k$-varieties, $g \colon \shC \to T$ be a smooth family of connected curves, and $\varphi\colon \shC \to S$ be a morphism with equi-dimensional fibers (i.e., there exists some integer $d \ge 0$ such that every geometric fiber of $\varphi$ is non-empty and equi-dimensional of dimension $d$). 
    Let $U$ be the maximal open subvariety of $\shC$ on which the composition of the (relative) tangent bundles $\sT(\shC/T) \into \sT \shC \to \varphi^*(\sT S)$ does not vanish, where $\sT S$ is the tangent bundle of $S$. 
    By assumption, there is then a natural morphism $U \to \IP(\sT S)$, induced by the identification $\shC \iso \IP(\sT(\shC/T))$ and the above sequence of the tangent bundles. 
    
Suppose that 
\begin{enumerate}[label=\upshape{(\roman*)}]
    \item the induced morphism $\shC \to T \times S$ is quasi-finite, 
    \item for every $k$-point $s \in S$, the fiber $U_s := U \cap \varphi^{-1}(s)$ is dense in $\varphi^{-1}(s)$; 
    \item the morphism $U \to \IP(\sT S)$ has equi-dimensional fibers. 
\end{enumerate}
Then we say that the family of curves $\shC/T$ \textbf{homogeneously dominates} $S$ (via the morphism $\varphi$). 
If there exists an open dense subvariety $T' \subseteq T$ such that the restriction $\shC|_{T'}$ homogeneously dominates $S$, then we say that $\shC/T$ \textbf{strongly dominates} $S$. 
\end{definition}

Roughly speaking, the family $\shC / T$ homogeneously dominates $S$ if there are curves parameterized by $T$ passing through every given point on $S$ in any given direction, and the sub-family of such curves has a fixed dimension. 

We make the following elementary observation on the reducedness of the intersection.
\begin{lemma}
    \label{lem: transverse intersection general}
    In the setting of \cref{def: strongly dominant}, suppose that $g : \shC \to T$ strongly dominates $S$ and $D \subseteq S$ is a reduced closed proper subvariety. 
    Then for a general $t \in T$, the fiber product $\shC_t \times_S D$ is reduced (i.e., the image of $\shC_t$ passes through $D$ transversely). 
\end{lemma}
\begin{proof}
    This follows from a simple dimension counting argument as in the proof of the Bertini theorem, so here we just give a sketch of the idea: There is no harm in assuming that $\shC$ homogeneously dominates $S$, and note that the reducedness assumption on $D$ implies that the smooth locus $D^{\mathrm{sm}}$ is open and dense. 
    
    For each $s \in S$, the subset $ T(s) := \{ t \in T \mid \shC_t \times_S \{ s \} \neq \emptyset \}$ has codimension at least $\dim S - 1$ in $T$. 
    As the singular locus $D^{\mathrm{sing}}$ of $D$ is codimension at least $2$ in $S$, $\{ t \in T(s) \mid s \in D^{\mathrm{sing}} \} $ is of codimension at least $1$ in $T$. 
    Likewise, for every subspace $V \subseteq \sT_s S$, viewed as a subvariety of $\sT S$, consider 
    $$ T(s, V) := \{ t \in T \mid U_t \times_{\IP(\sT S)} V \neq \emptyset \textrm{ or } (\shC \smallsetminus U)_t \times_S \{ s \} \neq \emptyset \}.  $$
    One readily checks that \Cref{def: strongly dominant}.(ii, iii) implies that 
    $$ \{ t \in T(s, V) \mid s \in D^{\mathrm{sm}}, V = \sT_s D^{\mathrm{sm}} \} $$
    is of codimension at least $1$ as well. 
\end{proof}

As we shall see, we will make use of complete intersection curves, and hence we make the following construction.
\begin{construction}
\label{const: strongly dominate}
Assume the notations as in \cref{set-up: surfaces}. 
    We choose an embedding of $\sfB$-schemes $\overline{\sfM} \into \IP^N_\sfB$ for large enough $N$, and let $r$ be the integer $\dim \overline{\sfM} -\dim \sfB -1$.
    Let $(d_1, \cdots, d_r)$ be a tuple of positive integers and consider the product 
     $$ \overline{\sT} = (\prod_{i = 1}^{r} |\sO_{\IP^N}(d_i)| \times \sfB). $$ 
    Then every point $t \in \overline{\sT}$ is given by a tuple of the form $([H_1], \cdots, [H_r], b)$ such that $b \in \sfB(k(t))$, and $[H_i]$ is a $k(t)$-point of $|\sO_{\IP^N_{k(b)}}(d_i)|$, which corresponds to a degree $d_i$ hypersurface $H_i$ in $\IP^N_{k(b)}$. 
    Let $\overline{\mathscr{C}} \to \overline{\sT}$ be the universal complete intersection whose fiber over each $([H_1], \cdots, [H_r], b)$ is the intersection $\overline{\sfM}_b \cap H_1 \cap \cdots \cap H_r$, which by our choice is a hypersurface in $\overline{\sfM}_b$ of dimension $\geq 1$.
    Let $\sT \subseteq \overline{\sT}$ be the open locus over which $\overline{\mathscr{C}}$ is smooth, irreducible and projective, and we write $\mathscr{C}$ for the restriction of $\overline{\mathscr{C}}$ to $\sT$.

    It is not hard to check that if $d_i$'s are chosen to be sufficiently large, then for every geometric point $b' \to \sfB$, $(\mathscr{C}/\sT)_{b'}$ is a family of curves which \emph{strongly dominates} $\overline{\sfM}_{b'}$. 
    In general, for any field-valued point $b$ on $\sfB$, we say that an irreducible smooth projective curve $C \subseteq \overline{\sfM}_b$ is an \textbf{admissible curve} if it arises as a fiber of $(\mathscr{C}/\sT)_b$ for such a family. Summarizing the construction, we get the following diagram of schemes
    \begin{equation}
    	\label{eq:diagram_of_family_of_curves}
    	\begin{tikzcd}
    		\mathscr{C} \arrow[d, "\text{family of curves}"'] \arrow[rr, "\text{strongly dominant}"] && \overline{\sfM} \ar[d] \\
    		\sT \arrow[rr] && B. 		
    	\end{tikzcd}
    \end{equation}
\end{construction}

\subsection{Locus of the mildest singularity}
In this subsection, we prove the Tate conjecture for the minimal model of the surfaces $\sX_s$ that are smooth or lie on $\sfD^+$ (those which morally have the mildest singularity).
For any such point $s$, we will construct an extended period morphism $\tau:T\to \shS$ from a curve $T$ in positive characteristic such that $s$ is contained in $T$.
Then we prove that any special endomorphism of the Kuga--Satake abelian variety of $\sX_s$ admits a mixed characteristic deformation along a lift of the curve $T$, by combining the tools developed in \cref{sec: ramified covers} and \Cref{sec:period_mor} together with a key trick in \cite{HYZ}.

As a quick preparation, we recall the ``$\lambda$-number'' introduced in \textit{loc. cit.}, and give it a name. 
\begin{definition}
\label{def: lambda number}
    Let $S$ be a noetherian integral normal scheme. 
    Let $\ell \in \sO_S^\times$ be a prime and let $\sfW_\ell$ be an \'etale $\IQ_\ell$-local system over $S$. 
    We denote by $\lambda(\sfW_\ell)$ the \textbf{monodromy invariant number} of $\sfW_\ell$, defined as the integer $\dim \varinjlim_U \sfW_{\ell, \bar{s}}^{U}$ where $\bar{s}$ is a geometric point on $S$ and $U$ runs through open subgroups of $\pi_1^\et(S, \bar{s})$.
\end{definition}
In the above definition, one may alternatively define $\lambda(\sfW_\ell)$ as $\dim \varinjlim_U \sfW_{\ell, \bar{\eta}}^{U}$, where $\eta$ is the generic point of $S$ and $U$ runs through the open subgroups of $\Gal(\bar{\eta}/\eta)$ (for a fixed geometric point $\bar{\eta}$ over $\eta$), thanks to the fact that the natural map $\Gal(\bar{\eta}/\eta) \to \pi_1^\et(S, \bar{s})$ is a surjection (\cite[\href{https://stacks.math.columbia.edu/tag/0BQM}{Tag 0BQM}]{stacks-project}.)

\begin{definition}
\label{def: good compactification}
    Let $T$ be a noetherian base scheme and $S$ be a smooth $T$-scheme of finite type. 
    Let $\ell \in \sO_T^\times$ be a prime and let $\sfW_\ell$ be an \'etale $\IQ_\ell$-local system. 
    We say that the local system $\sfW_\ell$ \textbf{has constant fiberwise monodromy invariant number} (with respect to $T$) if there exists a number $\lambda^\geo=\lambda^\geo(\sfW_\ell)$ such that for every geometric point $\bar{t} \to T$ and every connected component $S^\circ_{\bar{t}}$ of $S_{\bar{t}}$, we have $\lambda(\sfW_\ell|_{S^\circ_{\bar{t}}}) = \lambda^\geo$.
\end{definition}


Now we prove the motivic alignment for an admissible curve that is away from the worst locus $\sfD^-$.
For the convenience of discussion, in the following for any scheme $\overline{T}$ over $\overline{\sfM}$, we use $T$ to denote the fiber product $\overline{T}\times_{\overline{\sfM}} \sfM$, and use $T^\circ$ to denote the fiber product $\overline{T}\times_{\overline{\sfM}} \sfM^\circ$.
\begin{theorem}
\label{prop: good points}
    Let $k$ be a perfect field of characteristic $p$, let $b$ be a $k$-point on $\sfB$. 
    Let $\overline{C}$ be an admissible curve in $\overline{\sfM}_b$ such that it intersects $\sfD_{b} \cup \sfH_b$ transversely (i.e., the intersection is reduced) and the intersection is contained in $\sfD^+_b \cup \sfH^+_b$. Set $C := \overline{C} \cap \sfM$. 
      
    Let $C_1\to C$ be a connected, finite, and generically \'etale cover. 
    Then any extended period morphism $\tau : C_1 \to \shS_k$ is motivically aligned. 
\end{theorem}
\begin{proof}
    \textit{Step 0: Construct intermediate curve $C_0$.}
    Let $C_0^\circ$ be the unique connected component $\wt{\sfM}^\circ_k \times_{\sfM_k} C$ which contains the image of $C_1^\circ$. 
    Let $\overline{C}_0$ be the smooth compactification of $C_0^\circ$. 
    Then the morphism $C_0^\circ \to C$ extends uniquely to a morphism $\overline{C}_0 \to \overline{C}$, which further maps to $\overline{\sfM}$. 
    We let $C_0 \colonequals \overline{C}^\circ_0 \times_{\overline{\sfM}} \sfM$. 
    Then we claim that the map $C_0^\circ \to \shS_k$ extends to a morphism $\tau_0 : C_0 \to \shS_k$. 
    Indeed, if we choose a compactification $\overline{\shS}_k$ of $\shS_k$ and take $\overline{C}_1$ to be the unique smooth compactification of $C_1$, then $\tau$ will extend to a morphism $\overline{C}_1 \to \overline{\shS}_k$ which factors through $\overline{C}_0$ and send the open subscheme $C_1$ into $\shS_k$. 
    So the claim follows by noticing that $C_0$ is the image of $C_1$ in $\overline{C}_0$.

    To sum up, we obtain the commutative diagram on the left below. 
    As we will see in Step 2, we shall deform the morphisms over $W$ and obtain the diagram on the right. 
    \begin{equation}
    \label{eqn: extend C over k}
         \begin{tikzcd}
	{C_1^\circ } & {C_1} \\
	{C_0^\circ} & {C_0} \\
	{\wt{\sfM}^\circ_k} & {\sfM_k} \\
	&& {\shS_k,}
	\arrow[hook, from=1-1, to=1-2]
	\arrow[from=1-1, to=2-1]
	\arrow[from=1-2, to=2-2]
	\arrow["\tau", curve={height=-12pt}, from=1-2, to=4-3]
	\arrow[hook, from=2-1, to=2-2]
	\arrow[from=2-1, to=3-1]
	\arrow[from=2-2, to=3-2]
	\arrow["{{\tau_0}}", from=2-2, to=4-3]
	\arrow[from=3-1, to=3-2]
	\arrow["\rho_k", curve={height=12pt}, from=3-1, to=4-3]
	\arrow[from=3-2, to=4-3]
\end{tikzcd}
\quad \quad
\begin{tikzcd}
	{\sC_1^\circ } & {\sC_1} \\
	{\sC_0^\circ} & {\sC_0} \\
	{\wt{\sfM}^\circ_W} & {\sfM_W} \\
	&& {\shS_W.}
	\arrow[hook, from=1-1, to=1-2]
	\arrow[from=1-1, to=2-1]
	\arrow[from=1-2, to=2-2]
	\arrow["{\wt{\tau}}", curve={height=-12pt}, from=1-2, to=4-3]
	\arrow[hook, from=2-1, to=2-2]
	\arrow[from=2-1, to=3-1]
	\arrow[from=2-2, to=3-2]
	\arrow["{\wt{\tau}_0}", from=2-2, to=4-3]
	\arrow[from=3-1, to=3-2]
	\arrow["\rho", curve={height=12pt}, from=3-1, to=4-3]
	\arrow[from=3-2, to=4-3]
\end{tikzcd}
    \end{equation}

    \noindent \textit{Step 1: Deform the curve $\overline{C}$.} Let $K \colonequals W(k)[1/p]$ and fix an isomorphism $\iota: \overline{K} \simeq \IC$. 
		By assumption (d) in \cref{set-up: surfaces} and \cite[Lem.~6.2.1]{HYZ}, we can find a point $\wt{b} \in \sfB(W)$ lifting $b$ such that the VHS $(\bP_{B}, \bP_\dR|_{\sfM^\circ_\IC})$ over $\sfM^\circ_\IC$ is non-isotrivial on $\sfM^\circ_{\wt{b}}(\IC)$. 
		We then apply \cite[Lem.~6.2.2]{HYZ} to the family of curves $(\mathscr{C}/\sT)$ in \cref{const: strongly dominate}, we can deform $\overline{C}$ to a smooth and projective curve $\overline{\sC} \subseteq \overline{\sfM}_{\wt{b}}$ over $W$, such that $(\bP_{B}, \bP_\dR|_{\sfM^\circ_\IC})$ is non-isotrivial over $(\sC^\circ)_\IC \colonequals \sC^\circ \tensor_\iota \IC$. 
		In addition, by \cite[Prop.~5.1.4]{HYZ}
        \footnote{The assumption of a $\heartsuit$-family in \textit{loc. cit.} is explained in page 2 of \cite{HYZ}. 
        It means that the Kodaira--Spencer map is generically nonzero, or in other words the variation of Hodge structures defined by the second relative cohomology sheaf is non-isotrivial.}, we have the equality between the Picard numer and the monodromy invariant number
		\[
		\dim \PNS(\sX_{\bar{\eta}})_\IQ = \lambda(\bP_{\ell}|_{(\sC^\circ)_\IC}),
		\]
		where $\eta$ is the generic point of $\sC^\circ$. 
		
		On the other hand, combining assumptions (c, e) in \cref{set-up: surfaces} and the assumption that $\overline{C}$ intersects transversely with $\sfD_{b} \cup \sfH_b$ with the intersection contained in $\sfD^+_b \cup \sfH^+_b$, we deduce that $\overline{\sC} \cap (\sfD_b \cup \sfH_b)$ is \'etale over $W$.
		As a consequence, the complement of $\sC$ in $\overline{\sC}$ is a relative normal crossing divisor over $W$. 
		Hence by the Grothendieck's specialization theorem for the tame fundamental group \cite[Lem.~5.1.3]{HYZ}, we know $\bP_{\ell}|_{\sC^\circ}$ has constant fiberwise monodromy invariant $\lambda^{\geo}$ over $W$. \\

    \noindent\textit{Step 2: Deform the morphism $\tau:C_1 \to \shS_k$.}  
		By a result of Q. Liu in \cite[Prop.~5.2]{QingLiu}, the morphism $\overline{C}_0\to \overline{C}$ (in particular the curve $\overline{C}_0$) admits a lift $\overline{\sC}_0 \to \overline{\sC}$ over $W$.
		We then claim that the composition $\sC_0 \to \sC \into \sfM_W$ can be fitted into a commutative diagram 
			\[\begin{tikzcd}
				{\sC_0^\circ } & {\sC_0} \\
				{\wt{\sfM}^\circ_W} & \sfM_W \\
				&& {\shS},
				\arrow[from=1-1, to=1-2]
				\arrow[from=1-1, to=2-1]
				\arrow[from=1-2, to=2-2]
				\arrow["\wt{\tau}", curve={height=-12pt}, from=1-2, to=3-3]
				\arrow[from=2-1, to=2-2]
				\arrow["\rho", curve={height=12pt}, from=2-1, to=3-3]
			\end{tikzcd}\]
			which is part of the desired diagram on the right of (\ref{eqn: extend C over k}). 
		We first consider the inner square.
		Indeed, by a consequence of the relative Abhyankar's Lemma as in \cref{prop:etaleness_of_pullback}, the reduced subschemes of the pullback of $\overline{\sC}_0 \cap \sfD$ and $\overline{\sC}_0 \cap \sfH$ to $\overline{\sC}_0$ are both \'etale over $W$. 
		Therefore, the scheme $\overline{\sC}_1$ is a compactification of both $\sC_0$ and $\sC^\circ_0$ whose boundaries are relative normal crossing divisors over $W$. 
		Thus by the extension criterion \cref{prop:lift_closed_immersion}, 
		the closed immersion $\sC^\circ \into \sfM^\circ_{\wt{b}}$ extends to a map $\sC_1^\circ \to \wt{\sfM}^\circ_{\wt{b}}$.
        
		Now, observe that the special fiber of the latter map $(\sC^\circ_0)_k \to \shS_k$ extends over $(\sC_0)_k$ thanks to \cref{cor: extend DVR}. 
		Finally, we may apply \cref{lem: extend morphism from curve} to the case $\overline{T} = \overline{\sC}_0$ and deduce that $\tau_0$ extend over $\sC_0$. 
        Since the morphism $\overline{C}_1 \to \overline{C}_0$ extending $C_1 \to C_0$ is generically \'etale, we may apply \cite[Prop.~5.2]{QingLiu} again to deform $\overline{C}_1 \to \overline{C}_0$ to $\overline{\sC}_1 \to \overline{\sC}_0$, where $\overline{\sC}_1$ is a deformation of $\overline{C}_1$ over $W$. This way, we obtain the diagram on the right of (\ref{eqn: extend C over k}). 

        To help visualize the subsequent steps, we draw the following picture for $\sC_1$.

\tikzset{every picture/.style={line width=0.75pt}} 
\begin{center}
    \scalebox{0.8}{
\begin{tikzpicture}[x=0.75pt,y=0.75pt,yscale=-1,xscale=1]

\draw    (257.65,73.71) -- (200.61,112.29) ;
\draw    (472.37,132.86) -- (415.32,171.43) ;
\draw [line width=0.75]    (257.65,73.71) .. controls (289.35,48) and (440.68,158.57) .. (472.37,132.86) ;
\draw    (200.61,112.29) .. controls (232.3,86.57) and (383.63,197.14) .. (415.32,171.43) ;
\draw [line width=0.75]    (229.13,93) .. controls (260.82,67.29) and (412.15,177.86) .. (443.85,152.14) ;
\draw [color={rgb, 255:red, 0; green, 0; blue, 0 }  ,draw opacity=1 ] [dash pattern={on 4.5pt off 4.5pt}]  (245.77,114.86) -- (298.85,76.29) ;
\draw  [dash pattern={on 4.5pt off 4.5pt}]  (285.38,130.29) -- (338.47,91.71) ;
\draw  [dash pattern={on 4.5pt off 4.5pt}]  (315.49,144) -- (368.58,105.43) ;
\draw  [dash pattern={on 4.5pt off 4.5pt}]  (352.64,159.07) -- (405.73,120.5)(354.4,161.5) -- (407.49,122.93) ;
\draw [color={rgb, 255:red, 0; green, 0; blue, 0 }  ,draw opacity=1 ] [dash pattern={on 0.84pt off 2.51pt}]  (188.72,168.86) .. controls (212.37,133.04) and (256.42,166.8) .. (274.81,101.28) ;
\draw [shift={(275.08,100.29)}, rotate = 105.25] [color={rgb, 255:red, 0; green, 0; blue, 0 }  ,draw opacity=1 ][line width=0.75]    (10.93,-3.29) .. controls (6.95,-1.4) and (3.31,-0.3) .. (0,0) .. controls (3.31,0.3) and (6.95,1.4) .. (10.93,3.29)   ;
\draw [color={rgb, 255:red, 0; green, 0; blue, 0 }  ,draw opacity=1 ] [dash pattern={on 0.84pt off 2.51pt}]  (188.72,168.86) .. controls (212.37,133.04) and (296.42,182.93) .. (315.21,117.57) ;
\draw [shift={(315.49,116.57)}, rotate = 105.25] [color={rgb, 255:red, 0; green, 0; blue, 0 }  ,draw opacity=1 ][line width=0.75]    (10.93,-3.29) .. controls (6.95,-1.4) and (3.31,-0.3) .. (0,0) .. controls (3.31,0.3) and (6.95,1.4) .. (10.93,3.29)   ;
\draw [color={rgb, 255:red, 0; green, 0; blue, 0 }  ,draw opacity=1 ] [dash pattern={on 0.84pt off 2.51pt}]  (184.76,171.43) .. controls (208.41,135.61) and (323.84,194.83) .. (342.94,129.56) ;
\draw [shift={(343.22,128.57)}, rotate = 105.25] [color={rgb, 255:red, 0; green, 0; blue, 0 }  ,draw opacity=1 ][line width=0.75]    (10.93,-3.29) .. controls (6.95,-1.4) and (3.31,-0.3) .. (0,0) .. controls (3.31,0.3) and (6.95,1.4) .. (10.93,3.29)   ;
\draw    (451.77,165.43) -- (489.51,179.31) ;
\draw [shift={(491.38,180)}, rotate = 200.19] [color={rgb, 255:red, 0; green, 0; blue, 0 }  ][line width=0.75]    (10.93,-3.29) .. controls (6.95,-1.4) and (3.31,-0.3) .. (0,0) .. controls (3.31,0.3) and (6.95,1.4) .. (10.93,3.29)   ;
\draw    (531,163.71) -- (473.95,202.29) ;
\draw [color={rgb, 255:red, 0; green, 0; blue, 0 }  ,draw opacity=1 ][line width=0.75]  [dash pattern={on 0.75pt off 0.75pt}]  (219.62,111.43) .. controls (220.28,108.92) and (221.7,107.92) .. (223.88,108.44) .. controls (226.23,108.87) and (227.63,107.95) .. (228.08,105.69) .. controls (228.75,103.33) and (230.21,102.46) .. (232.46,103.09) .. controls (234.52,103.88) and (235.96,103.16) .. (236.79,100.91) .. controls (237.82,98.7) and (239.35,98.14) .. (241.36,99.24) .. controls (243.32,100.52) and (245,100.26) .. (246.39,98.47) .. controls (248.15,96.86) and (249.81,96.97) .. (251.37,98.81) .. controls (252.86,100.65) and (254.5,100.78) .. (256.31,99.21) .. controls (257.9,97.49) and (259.58,97.36) .. (261.33,98.83) .. controls (263.34,100.13) and (264.98,99.76) .. (266.23,97.71) .. controls (267.26,95.64) and (268.89,95.08) .. (271.14,96.03) .. controls (273.32,96.96) and (274.9,96.36) .. (275.87,94.21) .. controls (276.9,92.09) and (278.44,91.57) .. (280.51,92.65) .. controls (282.48,93.8) and (283.99,93.37) .. (285.03,91.36) .. controls (286.36,89.31) and (288.06,88.91) .. (290.12,90.15) .. controls (292.08,91.45) and (293.71,91.14) .. (295,89.21) .. controls (296.31,87.31) and (297.96,87.06) .. (299.95,88.46) .. controls (301.82,89.91) and (303.46,89.71) .. (304.88,87.88) .. controls (306.27,86.07) and (307.87,85.93) .. (309.68,87.45) .. controls (311.54,88.98) and (313.21,88.85) .. (314.7,87.07) .. controls (316.31,85.26) and (317.98,85.11) .. (319.7,86.61) .. controls (321.69,87.99) and (323.3,87.64) .. (324.53,85.57) -- (325.79,84.86) ;
\draw    (271.12,44.57) -- (270.37,82.86) ;
\draw [shift={(270.33,84.86)}, rotate = 271.13] [color={rgb, 255:red, 0; green, 0; blue, 0 }  ][line width=0.75]    (10.93,-3.29) .. controls (6.95,-1.4) and (3.31,-0.3) .. (0,0) .. controls (3.31,0.3) and (6.95,1.4) .. (10.93,3.29)   ;
\draw  [dash pattern={on 4.5pt off 4.5pt}]  (15,63) -- (80,63) ;
\draw  [dash pattern={on 4.5pt off 4.5pt}]  (15,93.5) -- (80,93.5)(15,96.5) -- (80,96.5) ;
\draw  [dash pattern={on 0.75pt off 0.75pt}]  (15,124) .. controls (16.67,122.33) and (18.33,122.33) .. (20,124) .. controls (21.67,125.67) and (23.33,125.67) .. (25,124) .. controls (26.67,122.33) and (28.33,122.33) .. (30,124) .. controls (31.67,125.67) and (33.33,125.67) .. (35,124) .. controls (36.67,122.33) and (38.33,122.33) .. (40,124) .. controls (41.67,125.67) and (43.33,125.67) .. (45,124) .. controls (46.67,122.33) and (48.33,122.33) .. (50,124) .. controls (51.67,125.67) and (53.33,125.67) .. (55,124) .. controls (56.67,122.33) and (58.33,122.33) .. (60,124) .. controls (61.67,125.67) and (63.33,125.67) .. (65,124) .. controls (66.67,122.33) and (68.33,122.33) .. (70,124) .. controls (71.67,125.67) and (73.33,125.67) .. (75,124) .. controls (76.67,122.33) and (78.33,122.33) .. (80,124) -- (80,124) ;

\draw (102.14,179) node [anchor=north west][inner sep=0.75pt]   [align=left] {period morphism extends\\over these points};
\draw (500,182.5) node [anchor=north west][inner sep=0.75pt]   [align=left] {$\Spec(W)$};
\draw (251.81,18.43) node [anchor=north west][inner sep=0.75pt]   [align=left] {$(z, \zeta)$};
\draw (92,52) node [anchor=north west][inner sep=0.75pt]   [align=left] {$\sC_1 \times_{\overline{\sfM}} \sfD^+$};
\draw (94,84) node [anchor=north west][inner sep=0.75pt]   [align=left] {$\overline{\sC}_1 \times_{\overline{\sfM}} \sfH$};
\draw (91,115) node [anchor=north west][inner sep=0.75pt]   [align=left] {$\Def( z,\zeta )$};

\end{tikzpicture}}
\end{center}

    \noindent \textit{Step 3: Prepare for motivic alignment.} Let $k'$ be any algebraically closed field of characteristic $p$, let $z=\Spec(k') \to C_1$ be a geometric point.
		We want to prove that for the extended period morphism $\tau$, the pair $(z, \zeta)$ is motivically aligned for every $\zeta \in \LEnd(\sA_{\tau(z)})$. 
		Note that now we know two facts: 
		\begin{enumerate}[label=\upshape{(\roman*)}]
			\item By \cref{thm: period}(b), the local system $\bP_\ell |_{\sC_1^\circ}$ is isomorphic to $(\tau^* \bL_\ell(-1))|_{\sC_1^\circ}$. 
			\item By Step 1, the local system $\bP_{\ell}|_{\sC^\circ}$ has constant fiberwise monodromy invariant number over $W$, and $\lambda^{\geo}(\bP_{\ell}|_{\sC^\circ}) = \dim \PNS(\sX_{\bar{\eta}})_\IQ$. 
		\end{enumerate}
        By Artin-Brieskorn resolution \cite[p.322]{Artin-Res}, the Henselianization of the resolution space of $\sX|_{\sC_1}$ at $z$ is finite surjective over the Henselianization of $\sC_1$ at $z$. Therefore, we can find an \'etale neighborhood $\sU'$ of $z$ in $\sC_1$, and a finite surjective morphism $\sU \to \sU'$ of $W$-schemes, which is generically \'etale on both the special fiber and the generic fiber over $W$, such that $\sX|_\sU$ admits a simultaneous resolution $\breve{\sX} \to \sU$ (cf. the proof of \cref{lem: etale locally resolvable}). 
		We may lift $\bar{\eta}$ to a geometric generic point on $\sU$, which we still denote by $\bar{\eta}$. 
        By resolving singularities using Lipman's theorem \cite[\href{https://stacks.math.columbia.edu/tag/0ADX}{Tag 0ADX}, \href{https://stacks.math.columbia.edu/tag/0BGP}{0BGP}]{stacks-project}, we may assume that $\sU$ is a regular scheme of finite type over $W$, except that now $\sU$ is not necessarily finite over $\sU'$, but remains proper. 
        The composition $\sU \to \sC_1 \to \shS_W$ is then an extended period morphism. \\

        \noindent \textit{Step 4: Deform $(z, \zeta)$ to char $0$.} Consider the deformation space $\Def(z, \zeta)$ of $(z, \zeta)$ along $\sC_1$. It is a formal scheme over $W$ which pro-represents the functor that sends every Artin local ring $R$ with residue field $k'$ to the set of $R$-points on $\sC_1$ extending $z$ along which $\zeta$ deforms. We argue that $\Def(z, \zeta)$ is flat over $W$. 
        At this point the conclusion follows from the same proof as \cite[Thm.~5.2.3]{HYZ}, where we take the $\sfM$ therein to be the scheme $\sC_1$ in the current context. 
		For the reader's convenience, we explain the modification needed in order to adapt the argument in \textit{loc. cit.}
        If the flatness does not hold, then since $\Def(z, \zeta)$ is of codimension one in the formal completion of $\sC_1 \tensor_W W(k')$ at $z$, the special endomorphism $\zeta$ must comes from the specialization of some element in $\LEnd(\sA_{\bar{\eta}_p})$, where $\bar{\eta}_p$ is a geometric generic point of $\sC_{1}$. 
		The latter implies that 
		$$\dim \LEnd(\sA_{\bar{\eta}_p})_\IQ  > \dim \LEnd(\sA_{\bar{\eta}})_\IQ. $$ 
		On the other hand, we have 
		$$\dim \LEnd(\sA_{\bar{\eta}})_\IQ = \dim \PNS(\sX_{\bar{\eta}})_\IQ = \lambda^\geo(\bP_\ell |_\sU), $$ 
		and at the same time we have
		$$\dim \LEnd(\sA_{\bar{\eta}_p})_\IQ \le \lambda^\geo(\tau^*(\bL_\ell(-1))|_{C_1}) = \lambda^\geo(\bP_\ell |_{C_1}). $$ 
		This contradicts the fact that $\tau^*(\bL_\ell(-1))|_{\sC_1} \simeq \breve{\bP}_{\ell, \sC_1}$ has constant fiberwise monodromy invariant $\lambda^\geo$, which thus concludes the proof. \\

\noindent \textit{Step 5: Establish motivic alignment.} Recall that our goal is to show that $(z, \zeta)$ is motivically aligned for the geometric point $z$ on $C_1$. 
Since the extended period morphism $\sU_k \to \shS_k$ factors through $C_1$, it suffices to choose a lift $z'$ of $z$ to $\sU_k$ and prove that $(z', \zeta)$ is motivically aligned. 
As the $W$-scheme $\sU$ is proper and surjective over an \'etale neighborhood of $z$ on $\sC_1$, Step 4 implies that we may lift $(z', \zeta)$ to characteristic $0$ along $\sU$. 
Hence we may finish the proof by \cref{prop: char 0 motivically aligned} and \cref{lem: def invariance of MA}. 
\end{proof}

\begin{corollary}[The Tate conjecture for mildly singular fibers]
\label{cor: good points}
    \cref{thm: general theorem} holds when $s \in \sfD^+(k) \cup \sfM^\circ(k)$. 
\end{corollary}
\begin{proof}
    Let $b \colonequals \pi(s)\in \sfB(k)$. 
    We can always find some admissible curve $\overline{C} \subseteq \overline{\sfM}_b$ which passes through $s$ and intersects $\sfH_b \cup \sfD_b$ transversely and only along $\sfH^+_b \cup \sfD^+_b$. 
    To see this, one can show that when the $d_i$'s are large enough in \cref{const: strongly dominate}, the sub-family of $\shC/\sT$ which passes through $s$ in some fixed direction still strongly dominates $\overline{\sfM}_b$ away from $s$. 
    Then one applies \cref{lem: transverse intersection general}. 
    
    Then it suffices to construct an extended period morphism $\tau : C_1 \to \shS_k$ as in \cref{prop: good points}. We shall do so by reversing the construction in its proof: 
    Let $C_0^\circ$ be a connected finite \'etale cover of a connected component of $\wt{\sfM}^\circ \times_{\sfM^\circ_k} C^\circ$, let $\overline{C}_0$ be its compactification (which admits a map to $\overline{\sfM}$, and let $C_0 $ be the intersection $\overline{C}_0 \times_{\overline{\sfM}} \sfM$.
    Then the morphism $C^\circ_0 \to \shS_k$ extends to a morphism $C_0 \to \shS_k$ by \cref{cor: extend DVR}. 
    Now the only problem is that $\sX|_{C_0}$ may not be Zariski-locally resolvable, but we may resolve this by passing to a suitable generically \'etale cover $C_1$ of $C_0$ thanks to \cref{lem: etale locally resolvable}.
\end{proof}

\subsection{Contraction of curves by the period morphism}

To study the periods for the points in $\sfD^-$, we make some observations on when extended period morphism contracts a curve. In this subsection, $k$ is a perfect field of characteristic $p$.

\begin{proposition}
\label{prop: height jump}
    Let $C$ be a smooth \emph{proper} curve over $k$. 
    Let $\rho : C \to \shS_k$ be a morphism such that the F-isocrystal $\rho^* \bL_\crys[1/p]$ has constant slope filtration and is nowhere supersingular. Then the map $\rho$ contracts $C$ to a point. 
\end{proposition}
Here we remind the reader that an F-isocrystal is nowhere supersingular if its Newton polygon at any closed point is not a straight line.
\begin{proof}
    This observation is essentially due to Maulik, van der Geer and Katsura (\cite[Cor.~5.5]{Maulik}, \cite[Thm.~15.3]{vdGK}) in the context of moduli of K3 surfaces,
    and we explain how this extends to orthogonal Shimura varieties. 
    Consider the Hodge bundle $\lambda$ given by $\Fil^1 \bL_\dR$, and let $\lambda_\sA$ be the determinant Hodge bundle of the uniserval abelian scheme $a : \sA \to \shS$, i.e., $\det(a_* \omega_{\sA / \shS})$. 
    Then by adapting the argument of \cite[Prop.~5.8]{Maulik}, we have $\IQ_+ \cdot \lambda = \IQ_+ \cdot \lambda_\sA \subseteq \Pic(\shS)_\IQ$. 
    As the determinant Hodge bundle $\lambda_\sA$ is ample, so is $\lambda$. 

    Let $n = \rank \, \bL_\crys$. 
    Let $\bL$ be the restriction of $\bL_\crys$ at a point with perfect residue field. 
    If $\bL$ is not supersingular (i.e., is isoclinic of slope $0$), then its slope $< 0$ part is of the slope $-1/h$, where $h$ is called the height of $\bL$. 
    Consider now the height (i.e., Newton) stratification $\shS_{\IF_p} = \shS_1 \supsetneq \shS_{2} \supsetneq \cdots \supsetneq \shS_{\lfloor n / 2 \rfloor} \supsetneq \shS_\infty$, where $\shS_\infty$ is the supersingular locus.  
    Then each stratum is a Cartier divisor of the previous one, and the Chern class of this Cartier divisor equals that of $\lambda$ up to a scaling by $\IQ_+$. 
    One can prove this by suitably adapting the proof of \cite[Thm.~15.1]{vdGK}. 
    
    If $\rho$ does not contract $C$ to a point, then $\rho$ is a finite morphism of $C$ onto its scheme theoretic image and $\rho^*(\lambda)$ is ample.
    In particular, the preimage of the height stratification is non-trivial in $C$---this would contradict the assumption that $\rho^* \bL_\crys[1/p]$ has constant slope unless the image of $\rho$ is entirely contained in the supersingular locus $\shS_\infty$. 
\end{proof}

Next we recall a consequence of the Mazur--Ogus inequality.
\begin{proposition}
\label{prop: Mazur-Ogus}
    Let $Z$ be a smooth $k$-variety over $k$, equipped with a morphism $Z \to \shS_k$. 
    Let $F_{Z/k} : Z \to Z'$ be the relative Frobenius morphism. 
    Then the Hodge filtration on $\bL_\dR|_{Z'}$ (and hence $\bL_\dR|_Z$) is determined by the Frobenius structure $\varphi_{\bL_\crys(-1)} : \bL_\crys(-1)|_{Z'} \to \bL_\crys(-1)|_Z$\footnote{Note that $\bL_\crys$ is not an F-crystal but $\bL_\crys(-1)$ is. }.
    
    More precisely, suppose that $\wt{Z}$ is a Frobenius equivariant $p$-adic formal scheme that lifts $Z$ over $W$. 
    Then 
    \begin{equation}
        \label{eqn: Mazur-Ogus}
        \Fil^i (\bL_\dR|_{Z'}) =  (p\varphi_{\bL_\crys})^{-1} \big( p^i (\bL_\crys|_{Z})({\wt{Z}}) \big) \mod p.
    \end{equation}
\end{proposition}
\begin{proof}
    When the pair $(\bL_\crys, \bL_\dR)$ is replaced by $(\bH_\crys, \bH_\dR)$, (\ref{eqn: Mazur-Ogus}) is nothing but the Mazur--Ogus formula for the Hodge filtration (see \cite[Thm.~8.26]{BOBook}). 
    Then we use that $(\bL_\crys, \bL_\dR)$ is a direct summand of the internal Hom of the object $(\bH_\crys, \bH_\dR)$. 
\end{proof}

\begin{lemma}
\label{lem: constant K3 crystal}
    Let $C$ be an integral curve over $k$ with a morphism $C \to \shS_k$. 
    Then $C$ is contracted to a point if and only if $\bL_\crys|_C$ is constant $F$-crystal equipped with a pairing (i.e., for any $k$-point $t \in C$, there is an isometry of crystals $\bL_\crys \simeq (C \to \Spec(k))^* (\bL_\crys|_t)$). 
\end{lemma}
\begin{proof}
    Only the ``if'' direction requires a proof. 
    We may assume that $C$ is smooth and let $t \in C$ be any $k$-point. Let $u$ be a uniformizer at $t$ such that the formal completion of $C$ at $t$ is $\mathrm{Spf}(k[\![u]\!])$. 
    It suffices to show that the induced map $T_\infty := \mathrm{Spf}(k[\![u]\!]) \to \shS_k$ factors though the reduced point $\Spec(k)$, or equivalently, for any $m \in \IN$, the induced map $T_m := \mathrm{Spec}(k[u]/u^m) \to \shS_k$ factors through $\Spec(k)$. We shall prove this by induction. The base case $T_1 = \Spec(k)$ is clear. 
    Now suppose the statement holds for $m$, so that $\bL_\crys|_{T_m} = (T_m \to T_1)^*(\bL_{\crys}|_{T_1})$, which gives us an identification $$\bL_{\dR}|_{T_{m + 1}} = (\bL_\crys|_{T_m})(T_{m + 1}) = [(T_m \to T_1)^*(\bL_{\crys}|_{T_1})](T_{m + 1}) = (T_{m + 1} \to T_1)^*(\bL_\dR|_{T_1}), $$ when we view $T_{m + 1}$ as a divided power thickening of $T_m$. 
    \footnote{More precisely, in the first equality, we implicitly use that the evaluation of $\bL_\crys|_{T_{m+1}}$ at the thickening $(T_{m+1},T_{m+1})$ coincides with that of $\bL_\crys|_{T_m}$ at the thickening $(T_{m+1},T_m)$, together with the identification of $\bL_\crys|_{T_{m+1}}(T_{m + 1})$ with the de Rham cohomology $\bL_\dR|_{T_{m + 1}}$.}
    As the $\Fil^1(\bL_{\dR}|_{T_{m + 1}})$ is obtained from pulling back $\Fil^1 \bL_\dR |_{T_1}$, the statement follows from \cite[Prop.~5.16]{CSpin}.
\end{proof}

\begin{remark}
    Note that in characteristic $p$, we cannot simply argue that $C$ is contracted to a point when the map kills all tangent vectors, as the map might factor through an inseparable morphism. However, the problem goes away when we consider nilpontents of arbitrarily large order. 
\end{remark}

\subsection{Locus of worse singularity}
\label{sub:bad_points}
In this subsection, we prove the Tate conjecture for the minimal resolution of $\sX_s$ for a point $s\in \sfD^-$ in positive characteristic, and in particular finish the proof of \Cref{thm: general theorem}.
For any such point $s\in \sD^-_k$, we shall construct an extended period morphism $\tau:T\to \shS$ in positive characteristic, where $T$ is a $\sfM_k$-scheme $T$ of dimension $2$, such that $s$ is contained in the image of $T$ in $\sfM$.
As heuristics, we in particular prove that the ``period'' of the surface $\sX_s$, defined as the image $\tau(T_s)$ in the Shimura variety $\shS$, is in fact a single point, where $T_s$ is the preimage of $s$ in $T$ and potentially could be of higher dimension.
The latter is achieved by an essential application of the minimal model program for threefolds in positive characteristic, specifically the Kawamawa theorem as in \Cref{sub:MMP}.
Knowing the constancy the of period and in particular the constancy of the Kuga--Satake abelian variety of $s$, we may then study the deformation of a special endomorphism along certain curve in $T$ that is chosen with delicacy, where various arguments in intersection theory will come to place.

We start with some general observations on intersection theory. 
Following \cite{Fulton}, we use the following notation: If $D$ is a Cartier divisor on an integral scheme $X$, we write $|D|$ for its support.
For a $k$-dimensional irreducible closed subvariety $Z \subseteq X$ (or a formal sum of such), we write $[Z]$ for its class in $A_k(X)$ (the group of $k$-cycles modulo rational equivalence). 
\begin{lemma}
\label{lem: neg self intersection}
    Suppose that $\pi : S \to S'$ is a generically finite morphism bewteen projective surfaces over a field and $S$ is smooth. Then for every divisor $C$ on $S$ whose support is contracted by $\pi$ (i.e., $\pi(|C|)$ is a point on $S'$), the self-intersection $C^2$ is negative. 
\end{lemma}
\begin{proof}
Let $A$ be an ample divisor on $S'$. Note that $\pi^*(A)$ remains big and nef, as this property is invariant under generically finite pullback. In particular, $\pi(A)^2 > 0$. On the other hand, $\pi_*([C]) = 0 \in A_1(S')$ as $\pi(|C|)$ is a point (see e.g., \cite[\S1.4]{Fulton}). By the projection formula \cite[Prop.~2.3(c)]{Fulton}, $\pi^*(A) . [C] = \pi_*([C]) . A = 0 \in A_0( \pi(|C|)) \simeq \IZ$. Therefore, the Hodge index theorem on $S$ implies that $C^2 < 0$. 
\end{proof}

\begin{proposition}
\label{prop: apply Hodge index}
    Consider a diagram of morphisms
    \[\begin{tikzcd}
	S & Y \\
	S'
	\arrow["\pi"', from=1-1, to=2-1]
	\arrow["\tau", from=1-1, to=1-2]
    \end{tikzcd},\]
    between integral projective varieties over a field such that $S$ is a smooth surface and $\pi$ is generically finite. Suppose that $C \subseteq S$ is a connected reduced curve which is contracted by both $\pi$ and $\tau$. 
    \begin{enumerate}[label=\upshape{(\alph*)}]
        \item If $\dim Y = 1$, then there exists a connected curve $D \subseteq S$ such that $D \subseteq \tau^{-1}(\tau(C))$, $D \cap C \neq \emptyset$, and $D$ is not contracted by $\pi$. 
        \item If $\dim Y = 2$, then for every Cartier divisor $E$ on $Y$ with $\tau(C) \in |E|$, there exists an integral curve $D \subseteq S$ such that $D$ is contained in $\tau^{-1}(E)$, $D \nsubseteq C$ and $D \cap C \neq \emptyset$. 
    \end{enumerate}
\end{proposition}
\begin{proof}
Below we write $Q$ for the image $\tau(C)$, which is a reduced point on $Y$. 
We also assume that the image $\tau(S)$ is not a point; otherwise the claim is much easier since the map $\pi$ cannot contract every curve on $S$..

To prove (a), it is harmless to replace $Y$ by its normalization and assume that $Y$ is smooth and view $Q$ as a divisor on $Y$. Then by the projection formula for the intersection product, we know the self-intersection product $[\tau^{-1}(Q)]^2$ of the corresponding cycle class $[\tau^{-1}(Q)]$ is equal to $0$. 
Therefore, by \cref{lem: neg self intersection}, the divisor $|\tau^{-1}(Q)|$ cannot be contracted by $\pi$ and is not equal to $|C|$.
Note that by \cite[\href{https://stacks.math.columbia.edu/tag/0AY8}{Tag 0AY8}]{stacks-project}, the support $|\tau^*(Q)|$ is geometrically connected. 
Therefore, the divisor $|\tau^{-1}(Q)|$ must be of the form $C \cup D$ for some divisor $D$ satisfying the assumption in (a). 

Now we prove (b). 
Let $C_1, \cdots, C_n$ be the irreducible components of $|C|$. 
Note that as $S$ is a smooth surface, we may freely identify the Chow group $A_1(S)$ with $\Pic(S)$. 
By the assumption that $Q \in |E|$, the preimage $\tau^{-1}(E)$ is of the form $\sum a_i C_i + \wt{D}$, where $a_i \in \IN_{> 0}$ and $\wt{D}$ is a non-trivial divisor whose support does not contain any of the $C_i$'s as an irreducible component. 
Write $[C']$ for the element $\sum a_i [C_i] \in A_1(S)$. 
Then as each $C_i$ is contracted to a point, we know $\tau_*([C'])= 0 \in A_1(Y)$. 
Applying the projection formula to the map $\tau$, we have $[C'] . [\tau^{-1}(E)] = [\tau_*(C')] . [E] =  0 \in A_0(Q) \simeq \IZ$. 
On the other hand, we know $[C'] . [\tau^{-1}(E)] = [C']^2 + [C' ]. [\wt{D}]$ and $[C']^2 < 0$ by \cref{lem: neg self intersection}. 
Hence we must have $[C'] . [\wt{D}] > 0$. 
This implies that we may take $D$ to be some connected component of $|\wt{D}|$ which intersects $C$, which necessarily exists. 
\end{proof}

Our next lemma shows that for a one-dimensional family that is trivializable up the normalization, we may freely replace one fiber to another, for the purpose of proving the Tate conjecture.
\begin{lemma}
\label{lem: move points}
    Let $C$ be an integral curve over an algebraically closed field $\kappa$, and $f : X \to C$ be a smooth proper morphism of algberaic spaces with geometrically connected fibers. Take a prime $\ell \neq \mathrm{char\,}\kappa$. Let $\nu : C^\nu \to C$ be the normalization. 
    Assume the following data:
    \begin{itemize}
    	\item a trivialization of the family $\alpha : X \times_C C^\nu \simeq X_0 \times_\kappa C^\nu$ for some smooth proper algebraic space $X_0$ over $\kappa$;
    	\item a trivialization $\alpha_\ell$ of the local system $\alpha_\ell$ that identifies it with the constant local system with fiber $\H^2_\et(X_0, \IQ_\ell)$ over $C$;
    	\item the pullback $\nu^* \alpha_\ell$ is compatible with $\alpha$.
    \end{itemize}
Then for every two $\kappa$-points $c, c' \in C$, there is a canonical isomorphism $\NS(X_c)_\IQ \simeq \NS(X_{c'})_\IQ$ which is compatible with the identification $\H^2_\et(X_c, \IQ_\ell) \simeq \H^2_\et(X_{c'}, \IQ_\ell)$ offered by $\alpha_\ell$. 
\end{lemma}
\begin{proof}
    Choose points $\wt{c}, \wt{c}' \in C^\nu$ which lift $c, c'$ respectively. 
    Then the map $\alpha$ induces an isomorphism $X_c = X_{\wt{c}} \simeq X_{\wt{c}'} = X_{c'}$. 
    The resulting isomorphism $\NS(X_c)_\IQ \simeq \NS(X_{c'})_\IQ$ does the job. Note that the isomorphism $X_c \simeq X_{c'}$ may depend on the choices of $\wt{c}, \wt{c}'$, but the resulting $\NS(X_c)_\IQ \simeq \NS(X_{c'})_\IQ$ does not, as the latter has to be compatible with $\alpha_\ell$ over $C$. 
\end{proof}

By applying \cref{cor: res of constant family}, we prove that the extended period morphism $\tau$ would contract any curve that is contracted by $\pi$. 
For \cref{lem: double contraction} and \cref{lem: bad points} below, let $k$ be a perfect field of characteristic $p$. 
\begin{proposition}
\label{lem: double contraction}
    Let $\tau:T\to \shS$ be an extended period morphism for a $k$-variety $T$, and let $C$ be a smooth $k$-curve mapping to $T$. 
    Assume one of the following conditions:
    \begin{enumerate}[label=\upshape{(\roman*)}]
        \item For a general $t \in C$, the minimal resolution $\breve{\sX}_t$ is not supersingular, and the map $C \to \shS_k$ extends to its compactification $\overline{C}\to \shS_k$. 
        \item For a general point $t \in C$ and every $i,j$, we have an equality of Hodge numbers $h^{i, j}(\breve{\sX}_t) = h^{i, j}(\sX_\eta)$. 
    \end{enumerate}
    Then if the map $\pi:T \to \sfM$ contracts $C$, so does $\tau$. 
\end{proposition}
\begin{proof}
    Recall from \Cref{def: extended period} that the restriction of $\sX$ to $T$, and hence $C$, is Zariski-locally resolvable. 
    So we may take some open $U \subseteq T$ such that $C \times_T U \neq \emptyset$ and $\sX|_U$ admits a simultaneous resolution $\sY$, and for convenience, we replace $C$ by its open subset $C \times_T U$. 
    Note that under the assumption, the restriction $\sY|_C$ is a resolution of the constant family $\sX|_C$. 
    Now thanks to \cref{cor: res of constant family}, the family $\sY|_C$ must also be a constant family.

    We first prove the statement assuming (i).
    Recall that we have an isometry of F-isocrystals $\bL_{\crys}(-1)[1/p]|_{T^\circ} \simeq \bP_\crys|_{T^\circ}$. 
    By the full faithfulness of the restriction functor for F-isocrystals in \cite[Thm.~2.2.3]{DK17}, the isometry extends over $T$. 
    Therefore, using that $\sY|_C$ is a constant family, we deduce that $\bL_{\crys}[1/p]|_C$ is constant. 
    Moreover, since $C \to \shS_k$ extends to $\overline{C} \to \shS_k$, by applying \cite[Thm.~2.2.3]{DK17} again we know that $\bL_{\crys}[1/p]|_{\overline{C}}$ is also constant. 
    If this F-isocrystal is not supersingular, then by \cref{prop: height jump}, $\overline{C}$ must contract to a point. This proves (i).

    Now we consider (ii). 
    By the semi-continuity of Hodge numbers, up to shrinking $U$ to an open we may assume that the Hodge number condition in (ii) is satisfied for every point on $U$. 
    By \Cref{lem:local_free_of_crys_coh}, the relative $\H^2_\crys$ of the family $\sY|_U$ is a locally free F-crystal. Then we have by the matching of crystalline cohomology by period morphism $\bL_\crys(-1)|_{U \cap T^\circ} \simeq \bP_\crys|_{U \cap T^\circ}$. By \cref{lem: rigidity of F-crystals}, this extends to $\bL_\crys(-1) |_{U } \simeq \bP_\crys|_{U}$. 
    As $\sY|_C$ is a constant family, we infer that $\bL_\crys|_C$ is a constant F-crystal. 
    By \cref{prop: Mazur-Ogus}, this implies that the restriction of $(\bL_\dR, \Fil^\bullet \bL_\dR)$ to $C$ is a constant filtered flat vector bundle. 
    So the conclusion follows from \cref{lem: constant K3 crystal}.
\end{proof}

\begin{theorem}
\label{lem: bad points}
    Let $\tau:T\to \shS$ be an extended period morphism from a smooth connected $k$-surface $T$, and let $\pi$ be the associated morphism $T \to \sfM$. Let $t \in T$ be a closed point. 
    Assume the following conditions. 
    \begin{enumerate}[label=\upshape{(\roman*)}]
        \item The map $\pi$ factors through a surjective morphism onto a smooth $k$-subsurface $S \subseteq \sfM_k$ such that $\dim S \cap \sfD^-_k = 0$. 
        \item The generic point of any integral $k$-curve on $T$ that is not contracted by $\pi$ is motivically aligned. 
        \item The minimal resolution $\breve{\sX}_t$ is non-supersingular, or $h^{i, j}(\breve{\sX}_s) = h^{i, j}(\sX_\eta)$ for every closed point $s\in T$ and every $i, j$. 
    \end{enumerate}
     Then $\tau$ is motivically aligned at $t$. 
\end{theorem}
\begin{proof}
     By \Cref{def: motivically aligned}, we need to show that for every $\zeta \in \LEnd(\sA_{\tau(t)})$, the pair $(t, \zeta)$ is motivically aligned. We may assume that $k$ is algebraically closed. 
     Let $Y$ be the scheme-theoretic image $\tau(T)\subset \shS$. 
     Let $\overline{S}$ be the Zariski closure of $S$ in $\overline{\sfM}_k$. 
     
     We first define an integral $k$-curve $D\subset T$ to be \emph{good} with respect to the pair $(t,\zeta)$ if it satisfies the following properties: 
     \begin{enumerate}[label=\upshape{(\arabic*)}]
         \item The point $t$ is contained in $D$, and $D$ is not contracted by $\pi$. 
         \item There exists a path $\gamma : s \rightsquigarrow t$, where $s$ is a geometric generic point on $D$, such that $\zeta$ lies in the image of the specialization map $\sp_\gamma : \LEnd(\sA_{\tau(s)})_\IQ \to \LEnd(\sA_{\tau(t)})_\IQ$. 
     \end{enumerate}
     Granting this existence of a good curve, by assumptions (i) and (ii), we know that the point $s$ is motivically aligned as it is not mapped into $\sfD^-_k$. 
     Hence by the fact that being motivically aligned is preserved under specialization (\cref{lem: def invariance of MA}), we conclude that $(t, \zeta)$ is also motivically aligned.
     Here we note that even if a good curve $D$ might be singular at $t$, a path $\gamma : s \rightsquigarrow t$ always exists by a general fact about locally Noetherian schemes (\cite[\href{https://stacks.math.columbia.edu/tag/054F}{Tag 054F}]{stacks-project}). 
     
     Now we give the construction of the good curve $D$ (up to deforming $t$ along a constant family), based on the possible dimension of the image $Y$.
     As we shall see, the curve $D$ will be constructed such that any path $\gamma$ will satisfy the condition (2).\\
     
     \noindent \textit{Case (a).} Assume $\dim Y=0$.
     	We let $D\subset T$ be any curve through $t$ that is not contracted by $\pi$, which exists because $\pi$ is generically finite.
     	Let $\gamma : s \rightsquigarrow t$ be any path from a geometric generic point over $D$.
     	As $D$ is contracted by $\tau$, $\tau(s)$ and $\tau(t)$ map to the same point on $\shS_k$, so that $\LEnd(\sA_{\tau(t)})_\IQ$ and $\LEnd(\sA_{\tau(s)})_\IQ$ are naturally isomorphic, which gives the property (2).\\\\
     \noindent \textit{Case (b).} Assume $\dim Y=1$.
     	By the same reasoning as above, it suffices to show that there is a curve $D$ that satisfies (1) and is contracted by the map $\tau$.
     	Let $\overline{Y}$ be Zariski closure of $Y$ in $\overline{\shS}_k$. 
     	By applying resolution and compactification of surfaces as in \cite[\href{https://stacks.math.columbia.edu/tag/0C5H}{Tag 0C5H}]{stacks-project} we can always find a smooth compactification $\overline{T}$ of $T$ such that 
     	\begin{equation}
     		\label{eqn: extend to Tbar}
     		\begin{tikzcd}
     			T & Y \\
     			{S}
     			\arrow["\tau", from=1-1, to=1-2]
     			\arrow["\pi", from=1-1, to=2-1]
     		\end{tikzcd} \textrm{ extends to } \begin{tikzcd}
     			{\overline{T}} & {\overline{Y}} \\
     			{\overline{S}.}
     			\arrow["\overline{\tau}", from=1-1, to=1-2]
     			\arrow["\overline{\pi}", from=1-1, to=2-1]
     		\end{tikzcd}
     	\end{equation}
     We then let $C$ be the reduced subscheme of the connected component of $\overline{\pi}^{-1}(\pi(t))=\overline{T}\times_{\overline{\sfM}} \{\pi(t)\}$ containing the point $t$. 
     If $\dim C = 0$, namely when $C = \{t\}$, then there is no curve $D$ on $T$ that contains the point $t$ and is contracted by $\pi$.
     In particular, any curve $D\subset T$ that contains the point $t$ would satisfy (1).
     Notice that since $\overline{T}$ is smooth and connected of dimension $2$, yet $\dim \overline{Y}=1$, by \cite[\href{https://stacks.math.columbia.edu/tag/02FZ}{Tag 02FZ}]{stacks-project} each fiber of the map $\overline{\tau}$ has dimension at least one.
     Hence we may take a curve $D\subset \tau^{-1}(t)$ through $t$.
     
     We then assume that $\dim C = 1$.
     Since the curve $C$ is contracted by $\overline{\pi}$, by assumption (iii) and the \Cref{lem: double contraction}, we know $C$ is also contracted by $\overline{\tau}$.
     So by \cref{prop: apply Hodge index}.(a), we may find a curve $D \subset \overline{\tau}^{-1}(\overline{\tau}(C))$ that intersects with $C$ and is not contacted by the map $\overline{\pi}$.
     To continue, we discuss the two scenarios based on the position of the curve $C\subset \overline{\pi}^{-1}(\pi(t)) \subset \overline{T}$.
     \begin{itemize}
     	\item[(b1)] Assume the curve $C$ is already contained in the open subscheme $T$ of $\overline{T}$.
        Note that since the curve $C$ is contracted by $\pi$, the restriction $\sX|_{C}$ is a constant family. 
            By \cref{cor: res of constant family}, any (a priori Zariski local) simultaneous resolution of $\sX$ over the normalization $C^\nu$ of $C$ must be a constant family. 
            Hence by \cref{lem: double contraction}, $C$ is contracted by $\tau$. 
     	This implies that $\breve{\bP}_{\ell, T}|_C \simeq (\tau|_C)^* \bL_\ell$ is a constant $\ell$-adic local system. 
     	Therefore, we are in a situation to which \cref{lem: move points} applies, which implies that to prove that $(t, \zeta)$ is motivically aligned, we may deform it along the constant family over $C^\nu$ given by resolving $\sX|_{C^\nu}$ and assume that $t \in C \cap D$. 
     	\item[(b2)] In general, since the point $\pi(t)$ lies in $\sfM$, the scheme $\overline{\pi}^{-1}(\pi(t))=\overline{T} \times_{\overline{\sfM}} \pi(t)$ is always contained in $\overline{T} \times_{\overline{\sfM}} \sfM$.
     	We let $T' \colonequals \overline{T} \times_{\overline{\sfM}} \sfM$.
     	Using the extension property for curves in \cref{cor: extend DVR}, the morphism $\tau : T \to \shS_k$ extends to $T'$: Indeed, the morphism extends to $\overline{T} \to \overline{\shS}_k$ by construction, and in order to check whether a point $t' \in T$ maps into the interior $\shS_k$, we may take a general curve passing through $t'$ and apply \cref{cor: extend DVR}.
        
     	By \cref{lem: etale locally resolvable}, we can find a connected smooth $k$-surface $T''$, which is a connected generically \'etale cover of $T'$, such that $\sX|_{T''}$ is Zariski-locally resolvable. 
     	We take $C'' \subseteq T''$ to be a connected curve that maps surjectively to $C$.
     	As a consequence, we may conclude by applying the arguments in the preceeding paragraph by replacing $(T, C)$ by $(T'', C'')$. 
     \end{itemize}$~$\\
      \noindent \textit{Case (c).} Finally, we assume $\dim Y=2$, where our strategy is to extend the special endomorpohism to a hypersurface in $T$, using the basic facts about special divisors.
     	We let $m \in \IZ_{(p)}$ be the self-intersection number $\zeta\circ\zeta\in \LEnd(\sA_{\tau(t)})$, and let $\shZ(m)$ be the functor over $\mathsf{\Sch}/\shS$ defined by sending each $\shS$-scheme $?$ to the set 
     	$\{ \Xi \in \LEnd(\sA_?) \mid \Xi\circ \Xi = m \}$.
     	It is known that the functor $\shZ(m)$ is \'etale-locally representable by effective Cartier divisor on $\shS$ (cf. \cite[Prop.~4.5.8]{AGHMP}; see also \cite[Prop.~6.5.2]{HMP}). 
     	In particular, there exists an \'etale cover $\{ \sU_\alpha \}$ of $\shS$, such that the restriction $\shZ(m) |_{\sU_\alpha}$ is representable by an effective Cartier divisor $\sE_\alpha \subseteq \sU_\alpha$. 
     	By construction, we may extend the special endomorphism $\zeta\in \LEnd(\sA_{\tau(t)})$ to a special endomorphism in $\LEnd(\sA_{\sE_\alpha})$.
     	Namely, by picking a scheme $\sU := \sU_\alpha$ that is an \'etale neighborhood of the point $t$, we find a morphism $\iota : V \to \sU$ that covers $\tau$, where $V$ is a connected component of $T\times_{\shS} \sU$. 
        There exists a point $v \in V$ lifting $t \in T$ and an effective Cartier divisor $\sE \subseteq \sU$ passing through $\iota(v)$, together with a special endomorphism $\Xi \in \LEnd(\sA_\sE)$ such that $\Xi_{\iota(v)} = \zeta$ (under the identification $\sA_{\iota(v)} = \sA_{\tau(t)}$).

    We let $Z$ be the image $\iota(V)$ and choose an arbitrary compactification $\overline{Z}$. 
    Then by applying the compactification and the resolution for surfaces as before (\cite[\href{https://stacks.math.columbia.edu/tag/0C5H}{Tag 0C5H}]{stacks-project}), we can find a smooth compactification $\overline{T}$ of $T$ and $\overline{Z}$ of $Z$ which fit into a commutative diagram 
    \begin{equation}
        \begin{tikzcd}
	V & {\overline{V}} \\
	T & {\overline{T}} & Z & {\overline{Z}} \\
	S & {\overline{S}} & Y
	\arrow[hook, from=1-1, to=1-2]
	\arrow[from=1-1, to=2-1]
	\arrow["\iota"'{pos=0.4}, from=1-1, to=2-3]
	\arrow[from=1-2, to=2-2]
	\arrow["{{\overline{\iota}}}", from=1-2, to=2-4]
	\arrow[hook, from=2-1, to=2-2]
	\arrow[from=2-1, to=3-1]
	\arrow["\tau"'{pos=0.4}, from=2-1, to=3-3]
	\arrow[from=2-2, to=3-2]
	\arrow[hook, from=2-3, to=2-4]
	\arrow[from=2-3, to=3-3]
	\arrow[from=3-1, to=3-2]
\end{tikzcd}. 
    \end{equation}
    By taking the closure, the Cartier divisor $E \colonequals \sE \cap Z$ on $Z$ can be extended to $\overline{Z}$ and we denote the extension also by $\overline{E}$. 

    At this moment, we let $C$ be the reduced subscheme of the connected component of $\overline{V}_{\pi(t)}=\overline{V}\times_{\overline{S}} \{\pi(t)\}$ containing the point $v$. 
    If $\dim C = 0$, i.e., $C = \{v\}$, then we let $D$ be the connected component of $|\iota^{-1}(E)|$ passing through $v$, which then is not contracted by $\pi$ (and satisfies the condition (1)).
    In addition, since the special endomorphism $\zeta$ extends to a special endomorphism $\Xi$ over the entire divisor $E$, for any geometric generic point $s$ of $D$, we obtain an element $\Xi_{\iota(s)} \in \LEnd(\sA_{\iota(s)})$.
    Hence any path $\gamma:s \rightsquigarrow v$ on the curve specializes the element $\Xi_{\iota(s)}$ to $\zeta$, and in particular the curve $D$ is good with respect to the pair $(v, \zeta)$.

    If $\dim C = 1$, then again $C$ is contacted by $\pi$ and the family $\sX_C$ is a constant family by \Cref{lem: double contraction}, as in case (b).
    Assume first that $C \subseteq \overline{V}$ is already contained in the open part $V$, as in case (b1). 
    By \cref{prop: apply Hodge index}.(b), we may find a curve $D \subset \overline{\iota}^{-1}(\overline{E})$ that is not in $C$ but intersects $C$. 
    Thus by \Cref{lem: move points} and the same arguments as in the case (b1), we may replace the pair $(v, \zeta)$ by $(v', \Xi_{\iota(v')})$, where  $v' \in C \cap D$.
    Under the latter assumption, the special endomorphism $\Xi_{\iota(v')}$ admits a generalization $\Xi_{\iota(s)}$, where $s$ is any geometric generic point of $D$.
    Moreover, since $D$ is not in $C$, we know the curve $D$ is not contacted by $\pi$.
    Hence $D$ is good with respect to the pair $(v', \Xi_{\iota(v')})$. 
    Finally, if $C$ is not contained in $V$, one adapts the argument in (b2) and we omit the details. 
\end{proof}
Now we can finish the proof of \Cref{thm: general theorem} for those points that are left from \Cref{cor: good points}.
\begin{corollary}[The Tate conjecture for fibers with worse singularity]
\label{cor: bad points}
    \cref{thm: general theorem} holds for $s \in \sfD^-$ under the assumptions (ii) or (iii). 
\end{corollary}
\begin{proof}
    Recall that in \cref{thm: general theorem} we set $k:= k(s)$. By \cref{lem: reduction}, it suffices to construct an extended period morphism $\tau : T \to \shS_k$ which satisfies the hypothesis of \cref{lem: bad points} such that the point $s$ is in the image $S$ of $T$. 
    Recall that we are allowed to replace $k$ by a finite extension. 
    In particular, we may assume that all irreducible components of $\sfD^+_b$ and $\sfH_b$ are geometrically irreducible. Set $b \colonequals \pi(s) \in \sfB_k$.

    We shall construct a diagram as below. 
    \begin{equation}
        \label{eqn: bad points diagram}
        \begin{tikzcd}
	& {T^\circ} & T \\
	{\wt{\sfM}^\circ_k \times_{\sfM} S} & {T^\circ_0} & {T_0} & {\overline{T}_0} \\
	 & {S^\circ} & {\overline{S}} \\
	& {\wt{\sfM}^\circ_k} & {\overline{\sfM}_k} & {\shS_k} & {\overline{\shS}_k}
	\arrow[hook, from=1-2, to=1-3]
	\arrow[from=1-2, to=2-2]
	\arrow["{\textrm{generically \'etale}}", from=1-3, to=2-3]
	\arrow[from=2-1, to=3-2]
	\arrow["{\textrm{conn. comp.}}"', hook', from=2-2, to=2-1]
	\arrow[hook, from=2-2, to=2-3]
	\arrow[from=2-2, to=3-2]
	\arrow[from=2-3, to=2-4]
	\arrow[from=2-3, to=3-3]
	\arrow[from=2-3, to=4-4]
	\arrow[from=2-4, to=4-5]
	\arrow[hook, from=3-2, to=3-3]
	\arrow[from=3-2, to=4-2]
	\arrow[from=3-3, to=4-3]
	\arrow[from=4-2, to=4-3]
	\arrow["\rho"', curve={height=18pt}, from=4-2, to=4-4]
	\arrow[hook, from=4-4, to=4-5]
\end{tikzcd}
    \end{equation}
    Apply \cref{const: strongly dominate} to $\overline{\sfM}_b$ and assume that the positive integers $d_1, \cdots, d_{r - 1}$'s there are chosen to be sufficiently big. 
    Let $\overline{S} \subseteq \overline{\sfM}_b$ be a smooth subsurface which is obtained by intersecting the first $r - 1$ hypersurfaces $H_1, \cdots, H_{r - 1}$, contains $s$, has generically reduced intersection with every irreducible component of $\sfD^+_b$ and $\sfH_b$, and intersects $\sfD^-_b$ is a discrete set of (possibily non-reduced) points.  
    Set $S \colonequals \overline{S} \cap \sfM_b$.
    Next, take $T^\circ_0$ to be a connected component of $\wt{\sfM}^\circ \times_{\sfM} S$.
    As in \cref{eqn: extend to Tbar} of the preceeding proof, we may apply \cite[\href{https://stacks.math.columbia.edu/tag/0C5H}{Tag 0C5H}]{stacks-project} to find a smooth proper compactification $\overline{T}_0$ of $T_0^\circ$ such that the morphisms $S^\circ \leftarrow T^\circ_0 \to \shS_k$ extend to $\overline{S} \leftarrow \overline{T}_0 \to \overline{\shS}_k$. 
    Let $T$ be a connected generically \'etale cover of $T_0 := \overline{T}_0 \times_{\overline{S}} S$ such that $\sX|_T$ is Zariski-locally resolvable, which exists by \cref{lem: etale locally resolvable}.

    To prove the result, it then suffices to show that the resulting extended period morphism $T \to \shS_k$, which we denote by $\tau$, satisfies hypothesis (ii) of \cref{lem: bad points}.
    To check this, we may replace $k$ by its algebraic closure. 
    Let $\wt{s}'$ be the generic point of an integral $k$-curve $\wt{D}$ in $\overline{T}$. 
    Set $s' \colonequals \pi(\wt{s}')$ and let $D$ be the Zariski closure of $s'$, or equivalently the scheme-theoretic image of $\wt{D}$. 
    There are two possibilities: Either $s' \in \sfD^+_k$ or $s' \in \sfM^\circ_k$. 
    Note that in the former case $D$ is an irreducible component of $\overline{S} \cap \sfD^+_b$. 
    In either case, let $t \in D$ be a general point, so that for some lift $\wt{t} \in \wt{D}$ of $t$, $\sX|_{\wt{D}}$ has a simultaneous resolution near $\wt{t}$ whose total space is a projective scheme (\cref{lem: spreading out}).

    Now in \cref{const: strongly dominate}, we pick $H_{r}$ to obtain an irreducible addmissible curve $\overline{C} \subseteq \overline{S}$ which passes through $t$, intersects $\sfD_b \cup \sfH_b$ transversely along $\sfD^+_b \cup \sfH^+_b$, and generically lies in the maximal open subset of $S$ over which $T$ is \'etale. 
    Set $C := \overline{C} \cap S$ and let $\wt{C}$ be the connected component of $T \times_S C$ that contains $\wt{t}$. 
    Then $\wt{C}$ is generically \'etale over $C$. 
    By applying \cref{prop: good points} to the normalization of $\wt{C}$, we deduce that $\tau$ is motivically aligned at $\wt{t}$. 
    As $\wt{s}'$ specializes to $\wt{t}$ by assumption, 
    \cref{lem: def invariance of MA}\ref{item: def invariance of MA converse} tells us that $\tau$ is also motivically aligned at $\wt{s}'$. 
\end{proof}

\subsection{Specializations to \cref{thm: BSD} and \cref{thm: general type}}
\label{sec: thms 1.1 and 1.2}
In this subsection, we illustrate how the general result in \Cref{thm: general theorem} can be applied to concrete situations, and prove \cref{thm: BSD} and \cref{thm: general type}.
We will give the explicit constructions of the moduli spaces in those two situations, and check that they satisfy the assumptions of \Cref{set-up: surfaces}.

For readers' convenience, we note a simple observation:
\begin{lemma}
\label{lem: power map fin fibers}
    Let $L$ be a line bundle on an integral scheme $S$. 
    Then for any $r$, the map $\H^0(C, L) \to \H^0(C, L^r)$ given by raising a section to its $r$th power is a map of finite fibers of order at most $r$. 
\end{lemma}
\begin{proof}
    Choose a trivialization $L|_U \simeq \sO_S(U)$ for some dense open $U \subseteq S$. Suppose that $s_1, s_2 \in \H^0(C, L)$ are global sections and let $f_1, f_2 \in \sO_S(U)$ be the corresponding meromorphic functions via the trivialization. 
    Then by the assumption that $S$ is integral, $s_1^r = s_2^r \in \H^0(C, L^r)$ if and only if $f_1^r = f_2^r$. When $f_1$ is fixed, solving the equation $f_1^r = f_2^r$ in the faction field of $S$ leads to at most $r$ distinct roots. 
\end{proof}

Now we give the construction of the universal Weierstrass elliptic surface (cf. \cite{Kas}, \cite[\S7.1]{HYZ}).
\begin{construction}[Universal Weierstrass elliptic surface]
\label{const: Weierstrass}
Note that to construct the Universal Weierstrass elliptic surface over curves of genus one, we need to parametrize both the Weierstrass model of the base curve and the Weierstrass model of the fibration.
So the Weierstrass equation will appear twice.

    To start, we define $\Delta(a, b)$ to be the polynomial of two variables $4a^3 - 27b^2$.
    We let $\sfB$ be the complement of the vanishing locus of $\Delta(b_4, b_6)$ in the two dimensional affine space $\Spec(\IZ_{(p)}[b_4, b_6])$. 
    Let $\sE \subseteq \sfB \times_{\IZ_{(p)}} \Proj (\IZ_{(p)}[x, y, z])$ be the family of elliptic curves given by the reduced Weierstrass equation $y^2 z=x^3 - b_4 x z^2 - b_6 z^3$. 
    Then the map $\varpi\colon \sE\to \sfB$ is the universal Weierstrass model of elliptic curves.
    
    To continue, we let $\underline{\infty}$ be the infinity section given by $[0 : 1 : 0]$ of the family $\sE\to \sfB$, and let $\sL \colonequals \sO_\sE(\underline{\infty})$ be the corresponding relative line bundle on $\sE$.  
    For a given positive integer $r$, we let $\sV_r$ be the vector bundle over $\sfB$ defined by $\varpi_* \sL^r$. 
    Then we define $\overline{\sfM}$ to be the projective bundle $\IP(\sV_4 \oplus \sV_6 \oplus \sO)$ over $\sfB$. 
    
    We let $\IA(\sV_4 \oplus \sV_6)$ denote the relative affine space defined by the vector bundle $\sV_4 \oplus \sV_6$ over $\sfB$.
    So $\IP(\sV_4 \oplus \sV_6 \oplus \sO)$ is the relative compactification of $\IA(\sV_4 \oplus \sV_6)$ by adding in the infinity section. 
    Moreover, from the description, a point on $\IA(\sV_4, \sV_6)$ is given by a triple $(b, a_4, a_6)$, where $b$ is a point on $B$ and $(a_4, a_6) \in \H^0(\sfB, \sL_b^4) \times \H^0(\sfB, \sL_b^6)$, and we have $\sL_b = \sO_{\sE_b}(\underline{\infty}_b)$.
    We also define $\{ \Delta = 0 \} \subseteq \IA(\sV_4 \oplus \sV_6)$ to be the reduced closed subscheme parametrizing the points $(b, a_4, a_6)$ such that $\Delta(a_4, a_6) = 0 \in \H^0(\sL_b^{12})$.
    Then we let $\sfH \subseteq \IP(\sV_4 \oplus \sV_6 \oplus \sO)$ be the union of the infinity section and $\{ \Delta = 0 \}$. 
    It is not hard to see that for each $b \in \sfB$, $\dim \sfH_b \le \dim \H^0(\sL_b^4)$, because for each fixed $a_4 \in \H^0(\sL_b^4)$, there are at most $3$ distinct $a_6$ to make $\Delta$ vanish identically (\cref{lem: power map fin fibers}).

    Now the Weierstrass moduli space for Weierstrass fibrations is defined as the open subscheme $\sfM \colonequals \overline{\sfM} \smallsetminus \sfH$. Let $\IA(\sV_4 \oplus \sV_6)^*$ denote the complement of the zero section in $\IA(\sV_4 \oplus \sV_6)$.
    We let $\sX \subseteq \IA(\sV_4 \oplus \sV_6)^* \times \Proj (\IZ_{(p)}[X, Y, Z])$ be the closed subscheme defined by the equation $Y^2 Z=X^3 - a_4 X Z^2 - a_6 Z^3$ for $a_r \in \sV_r$. 
    Then the morphism $f:\sX\to \sfM$ is the universal Weierstrass fibration.
\end{construction}

\begin{remark}[Weierstrass fibration has RDPs]
\label{rmk: RDP on Weierstrass}
    The singularities on the Weierstrass normal form of an elliptic surface are always rational double points (RDPs). 
    In characteristic $0$, this was proved by Kas \cite[Thm.~1]{Kas} and the same proof works for characteristic $p \ge 5$ verbatim. 
    Alternatively, the notes of Conrad \cite{ConradElliptic} provides an argument for this fact which works for $p = 2, 3$ as well: The Weierstrass normal form (which is called Weierstrass model in \textit{loc. cit.}) is Gorenstein (\cite[\S 4]{ConradElliptic}) and has only rational singularities (\cite[\S 8]{ConradElliptic}). 
    Thus we may conclude by noticing that for a surface, an RDP is the same as a Gorenstein rational singularity.  
\end{remark}

We shall prove a slightly more general form of \cref{thm: BSD}.\footnote{Indeed, by taking the minimal resolution, the elliptic curve $\sE$ therein can be extended to a (unique) minimal smooth elliptic surface $X$ which satisfies the assumption of \Cref{thm: BSD_more_general}.
The BSD conjecture of $\sE$ is then equivalent to the Tate conjecture of $X$, as explained in the introduction.} 
\begin{theorem}
	\label{thm: BSD_more_general}
    Let $k$ be a finitely generated field over $\IF_p$ for $p \ge 5$. 
    Let $X$ be a minimal smooth proper surface with an elliptic fibration $\pi : X \to C$ over a genus $1$ $k$-curve $C$ such that $\deg (R^1 \pi_* \sO_X)^\vee = 1$, and assume $\pi$ admits a section. 
    Assume either that $X$ is not supersingular or has the expected Hodge numbers.
    Then $X$ satisfies the Tate conjecture. 
\end{theorem}
Recall that $X$ has the expected Hodge numbers (\cref{thm: Hodge diamond of elliptic surface}) whenever its $j$-invariant map is separable---which may fail only in the three cases listed in \cref{prop: inspearable j}. In particular, this condition is always true when $p \ge 11$. We also remind the reader that if $\sE$ is isotrivial (i.e., $j$-invariant map is constant), then $X$ admits a dominant rational map from a product of curves, so that the Tate conjecture for $X$ is well known, so we only consider the non-isotirival case.

\begin{proof}
    Up to replacing $k$ by a finite extension, there exists a point $s \in \sfM(k)$ such that $X$ is the minimal resolution $\breve{\sX}_s$ for the universal family of Weierstrass fibration $\sX \to \sfM$ in \cref{const: Weierstrass} and a point $s\in \sfM$. 

    Let us check that the family $f:\sX \to \sfM$ satisfies the hypotheses of \cref{set-up: surfaces}. 
    Let $\eta$ be the generic point of $\sfM$.
    We shall take the sublattice of polarization $\Lambda$ to be the subgroup of $\NS(\sX_\eta)_{\mathrm{tf}}$ generated by the classes of the zero section $\mathbf{e}$ and a general fiber in the elliptic fibration $\mathbf{f}$. 
    Then we have the intersection products that $\mathbf{f}^2 = 0$ and $\mathbf{e}. \mathbf{f} = 1$. 
    The fact that $\mathbf{e}^2 = -1$ follows from the assumption that $\deg (R^1 \pi_* \sO_X)^\vee  = 1$ (\cite[(II.3.6)]{Miranda}). 
    Since $\Lambda$ is self-dual and contains an ample line bundle,
    $\Lambda$ satisfies condition (a). 
    The conditions (b) and (c) are from the construction of the scheme $\sfM$. 
    The statement about Kodaira--Spencer map in (d) can be deduced from the main theorem of \cite{EGW24} (cf. \cite[Lem.~7.4.3]{HYZ}). 
    Furthermore, condition (e) follows from \cite[Lem.~7.3.3]{HYZ}. 
    Finally, by \cref{rmk: RDP on Weierstrass} above, we have (f). Now the result directly follows from \cref{thm: general theorem}. 
\end{proof}

\begin{example}
\label{exp: not good prime}
Let $C$ be a genus $1$ curve defined over a finitely generated extension $k$ of $\IF_p$ with a rational point $e \in E(k)$. 
Set $L := \sO_C(e)$. 
By Riemann-Roch theorem, we easily deduces that there exist global sections $a_4, a_6$ of $L^4, L^6$ respectively such that $\val_e(a_4) = 4$ and $\val_e(a_6) = 4$. 
Let $Y \to C$ be the corresponding Weierstrass normal form. Then by Kodaira's table \cite[p.41, IV.3.1]{Miranda}, we find that the fiber $Y_e$ is irreducible of type $\mathrm{IV}^*$, whose singularity is an $E_6$-singularity on the surface $Y$.
When $p = 5$, then $p$ divides the order of the Weyl group corresponding to $E_6$, which is $2^7 3^4 5$. 
This means that $p = 5$ is not ``sufficiently good'' for an $E_6$-singularity in the sense of \cite{Kou23}. 
Nonetheless, our \cref{thm: BSD_more_general} is applicable to the minimal resolution $X = \breve{Y}$, which shows that our method is insensitive to this type of characteristic $p$ subtleties that may influence local monodromy along a DVR (see for example Thm.~1.1 of \textit{loc. cit.} which would be false without the $p$ being sufficiently good assumption, cf. Shepherd-Barron's \cite[p.~598, Rmk(2)]{SBSing}). 
\end{example}

\begin{remark}[Remaining cases for $p=5,7$]
\label{rmk: purity and supersingular}
    If the purity statement for the Newton stratification is true for $\sfM$ (see \cref{rmk: purity of Newton}), then we can prove \Cref{thm: BSD_more_general} for the remaining three cases in \cref{prop: inspearable j} as well. 
    Indeed, suppose that $X$ is such a surface and is supersingular. By \cref{cor: bound inseparable} and the fact that the supersingular locus should have codimension at most $5$ by purity, we can deform $X$ to another surface $X'$ within the supersingular locus, so that $X'$ does not belong to the three cases of \cref{prop: inspearable j}. Then we deform all line bundles on $X'$ to the generic fiber of this family by Morrow's generalization of Artin's theorem \cite[Thm.~3.11]{Morrow} and then specialize to $X$. 
\end{remark}

Next we consider the surfaces of general type as in \cref{thm: general type}.
We first recall the following fact from \cite{Catanese, Todorov}: The canonical model of a minimal smooth projective surface $X$ over $\IC$ with properties $p_g(X) = 1$ and $K_X^2 = 1$ 
can be written as the intersection of two degree $6$ hypersurfaces in the weighted projective space $\IP(1, 2, 2, 3, 3)$. 
As argued in \cite[Thm.~8.0.1]{HYZ}, over an algebraically closed field $\kappa$ of characteristic $p \ge 5$, the same is true if we impose that the irregularity $q = h^{0, 1}$ is $0$ (which is automatic over $\IC$) as an additional requirement.

Now we give the construction of the parametrizing space and the universal family for the aforementioned surfaces of general type.

\begin{construction}
\label{const: general type}
    Fix a prime number $p \ge 5$.
    We denote the weighted projective space $\IP(1, 2, 2, 3, 3)$ over $\IZ_{(p)}$ by $\IP$. 
    Let $|\sO_{\IP}(6)|$ be the projective space associated to the linear system that parametrizes degree $6$ hypersurfaces. 
    For $i = 0, 1$, let $\mathsf{P}_i$ be a copy of $|\sO_{\IP}(6)|$, let $\sH_i \subseteq \mathsf{P}_i \times \IP$ be the universal hypersurface, let $\sfD_i \subseteq \mathsf{P}_i$ be the discriminant locus of $\sH_i \to \mathsf{P}_i$, and let $\mathsf{P}^\circ_i$ be the open complement $\mathsf{P}_i \smallsetminus \sfD_i$. 
    
    Now we let $\overline{\sfM} \colonequals \mathsf{P}_0 \times \mathsf{P}_1$, and let $\overline{\sX}$ be the closed subscheme of $\IP \times \overline{\sfM}$ given by the incidence $\{ (x, (\alpha, \beta)) : x \in \sH_{0, \alpha} \cap \sH_{1, \beta} \}$. 
    Let $\sfM$ be the maximal open subscheme of $\mathsf{P}_1 \times \mathsf{P}_2$ over which $\overline{\sX}$ is flat and the fibers has at worst RDP singularities (\cref{prop: def RDP}), and let $\sfH$ be the reduced complement $\overline{\sfM}\smallsetminus \sfM$.
    Then we obtain the induced family $f: (\sX\colonequals \overline{\sX}\otimes_{\overline{\sfM}} \sfM)  \to \sfM$.
    The discriminant locus $\sfD \subseteq \sfM$ is defined as the reduced subscheme consisting of the points $t\in \sfM$ such that $\sX_t$ is singular, and we let $\sfM^\circ \colonequals \sfM \smallsetminus \sfD$. 
\end{construction}

\noindent \textit{Proof of \cref{thm: general type}.} 
By the paragraphs before \Cref{const: general type} and by \Cref{thm: general theorem}, it suffices to prove that \cref{const: general type} satisfies the hypothesis of \cref{set-up: surfaces}. Note that unlike \cite[Thm.~C]{HYZ}, here we additionally imposed the assumption $h^{1, 0} = 0$ to make sure that the surfaces to which we apply \Cref{thm: general theorem} have the expected Hodge diamond. 

For condition (a), we take $\Lambda \subseteq \NS(\sX_\eta)_{\mathrm{tf}}$ to the sublattice spanned by the canonical divisor, whose self-intersection number is $1$ by assumption. 
So the dual lattice $\Lambda^\vee$ is equal to $\Lambda$ itself, and $\Lambda$ contains an ample divisor.
For condition (b), we simply define $\sfB$ to be $\Spec(\IZ_{(p)})$ itself. 
For (d), the statement about the Kodaira--Spencer map holds because the generic Torelli theorem holds for these surfaces, which have a moduli space of dimension $18$ (see \cite[p.1]{Catanese}). 
In addition, condition (f) holds by construction. 

We are left with (c) and (e), which we shall check simultaneously. 
Let $\kappa = \overline{\IQ}$ or $\overline{\IF}_p$. 
Let $D$ be any irreducible component of $\sfD_{\kappa}$ or $\sfH_\kappa$ that has codimension $1$ in $\overline{\sfM}_\kappa$. 
We first argue that the image of $D$ under the projection $\overline{\sfM} \to \mathsf{P}_{1}$ to the second factor cannot lie in $\sfD_{1, \kappa}$. 
Indeed, suppose that this does happen. 
Then $D$ must be surjective onto $\mathsf{P}_{0, \kappa} \times \sfD_{1, \kappa}$, because using standard Bertini-type arguments, one easily sees that $\sfD_{1, \kappa}$ is irreducible. 
We then claim that there is a $\kappa$-point $(\alpha, \beta) \in \mathsf{P}_{0, \kappa} \times \sfD_{1, \kappa}$ such that $\sH_{0, \alpha} \cap \sH_{1, \beta}$ is smooth, which would contradict the assumption that $D$ is within the discriminant locus of the family $\sX\to \sfM$.
In fact, we can let $\beta$ be a general point on $\sfD_{1, \kappa}$. 
Then $\sH_{1, \beta}$ only has isolated singularities. 
As $\sO_{\IP}(6)$ is very ample, so is its restriction $\sO_{\sH_{1, \beta}}(6)$ to $\sH_{1, \beta}$. 
In particular, the line bundle $\sO_{\sH_{1, \beta}}(6)$ has no base locus and its associated morphism, which is a closed embedding of $\sH_{1, \beta}$, induces separably finitely generated field extensions (in fact isomorphisms) at every point. 
Therefore, by the second theorem of Bertini (cf. \cite[Cor.~2]{CGM86}), a general section of $\sO_{\sfH_{0, \alpha}}(6)$ is smooth. 
Namely, for a general point $\alpha \in \mathsf{P}_{0, \kappa}$, the intersection $\sH_{0, \alpha} \cap \sH_{1, \beta}$ is smooth. 

Now, to prove (c) and (e) of \Cref{set-up: surfaces}, it suffices to look at the intersection of $D$ with the open subscheme $\mathsf{P}_0 \times \mathsf{P}^\circ_1 \subseteq \overline{\sfM}$. 
Suppose now that $\beta$ is a $\kappa$-point of $\mathsf{P}^\circ_1$ such that $\sH_{1, \beta}$ is smooth. 
Then as we argued in \cite[Prop.~8.0.3]{HYZ}, the embedding induced by $\sO_{\sH_{1, \beta}}(6)$ is a Lefschetz embedding. 
This has two consequences: 
\begin{enumerate}[label=\upshape{(\roman*)}]
    \item For a general line $L$ in the projective space $\mathsf{P}_0$, the singular fibers of $\overline{\sX}|_L$ only have ordinary double points as singularities. 
    \item The support of $\sfD_{\beta}$, i.e., the point set $\{ \alpha \in \mathsf{P}_0 \mid \sH_{0, \alpha} \cap \sH_{1, \beta} \textit{ is singular} \}$, is irreducible. 
\end{enumerate}
Note that (i) implies that $\overline{\sX}|_L = \sX|_L$, which further implies that $D$ cannot be an irreducible component of $\sfH_\kappa$. 
As a consequence, the complement $\sfH_\kappa$ has no irreducible component of codimension $1$ in $\overline{\sfM}$, and condition (c) is trivally true. 
Moreover, by Prop.~8.0.5 \textit{ibid}, the fiber of $\sfD$ over $\beta$ is reduced. 
This implies that $\sfD_\kappa$ is generically reduced. Since $\mathsf{P}^\circ_1$ is irreducible, (ii) implies that $\sfD_\kappa$ is in fact irreducible. 
As this is ture for both $\kappa = \overline{\IQ}$ and $\overline{\IF}_p$, we know that $\sfD_{\IF_p}$ must lie in the Zariski closure of $\sfD_\IQ$. 
Thus by combining the above with the generic reducedness of $\sfD_\kappa$, we obtain condition (e) of \Cref{set-up: surfaces}. \qed\\\\
\noindent \textbf{Acknowledgements} We thank Jason Kountouridis, Zhiyuan Li, Gebhard Martin, Teppei Takamatsu, Botong Wang, Jakub Witaszek, and Ziquan Zhuang for helpful discussions on birational geometry. A special thank goes to Gebhard Martin for his patience in helping the authors understand singularities. During the preparation of the paper, the first author was supported by the University of Chicago, and the second author benefited from the AMS travel fund and CUHK's start-up fund for new faculties.

\printbibliography

\vspace{2em}

\noindent \textbf{Haoyang Guo} {\footnotesize Department of Mathematics, The University of Chicago, Chicago, Illinois, USA.  \,\,\,  Email: \url{ghy@uchicago.edu}}\\\\
\noindent \textbf{Ziquan Yang} {\footnotesize The Institute of Mathematical Sciences and Department of Mathematics, The Chinese University of Hong Kong, Shatin, N.T., Hong Kong.  \,\,\,  Email: \url{zqyang@cuhk.edu.hk}}

\end{document}